\documentclass[a4paper,12pt]{article}
\makeatletter
\newcommand*{\rom}[1]{\expandafter\@slowromancap\romannumeral #1@}
\makeatother

\usepackage{lipsum,relsize,graphicx,amsmath,amsthm,amssymb,mathtools,stmaryrd }
\usepackage[margin=1in]{geometry}
\usepackage{hyperref}
\DeclarePairedDelimiter\ceil{\lceil}{\rceil}

\newcounter{tmp}

\DeclareSymbolFont{largesymbols}{OMX}{yhex}{m}{n}
\DeclareMathAccent{\wideparen}{\mathord}{largesymbols}{"F3}
\newcommand{\w}[1]{\wideparen{#1}}

\newtheorem{theorem}{Theorem}[section]
\newtheorem{lemma}[theorem]{Lemma}
\newtheorem{corollary}[theorem]{Corollary}
\newtheorem{definition}{Definition}[section]
\newtheorem{proposition}[theorem]{Proposition}

\DeclareMathOperator{\ad}{ad}
\DeclareMathOperator{\im}{im}
\DeclareMathOperator{\Ann}{Ann}

\begin{document}

\title{Affinoid Dixmier Modules and the deformed Dixmier-Moeglin equivalence}

\author{Adam Jones}

\date{\today}

\maketitle

\begin{abstract}

\noindent The affinoid envelope $\widehat{U(\mathcal{L})}_K$ of a free, finitely generated $\mathbb{Z}_p$-Lie algebra $\mathcal{L}$, has proven to be useful within the representation theory of compact $p$-adic Lie groups. Our aim is to further understand the algebraic structure of $\widehat{U(\mathcal{L})}_K$, and to this end, we will define a Dixmier module over $\widehat{U(\mathcal{L})}_K$, and prove that this object is generally irreducible in case where $\mathcal{L}$ is nilpotent. Ultimately, we will prove that all primitive ideals in the affinoid envelope can be described in terms of the annihilators of Dixmier modules, and using this, we aim towards proving that these algebras satisfy a version of the classical Dixmier-Moeglin equivalence.

\end{abstract}

\tableofcontents

\section{Motivation}

\noindent Throughout, fix $p$ a prime, $K\backslash\mathbb{Q}_p$ a finite extension, $\mathcal{O}$ the valuation ring of $K$, $\pi\in \mathcal{O}$ a uniformiser.

\subsection{Classical Motivation -- The Dixmier-Moeglin equivalence}

This research is partially inspired by a problem in classical non-commutative algebra. Let $k$ be any field of characteristic 0, and let $R$ be a Noetherian $k$-algebra. We say that a prime ideal $P$ of $R$ is:

\begin{itemize}

\item \emph{Primitive} if $P=Ann_RM=\{r\in R:rM=0\}$ for some irreducible $R$-module $M$.

\item \emph{Weakly rational} if $Z(R/P)$ is an algebraic field extension of $k$.

\item \emph{Rational} if $Z(Q(R/P))$ is an algebraic field extension of $k$, where $Q(R/P)$ is the Goldie ring of quotients of $R/P$, in the sense of \cite[Theorem 2.3.6]{McConnell}.

\item \emph{Locally closed} if $P\neq \cap\{Q\trianglelefteq R: Q$ prime, $P\subsetneq Q\}$, i.e. if $\{P\}$ is a locally closed subset of Spec $R$ with respect to the Zariski topology.

\end{itemize}

\noindent It is not difficult to see that if $P$ is rational then it is weakly rational.

\begin{definition}\label{Dix-Moeg-eq}

Let $R$ be a Noetherian $k$-algebra. We say that $R$ satisfied the \emph{Dixmier-Moeglin equivalence} if for all prime ideals $P$ of $R$, we have:

\begin{center}
$P$ is primitive $\iff$ $P$ is rational $\iff$ $P$ is locally closed.
\end{center}

\end{definition}

\noindent Note that if $R$ is commutative and $k$ is algebraically closed, then this condition is essentially a strong version of Hilbert's Nullstellensatz, i.e. that all prime ideals $P$ of $R$ arise as an intersection of maximal ideals $\mathfrak{m}$ of $R$ of codimension 1.\\

\noindent This condition is known to be satisfied for numerous examples of Noetherian $k$-algebras, most notably the enveloping algebra $U(\mathfrak{h})$ for a finite dimensional $k$-Lie algebra $\mathfrak{h}$ as proved independently by Dixmier and Moeglin in \cite{Dixmier} and \cite{Moeglin}. An ongoing project is to classify algebras satisfying the equivalence, and in \cite{Dix-Moeg}, a detailed analysis of what is known to date on this subject is given. In this paper, however, we will not be concerned with the classical picture, but with a $p$-adic analogue.\\

\noindent Specifically, let $K$ be our $p$-adic field, and let $\widehat{R}$ be a Noetherian, Banach $K$-algebra. Unfortunately, to prove the full Dixmier-Moeglin equivalence for such algebras may prove to be impractical, since the techniques classically used typically involve discrete generating sets, which are absent here. However, there is a weaker statement that might be more approachable in this setting.\\

\noindent For each $n\in\mathbb{N}$, in Definition \ref{deformation} below we will define the \emph{level $n$ deformation} $\widehat{R}_n$ of $\widehat{R}$, which is a dense Banach subalgebra of $\widehat{R}$, and using this we make the following definition:

\begin{definition}\label{deformed-Dix-moeg}

We say that the Noetherian, Banach $K$-algebra $\widehat{R}$ satisfies the \emph{deformed Dixmier-Moeglin equivalence} if for all prime ideals $P$ of $R$, there exists $N\in\mathbb{N}$ such that for all $n\geq N$:

\begin{center}
$P\cap\widehat{R}_n$ is primitive $\iff$ $P \cap\widehat{R}_n$ is weakly rational $\iff$ $P\cap \widehat{R}_n$ is locally closed.
\end{center}

\end{definition}

\noindent\textbf{Key example:} If we take $R=U(\mathfrak{g})$ and let $\mathcal{R}:=U(\mathcal{L})$ for some $\mathcal{O}$-Lie lattice $\mathcal{L}$ in $\mathfrak{g}$, then the $\pi$-adic completion $\widehat{R}=\widehat{U(\mathcal{L})}_K$ of $R$ with respect to $\mathcal{R}$ is called the \emph{affinoid enveloping algebra} (or \emph{affinoid envelope}) of $\mathcal{L}$, and $\widehat{R}_n$ is just $\widehat{U(\pi^n\mathcal{L})}_K$. 

We believe that the affinoid enveloping algebra $\widehat{U(\mathcal{L})}_K$ should satisfy the deformed Dixmier-Moeglin equivalence, but to date, this is only known in the case where $\mathfrak{g}$ is nilpotent and contains an abelian ideal of codimension 1, as shown in \cite[Corollary 1.6.2]{Lewis}.

\subsection{Second motivation -- Iwasawa algebras}

The other main motivation for this research lies within representation theory of compact $p$-adic Lie groups.\\

\noindent Fix $G$ a uniform pro-$p$ group in the sense of \cite[Definition 4.1]{DDMS}. When studying the $p$-adic representation theory of $G$, there are several avenues we can consider. If we were interested in general, abstract $K$-linear representations of $G$, naturally we would study modules over the standard group algebra $K[G]$, but since $G$ is not discrete, this object is of little practical use. Instead, define the \emph{rational Iwasawa algebra}: 

\begin{center}
$KG:=\underset{U\trianglelefteq_o G}{\varprojlim}{\mathcal{O}[G/U]}\otimes_{\mathcal{O}}K$,
\end{center} 

\noindent this is a Noetherian, topological $K$-algebra whose module structure completely describes the \emph{continuous} representations of $G$, a class which includes all finite dimensional, smooth and locally analytic representations.\\

\noindent Let $\mathcal{L}:=\frac{1}{p}\log(G)$ be the $\mathbb{Z}_p$-Lie algebra of $G$, as defined in \cite[Section 10]{annals}, and let $\widehat{U(\mathcal{L})}_K$ be the \emph{affinoid envelope} of $\mathcal{L}$ with coefficients in $K$, as defined above.\\

\noindent Using \cite[Theorem 10.4]{annals}, we see that there exists a dense embedding of topological $K$-algebras $KG\to\widehat{U(\mathcal{L})}_K$, i.e. $\widehat{U(\mathcal{L})}_K$ is a Banach completion of $KG$, and thus we see that representations of the affinoid enveloping algebra $\widehat{U(\mathcal{L})}_K$ naturally have the structure of $KG$-modules, which allows us to explore the representation theory of $G$ via the representation theory of $\mathcal{L}$.\\

\noindent In this paper, we will study the algebraic structure of the affinoid envelope, and in a forthcoming work \cite{new}, we will show how we can use our results to deduce information about primitive ideals in $KG$.

\subsection{Main results}

The key aim of this paper is to describe the primitive ideal structure of the affinoid enveloping algebra $\widehat{U(\mathcal{L})}_K$, focusing on the case where $\mathcal{L}$ is nilpotent. Specifically, we aim to prove that under suitable conditions this algebra satisfies the deformed Dixmier-Moeglin equivalence.\\

\noindent To this end, we will follow the representation theoretic approach outlined by Jaques Dixmier in \cite{Dixmier} when studying the classical enveloping algebra $U(\mathfrak{g})$. This approach was to define a class of irreducible induced representations of $U(\mathfrak{g})$ whose annihilator ideals completely describe the primitive ideals in $U(\mathfrak{g})$. In section 2, we will adapt this approach to the affinoid envelope. 

Specifically, to each linear form $\lambda\in$ Hom$_{\mathbb{Z}_p}(\mathcal{L},\mathcal{O})$, we will associate an induced module $\widehat{D(\lambda)}$ over $\widehat{U(\mathcal{L})}_K$, which we call a \emph{Dixmier module}.\\

\noindent We will prove in section 3 that Dixmier modules are generally irreducible, and thus the ideal $I(\lambda):=$ Ann$_{\widehat{U(\mathcal{L})}_K}\widehat{D(\lambda)}$ is primitive in $\widehat{U(\mathcal{L})}_K$. We call ideals of this form \emph{Dixmier annihilators}. Our main result, which we will prove in section 6, describes all weakly rational ideals of $\widehat{U(\mathcal{L})}_K$ in terms of Dixmier annihilators: 

\begingroup
\setcounter{tmp}{\value{theorem}}% store current value of theorem counter
\setcounter{theorem}{0} %assign desired value to theorem counter
\renewcommand\thetheorem{\Alph{theorem}}
\begin{theorem}\label{A}

Let $\mathcal{L}$ be an $\mathcal{O}$-Lie lattice in a nilpotent Lie algebra $\mathfrak{g}$, and let $P$ be a weakly rational ideal in $\widehat{U(\mathcal{L})}_K$. Then there exists $N\in\mathbb{N}$, $\lambda\in$ Hom$_{\mathcal{O}}(\pi^N\mathcal{L},\mathcal{O}_{\overline{K}})$ such that for all $n\geq N$:

\begin{center}
$P\cap\widehat{U(\pi^n\mathcal{L})}_K=I(\lambda)=Ann_{\widehat{U(\pi^n\mathcal{L})}_K}\widehat{D(\lambda)}$
\end{center}

\end{theorem}

\endgroup

\noindent\textbf{Note:} This makes sense because if $\lambda$ takes values in the algebraic closure $\overline{K}$ of $K$, then it must take values in $F$ for some finite extension $F$ of $K$, and we can consider the action of $\widehat{U(\mathcal{L})}_K$ on $\widehat{D(\lambda)}$ via the embedding of $\widehat{U(\mathcal{L})}_K$ into $\widehat{U(\mathcal{L})}_F$.\\

\noindent This result aims towards a correspondence between orbits of linear forms in Hom$_{\mathbb{Z}_p}(\mathcal{L},\mathcal{O}_{\overline{K}})$ and primitive ideals in $\widehat{U(\mathcal{L})}_K$, in line with the classical result of Dixmier \cite[Theorem 6.2.4]{Dixmier}. In a subsequent paper \cite{new}, Theorem \ref{A} will become a key step in classifying primitive ideals in $KG$ for $G$ nilpotent.\\

\noindent Finally, in section 7, we will explore the properties of Dixmier modules and Dixmier annihilators further, and prove the deformed Dixmier-Moeglin equivalence in a specific case:

\begingroup
\setcounter{tmp}{\value{theorem}}% store current value of theorem counter
\setcounter{theorem}{1} %assign desired value to theorem counter
\renewcommand\thetheorem{\Alph{theorem}}
\begin{theorem}\label{B}

Let $\mathfrak{g}$ be a nilpotent $K$-Lie algebra such that $[\mathfrak{g},\mathfrak{g}]$ is abelian, and let $\mathcal{L}$ be an $\mathcal{O}$-Lie lattice in $\mathfrak{g}$. Then $\widehat{U(\mathcal{L})}_K$ satisfies the deformed Dixmier-Moeglin equivalence. In fact, for all weakly rational ideals $P$ of $\widehat{U(\mathcal{L})}_K$, $n\in\mathbb{N}$ sufficiently high, $\frac{\widehat{U(\pi^n\mathcal{L})}_K}{P\cap\widehat{U(\pi^n\mathcal{L})}_K}$ is a simple domain.

\end{theorem}

\endgroup

\noindent This result provides an insight into the deeper algebraic structure of the affinoid envelope, and we hope it will be of independent interest within the field of non-commutative algebra, and that it will advance the Dixmier-Moeglin project into the $p$-adic setting.\\

\noindent\textbf{Acknowledgments:} I am very grateful to Konstantin Ardakov and Ioan Stanciu for many helpful conversations and discussions. I would also like to thank EPSRC and the Heilbronn Institute for Mathematical Research for supporting and funding this research.

\section{Preliminaries}

In this section, we sometimes want to consider the classical picture. So throughout, take $k$ to be any field of characteristic 0.

\subsection{Affinoid algebras and Rigid Geometry}

First, we will recap some notions from rigid geometry. For a more detailed discussion, see \cite{Bosch}.\\

\noindent Recall from \cite[Definition 2.2.2]{Bosch} that if $R$ is a ring carrying a complete, separated filtration $w$, the \emph{Tate algebra} in $d$ variables $t_1,\cdots,t_d$ over $R$ is the algebra:

\begin{center}
$R\langle t_1,\cdots,t_d\rangle:=\{\underset{\alpha\in\mathbb{N}^d}{\sum}{\lambda_{\alpha}t_1^{\alpha_1}\cdots t_d^{\alpha_d}}:w(\lambda_{\alpha})\rightarrow\infty$ as $\alpha\rightarrow\infty\}$.
\end{center}

\noindent In other words, the Tate algebra is the ring of power series with coefficients in $R$ that converge on the unit disc $R_+^d$, we call these \emph{Tate power series}. This ring carries a separated filtration given by $w_{\inf}(\underset{\alpha\in\mathbb{N}^d}{\sum}{\lambda_{\alpha}t_1^{\alpha_1}\cdots t_d^{\alpha_d}}):=\inf\{w(\lambda_{\alpha}):\alpha\in\mathbb{N}^d\}$.\\

\noindent Normally, $R$ is assumed to be commutative, and in our case, we will usually take $R=K$, in which case the Tate algebra is Noetherian, and the filtration $w_{\inf}$ is \emph{Zariskian} in the sense of \cite[Ch. \rom{2} Definition 2.1.1]{LVO}. Recall from \cite[Definition 3.1.1]{Bosch} that we define an affinoid algebra to be any quotient of a Tate algebra over a complete, discretely valued field. Clearly any affinoid algebra $A$ will carry a complete, Zariskian filtration $w_A$ given by the quotient filtration with respect to $w_{\inf}$.\\

\noindent Affinoid algebras play a similar role in rigid geometry as commutative algebras play in standard algebraic geometry. Specifically, if $A$ is an affinoid algebra, we define Sp $A$ to be the space of maximal ideals of $A$. We call this an \emph{affinoid variety}, and we can realise $A$ as a ring of $\overline{K}$-valued functions on Sp $A$, where $a(\mathfrak{p}):=a+\mathfrak{p}$. This takes values in $\overline{K}$ since any maximal ideal in the Tate algebra has finite codimension by \cite[Corollary 2.2.12]{Bosch}.

We say that $A$ is the \emph{ring of analytic functions} on the affinoid variety Sp $A$. Note that for any ring homomorphism $\phi:A\to B$ between affinoid algebras induces a map $\phi^{\#}:$ Sp $B\to$ Sp $A,\mathfrak{q}\mapsto\phi^{-1}(\mathfrak{q})$, continuous with respect to the Zariski topology, and we call this a \emph{morphism of affinoid varieties}. Therefore, we can realise affinoid varieties as a category, equivalent to the category of affinoid algebras, via an equivalence where each variety is sent to its ring of analytic functions.\\

\noindent Affinoid varieties are useful in $p$-adic analysis, since they can indeed be realised as non-archimedean spaces. Recall that for each $\epsilon\in\mathbb{R}$, we define the $d$-dimensional polydisc of radius $\epsilon$ to be the space 

\begin{center}
$\mathbb{D}_{\epsilon}^d:=\{\underline{\alpha}\in\overline{K}^d:v_{\pi}(\alpha_i)\geq\epsilon$ for each $i\}$.
\end{center}

\noindent When $\epsilon=0$ we call this the \emph{unit disc}. We can consider this disc an affinoid space, isomorphic to Sp $K\langle u_1,\cdots,u_d\rangle$, and thus all discs are isomorphic, regardless of the radius.\\ 

\noindent Note that the Tate algebra $K\langle u_1,\cdots,u_d\rangle$ is precisely the set of power series converging on $\mathbb{D}_0^d$, so we can indeed realise the Tate algebra as the ring of analytic functions on the unit disc. Moreover, for each $n\in\mathbb{N}$, the subalgebra $K\langle \pi^nu_1,\cdots,\pi^nu_d\rangle$ is precisely those functions which converge on $\mathbb{D}_{-n}^d$.\\

\noindent Now, recall the following definition (\cite[Definition 3.3.9]{Bosch}):

\begin{definition}\label{affinoid-subdomain}

Let $X$ be an affinoid variety. A subset $U$ of $X$ is called an \emph{affinoid subdomain} if there exists an affinoid variety $Y$ with a morphism $\alpha:X\to Y$ such that $\alpha(X)\subseteq U$, and the following universal property is satisfied: If $\beta:Z\to X$ is any morphism of affinoid varieties such that $\beta(Z)\subseteq U$ then there exists a unique morphism $\gamma:Z\to Y$ such that $\beta=\alpha\circ\gamma$.

\end{definition}

\noindent\textbf{Note:} 1. In the above definition, it follows from \cite[Lemma 3.3.10]{Bosch} that the map $\alpha$ is an embedding of affinoid spaces, and its image is equal to $U$, so we may sometimes identify $Y$ with $U$.\\

\noindent 2. If $X$ is an affinoid variety, then $X$ carries a Grothendieck topology, in the sense of \cite[Definition 5.1.1]{Bosch}, where the admissible open subsets are the affinoid subdomains $U$ of $X$.\\

\noindent Roughly speaking, we define a \emph{rigid space} to be a space carrying a Grothendieck topology, that is locally equivalent to an affinoid variety. For example. we can consider the affine plane $\mathbb{A}_d(\overline{K})=\overline{K}^d$ to be a rigid variety, where its admissible open subsets consist of the unit discs $\mathbb{D}_{\epsilon}^d$ and their affinoid subdomains. Indeed, for any affine variety $X$, a similar procedure known as \emph{analytification} allows us to realise $X$ as a rigid variety.

\subsection{Lattices and Completions}

\noindent Recall the following definition \cite[Definition 2.7]{annals}:

\begin{definition}

Let $V$ be a $K$-vector space. 

\begin{itemize}

\item An \emph{$\mathcal{O}$-lattice} in $V$ is an $\mathcal{O}$-submodule $N$ of $V$ such that $\underset{n\in\mathbb{N}}{\bigcap}{\pi^nN}=0$ and $N\otimes_{\mathcal{O}}K=V$.

\item The \emph{completion} of $N$ is the $\mathcal{O}$-module $\widehat{N}:=\underset{n\in\mathbb{N}}{\varprojlim}\frac{N}{\pi^nN}$.

\item The \emph{completion} of $V$ with respect to $N$ is the $K$-vector space $\widehat{V}=\widehat{V}_N:=\widehat{N}\otimes_{\mathcal{O}}K$.

\end{itemize}

\end{definition}

\noindent Note that if $V$ is finite dimensional then $\widehat{V}=V$ for any choice of lattice. More generally, every $\mathcal{O}$-lattice $N$ in $V$ induces an exhaustive, separated filtration $w_N$ on $V$ given by 

\begin{center}
$w_N(u):=\sup\{n\in\mathbb{Z}: u\in\pi^nN\}$.
\end{center}

\noindent This is the \emph{$\pi$-adic filtration} associated to $N$, and the topological completion of $V$ with respect to $w_N$ is precisely the completion $\widehat{V}_N$.\\

\noindent Now, let $R$ be a Noetherian $K$-algebra, and let $\mathcal{R}$ be a Noetherian $\mathcal{O}$-lattice subaglebra of $R$. Then the completion $\widehat{R}$ of $R$ with respect to $\mathcal{R}$ is a Noetherian, Banach $K$-algebra.

\begin{definition}\label{deformation}

Let $V\subseteq R$ be a $K$-vector subspace, and suppose that the lattice $M:=V\cap\mathcal{R}$ generates $\mathcal{R}$ as an $\mathcal{O}$-algebra. Then $\mathcal{R}_n:=\mathcal{O}\langle \pi^n M\rangle\subseteq\mathcal{R}$ is an $\mathcal{O}$-subalgebra lattice in $R$, and we define the \emph{level $n$ deformation} of $\widehat{R}$ to be $\widehat{R}_n$, the completion of $R$ with respect to $\mathcal{R}_n$.

\end{definition}

\noindent\textbf{Example:} If $R$ is the polynomial ring $K[t_1,\cdots,t_r]$ and $\mathcal{R}=\mathcal{O}[t_1,\cdots,t_r]$, then the completion $\widehat{R}$ of $R$ with respect to $\mathcal{R}$ is the Tate algebra $K\langle t_1,\cdots,t_r\rangle$. Moreover, we can take the deformed Tate algebra $K\langle\pi^nt_1,\cdots,\pi^nt_r\rangle$ to be the level $n$ deformation of $\widehat{R}$.\\

\noindent Now, let $\mathfrak{g}$ be a finite dimensional $K$-Lie algebra, and let $\mathcal{L}$ be an $\mathcal{O}$-Lie lattice in $\mathfrak{g}$, i.e. $\mathcal{L}$ is an $\mathcal{O}$-lattice in $\mathfrak{g}$ and it is closed under the Lie bracket. Note that the enveloping algebra $U(\mathcal{L})$ is an $\mathcal{O}$-lattice in $U(\mathfrak{g})$.

\begin{definition}\label{aff-env}

Define the \emph{affinoid enveloping algebra} of $\mathcal{L}$ with coefficients in $\mathcal{O}$ to be $\widehat{U(\mathcal{L})}$, the completion of $U(\mathcal{L})$ with respect to its $\pi$-adic filtration.\\

\noindent Also, define the \emph{affinoid enveloping algebra} of $\mathcal{L}$ with coefficients in $K$ to be 

\begin{center}
$\widehat{U(\mathcal{L})}_K:=\widehat{U(\mathcal{L})}\otimes_{\mathcal{O}} K$.
\end{center} 

\noindent This is the completion of $U(\mathfrak{g})$ with respect to the $\pi$-adic filtration associated to $U(\mathcal{L})$.

\end{definition}

\noindent\textbf{Note:} We can take the level $n$ deformation of $\widehat{U(\mathcal{L})}_K$ to be $\widehat{U(\pi^n\mathcal{L})}_K$.\\

\noindent The following lemma is a straightforward consequence of the Poincarr\'{e}-Birkoff Witt theorem, see e.g. \cite[Proposition 2.5.1]{Lewis} for the proof.

\begin{lemma}\label{PBW}

If we let $\{x_1,\cdots,x_d\}$ be a $K$-basis for $\mathfrak{g}$ which forms an $\mathcal{O}$-basis for $\mathcal{L}$, then $\widehat{U(\mathcal{L})}$ is isomorphic as an $\mathcal{O}$-module to the Tate algebra $\mathcal{O}\langle x_1,\cdots,x_d\rangle$, and hence $\widehat{U(\mathcal{L})}_K$ is isomorphic to $K\langle x_1,\cdots,x_d\rangle$ as a $K$-vector space.

\end{lemma}

\vspace{0.2in}

\noindent Now, let $M$ be a $\mathfrak{g}$-representation, i.e. a $U(\mathfrak{g})$-module, and let $N$ be an $\mathcal{O}$-lattice in $M$ such that $\mathcal{L}\cdot N\subseteq N$, and suppose that $N$ is $\pi$-adically complete. Then it follows that $M$ has the structure of a $\widehat{U(\mathcal{L})}_K$-module. Unless otherwise stated, we will assume that all modules are left modules.

\begin{proposition}\label{fin-gen}

Let $M$ be a finitely generated $\widehat{U(\mathcal{L})}_K$-module. Then $M$ contains an $\mathcal{O}$-lattice $N$ such that $M$ is complete with respect to $N$ and $\mathcal{L}\cdot N\subseteq N$.

\end{proposition}

\begin{proof}

$M=\widehat{U(\mathcal{L})}_K m_1+\cdots+\widehat{U(\mathcal{L})}_K m_s$, so let $N:=\widehat{U(\mathcal{L})} m_1+\cdots+\widehat{U(\mathcal{L})} m_s$, clearly $N$ is an $\mathcal{O}$-lattice in $M$.\\

\noindent Note that since $\widehat{U(\mathcal{L})}$ is $\pi$-adically complete and gr $\widehat{U(\mathcal{L})}\cong \frac{\mathcal{O}}{\pi\mathcal{O}}[t,u_1,\cdots,u_d]$ is Noetherian, it follows from \cite[Ch.\rom{2}, Theorem 2.1.2]{LVO} that the $\pi$-adic filtration on $\widehat{U(\mathcal{L})}$ is Zariskian, and hence any left submodule of $\widehat{U(\mathcal{L})}^s$ is closed in $\widehat{U(\mathcal{L})}^s$.\\ 

\noindent Since $N\cong\frac{\widehat{U(\mathcal{L})}^s}{J}$ for some left submodule $J$ of $\widehat{U(\mathcal{L})}$, it follows that $N$ is $\pi$-adically complete.\end{proof}

\subsection{Polarisations and Dixmier modules}

\noindent Let $\mathfrak{h}$ be a finite dimensional Lie algebra over $k$. First recall the following definition \cite[1.12.8]{Dixmier}:

\begin{definition}\label{polarisation}

Given $\lambda\in\mathfrak{h}^*$, define a \emph{polarisation} of $\mathfrak{h}$ at $\lambda$ to be a solvable subalgebra $\mathfrak{b}$ of $\mathfrak{h}$ such that for all $u\in\mathfrak{h}$, $\lambda([u,\mathfrak{b}])=0$ if and only if $u\in\mathfrak{b}$.

\end{definition}

\noindent In particular, if $\mathfrak{b}$ is a polarisation of $\mathfrak{h}$ at $\lambda$, then $\lambda([\mathfrak{b},\mathfrak{b}])=0$, i.e. $\lambda$ restricts to a character of $\mathfrak{b}$. Note that polarisations need not always exist.

\begin{lemma}\label{polarisation-properties}

\noindent Given $\lambda\in\mathfrak{h}^*$, let $\mathfrak{h}^{\lambda}$ be the subalgebra $\{u\in\mathfrak{h}:\lambda([u,\mathfrak{h}])=0\}$ of $\mathfrak{g}$. If $\mathfrak{b}$ is a polarisation of $\mathfrak{h}$ at $\lambda$, then: 

\begin{itemize}

\item $\mathfrak{h}^{\lambda}\subseteq\mathfrak{b}$.

\item $\dim_F\mathfrak{b}=\frac{1}{2}(\dim_F\mathfrak{h}+\dim_F\mathfrak{h}^{\lambda})$.

\item If $\mathfrak{b}'$ is a subalgebra of $\mathfrak{h}$ such that $\lambda([\mathfrak{b}',\mathfrak{b}'])=0$ and $\dim_F\mathfrak{b}'=\frac{1}{2}(\dim_F\mathfrak{h}+\dim_F\mathfrak{h}^{\lambda})$, then $\mathfrak{b}'$ is a polarisation of $\mathfrak{h}$ at $\lambda$.

\item $\mathfrak{b}$ contains every ideal $\mathfrak{a}$ of $\mathfrak{h}$ such that $\lambda([\mathfrak{h},\mathfrak{a}])=0$. In particular, $\mathfrak{b}$ contains $Z(\mathfrak{h})$ and every ideal $\mathfrak{a}$ such that $\lambda(\mathfrak{a})=0$.

\end{itemize}

\end{lemma}

\begin{proof} 

The first statement follows from the definition of a polarisation, since $\lambda([\mathfrak{b},\mathfrak{h}^{\lambda}])=0$, while the second and third follow from \cite[1.12.1]{Dixmier}. The final statement follows from the first since if $\lambda([\mathfrak{a},\mathfrak{h}])=0$ then $\mathfrak{a}\subseteq\mathfrak{h}^{\lambda}\subseteq\mathfrak{b}$.\end{proof}

\vspace{0.2in}

\noindent\textbf{Examples:} 1. If $\lambda$ is a character of $\mathfrak{h}$, i.e. $\lambda([\mathfrak{h},\mathfrak{h}])=0$, then $\mathfrak{h}$ is a polarisation of $\mathfrak{h}$ at $\lambda$. In particular, if $\mathfrak{h}$ is abelian, then for any $\lambda\in\mathfrak{h}^*$, $\mathfrak{h}$ is a polarisation of $\mathfrak{h}$ at $\lambda$.\\

\noindent 2. If $\mathfrak{h}=\mathfrak{a}\rtimes kx$ for some abelian subalgebra $\mathfrak{a}$ of $\mathfrak{h}$, $x\in\mathfrak{h}$, then for any $\lambda\in\mathfrak{h}^*$, if $\lambda([\mathfrak{h},\mathfrak{h}])\neq 0$ then $\mathfrak{a}$ is a polarisation of $\mathfrak{h}$ at $\lambda$.\\

\noindent 3. If $k$ is algebraically closed, and $\mathfrak{h}$ is semisimple, then $\lambda\in\mathfrak{h}^*$ has a polarisation $\mathfrak{b}$ if and only if $\lambda$ is \emph{regular} in the sense of \cite[1.11.6]{Dixmier}. In this case $\mathfrak{b}$ is a Borel subalgebra of $\mathfrak{h}$ by \cite[Proposition 1.12.18]{Dixmier}, and hence $\lambda$ is a character of a Cartan subalgebra.\\

\noindent In our case, we will be interested in the case where $\mathfrak{h}$ is solvable or nilpotent. Recall that $\mathfrak{h}$ is \emph{completely solvable} if there exists a chain of ideals $0=\mathfrak{h}_0\subseteq\mathfrak{h}_1\subseteq\cdots\subseteq\mathfrak{h}_d=\mathfrak{h}$ such that dim$_k\mathfrak{h}_i=i$ for each $i$, e.g. if $\mathfrak{h}$ is nilpotent. Note that for any solvable Lie algebra $\mathfrak{h}$, we can choose a finite extension $F/k$ such that $\mathfrak{h}\otimes_{k}F$ is completely solvable.

\begin{proposition}\label{standard-polarisation}

Suppose that $\mathfrak{h}$ is completely solvable. Then given $\lambda\in\mathfrak{h}^*$, and an ideal $\mathfrak{a}$ of $\mathfrak{g}$ such that $\lambda([\mathfrak{a},\mathfrak{a}])=0$, there exists a polarisation $\mathfrak{b}$ of $\mathfrak{h}$ at $\lambda$ such that $\mathfrak{a}\subseteq\mathfrak{b}$.

\end{proposition}

\begin{proof}

Choose a chain of ideals $0=\mathfrak{h}_0\subseteq\mathfrak{h}_1\subseteq\cdots\subseteq\mathfrak{h}_d=\mathfrak{h}$ such that dim$_k\mathfrak{h}_i=i$ for each $i$, and we may choose this chain such that $\mathfrak{h}_j=\mathfrak{a}$ for some $j$.\\

\noindent Setting $\lambda_i:=\lambda|_{\mathfrak{h}_i}$ for each $i$, the subalgebra $\mathfrak{b}:=\mathfrak{h}_1^{\lambda_1}+\cdots+\mathfrak{h}_d^{\lambda_d}$ is a polarisation of $\mathfrak{h}$ at $\lambda$ by \cite[Proposition 1.12.18]{Dixmier}. Moreover, since $\lambda([\mathfrak{a},\mathfrak{a}])=0$, it follows that $\mathfrak{h}_j^{\lambda_j}=\mathfrak{a}^{\lambda_j}=\mathfrak{a}$, and hence $\mathfrak{a}\subseteq\mathfrak{b}$ as required.\end{proof}

\noindent In particular, taking $\mathfrak{a}=0$, this result ensures that polarisations always exist if $\mathfrak{h}$ is completely solvable. In this paper, we mainly focus on the case where $\mathfrak{h}$ is nilpotent.\\

\noindent Now, for any $k$-Lie algebra $\mathfrak{h}$, $\lambda\in\mathfrak{h}^*$, and any polarisation $\mathfrak{b}$ of $\mathfrak{h}$ at $\lambda$, since $\lambda$ restricts to a character of $\mathfrak{b}$, it follows that $k_{\lambda}:=kv$ is a $U(\mathfrak{b})$-module via the $\mathfrak{b}$-action $x\cdot v:=\lambda(x)v$. This gives us the following definition:

\begin{definition}\label{Dixmier-module}

Let $\lambda\in\mathfrak{h}^*$, and let $\mathfrak{b}$ be a polarisation of $\mathfrak{h}$ at $\lambda$. We define the $\mathfrak{b}$-\emph{Dixmier module} of $\mathfrak{h}$ at $\lambda$ to be the $U(\mathfrak{h})$-module:

\begin{equation}
D(\lambda)=D(\lambda)_{\mathfrak{b}}:=U(\mathfrak{h})\otimes_{U(\mathfrak{b})}k_{\lambda}
\end{equation}

\end{definition}

\noindent Note that $D(\lambda)$ is a cyclic $U(\mathfrak{h})$ module, generated by a vector $v_{\lambda}$ on which $U(\mathfrak{b})$ acts by scalars.\\

This definition is useful, because in the case where $k$ is algebraically closed and $\mathfrak{h}$ is semisimple, these Dixmier modules are precisely the well-known Verma modules, which are fundamental within the representation theory of semisimple Lie algebras. So we may think of Dixmier modules as a generalisation of Verma modules.\\

\noindent\textbf{Examples:} 1. If $\mathfrak{h}$ is semisimple, $k$ is algebraically closed, the Verma module $D(\lambda)$ has a unique simple quotient $L(\lambda)$, known as a \emph{simple highest weight module} with weight $\lambda$.\\

\noindent 2. If $\mathfrak{h}$ is abelian, or more generally if $\lambda$ is a character of $\mathfrak{h}$, then $D(\lambda)=k$ always.\\

\noindent 3. If $\mathfrak{h}=\mathfrak{a}\rtimes kx$, for $\mathfrak{a}$ abelian, then if $\lambda$ is not a character of $\mathfrak{h}$, $D(f)\cong k[t]$, where $x$ acts by $t$, and each $u\in\mathfrak{a}$ acts by a polynomial in $k[\frac{d}{dt}]$.

\begin{theorem}[{\cite[Theorem 6.1.1]{Dixmier}}]\label{Dix1}

Let $\mathfrak{h}$ be solvable, and let $\lambda\in\mathfrak{h}^*$. Then there exists a polarisation $\mathfrak{b}$ of $\mathfrak{h}$ at $\lambda$ such that $D(\lambda)_{\mathfrak{b}}$ is an irreducible $U(\mathfrak{h})$-module.

\end{theorem}

\noindent The following result yields a complete description of primitive ideals in $U(\mathfrak{h})$:

\begin{theorem}[{\cite[Theorem 6.1.7]{Dixmier}}]\label{Dix2}

Let $\mathfrak{h}$ be solvable, and let $P$ be a primitive ideal of $\mathfrak{h}$. Then there exists a finite extension $F/k$, $\lambda\in\mathfrak{h}_F^*$ and a polarisation $\mathfrak{b}$ of $\mathfrak{h}_F$ at $\lambda$ such that $P=$\emph{ Ann}$_{U(\mathfrak{h})}D(\lambda)_{\mathfrak{b}}$.

\end{theorem}

\noindent Now we will return to the $p$-adic case, i.e. $K/\mathbb{Q}_p$ is a finite extension, $\mathfrak{g}$ is a $K$-Lie algebra and $\mathcal{L}$ is an $\mathcal{O}$-lattice in $\mathfrak{g}$. We want to extend the conclusion of Theorems \ref{Dix1} and \ref{Dix2} to the affinoid setting.

In this case, when we choose $\lambda\in\mathfrak{g}^*$, we will always assume further that $\lambda(\mathcal{L})\subseteq\mathcal{O}$, or in other words $\lambda\in\mathcal{L}^*=$ Hom$_{\mathcal{O}}(\mathcal{L},\mathcal{O})$.\\

\noindent Firstly, note that for any polarisation $\mathfrak{b}$ of $\mathfrak{g}$ at $\lambda$, if we set $\mathcal{B}:=\mathfrak{b}\cap\mathcal{L}$, then $\mathcal{B}$ is an $\mathcal{O}$-Lie lattice in $\mathfrak{b}$. Furthermore, if we let $K_{\lambda}:=Kv$ be the one dimensional $U(\mathfrak{b})$-module induced by $\lambda$, then since $\pi^n U(\mathcal{B})v\subseteq p^n\mathcal{O}v$, it follows that and $K_{\lambda}$ carries the structure of a $\widehat{U(\mathcal{B})}_K$-module.

\begin{definition}\label{aff-Dixmier-module}

Let $\lambda\in\mathfrak{g}^*$ such that $\lambda(\mathcal{L})\subseteq\mathcal{O}$, and let $\mathfrak{b}$ be a polarisation of $\mathfrak{g}$ at $\lambda$. Define the $\mathfrak{b}$-\emph{affinoid Dixmier module} of $\mathcal{L}$ at $\lambda$ to be the $\widehat{U(\mathcal{L})}_K$-module defined by:

\begin{equation}
\widehat{D(\lambda)}=\widehat{D(\lambda)}_{\mathfrak{b}}:=\widehat{U(\mathcal{L})}_K\otimes_{\widehat{U(\mathcal{B})}_K}K_{\lambda}
\end{equation}

\end{definition}

\noindent\textbf{Notation:} If it is unclear what the ground field $K$ is, we may sometimes write $\widehat{D(\lambda)}_K$ for $\widehat{D(\lambda)}$. Also, if it is unclear which lattice $\mathcal{L}$ we are considering, we may sometimes write $\widehat{D(\lambda)}_{\mathcal{B}}$ instead of $\widehat{D(\lambda)}_{\mathfrak{b}}$.\\

\noindent Note that as in the classical case, $\widehat{D(\lambda)}$ is a cyclic $\widehat{U(\mathcal{L})}_K$-module, so $\widehat{D(\lambda)}_{\mathfrak{b}}=\widehat{U(\mathcal{L})}_Kv_{\lambda}$, and $\widehat{U(\mathcal{B})}_K$ acts by scalars on $v_{\lambda}$.

In particular, using Proposition \ref{fin-gen}, we see that $\widehat{D(\lambda)}$ is $\pi$-adically complete with respect to some lattice. In fact it is a $\pi$-adic completion of the classical Dixmier module $D(\lambda)$.\\

\noindent\textbf{Examples:} 1. If $\mathfrak{g}$ is split semisimple with Borel subalgebra $\mathfrak{b}$, the \emph{affinoid Verma module} $\widehat{V(\lambda)}$ arises as a Dixmier module. This still has a unique simple quotient $\widehat{L(\lambda)}$.\\

\noindent 2. If $\mathfrak{g}$ is abelian, or more generally if $\lambda$ is a character of $\mathfrak{g}$, then $\widehat{D(\lambda)}=K$ always.\\

\noindent 3. If $\mathfrak{g}=\mathfrak{a}\rtimes Kx$, for $\mathfrak{a}$ abelian, then if $\lambda$ is not a character of $\mathfrak{g}$, $\widehat{D(f)}\cong K\langle t\rangle$, where $x$ acts by $t$, and each $u\in\mathfrak{a}$ acts by a polynomial in $K[\frac{d}{dt}]$.

\subsection{Reducing Quadruples}

The following definition (\cite[4.7.7]{Dixmier})) will be very useful to us throughout.

\begin{definition}\label{reducing-quadruple}

Let $\mathfrak{h}$ be a $k$-lie algebra. A \emph{reducing quadruple} of $\mathfrak{h}$ is a 4-tuple $(x,y,z,\mathfrak{h}')$ where: 

\begin{itemize}

\item $0\neq x,y,z\in\mathfrak{h}$ and $\mathfrak{h}'$ is an ideal of $\mathfrak{h}$ of codimension 1,

\item $y,z\in\mathfrak{h}'$ and $x\notin\mathfrak{h}'$,

\item $z$ is central in $\mathfrak{h}$ and $y$ is central in $\mathfrak{h}'$, 

\item $[x,y]=\alpha z$ for some $0\neq\alpha\in k$.
\end{itemize}

\end{definition}

\noindent The following results link reducing quadruples to polarisations, and they can be found in the proof of \cite[Theorem 6.1.1]{Dixmier} and \cite[Theorem 6.1.4]{Dixmier}:

\begin{proposition}\label{sub-polarisation}

Suppose that $\mathfrak{h}$ is nilpotent, $\dim(\mathfrak{h})>1$, and let $\lambda\in\mathfrak{h}^*$. If we assume that $\lambda(\mathfrak{a})\neq 0$ for all non-zero ideals $\mathfrak{a}$ of $\mathfrak{g}$, then:\\

\noindent $(i)$ There exist $x,y,z\in\mathfrak{g}$ and $\mathfrak{g}'\trianglelefteq\mathfrak{g}$ such that $(x,y,z,\mathfrak{g}')$ forms a reducing quadruple for $\mathfrak{g}$.\\ 

\noindent $(ii)$ If $\mu:=\lambda|_{\mathfrak{g}'}$ and $\mathfrak{b}\subseteq\mathfrak{g}'$ is a polarisation of $\mathfrak{g}'$ at $\mu$, then $\mathfrak{b}$ is a polarisation of $\mathfrak{g}$ at $\lambda$.

\end{proposition}

\begin{proof}

$(i)$ Firstly, for any $0\neq z$ in the centre of $\mathfrak{g}$, $Kz$ is an ideal of $\mathfrak{g}$, and hence $\lambda(z)\neq 0$. Therefore, $\lambda:Z(\mathfrak{g})\to K$ is injective, meaning that $Z(\mathfrak{g})$ must have dimension 0 or 1.\\

\noindent But since $\mathfrak{g}$ is nilpotent, the centre cannot be 0, hence $Z(\mathfrak{g})=Kz$ for some $z\in Z(\mathfrak{g})$. Thus note that $\mathfrak{g}$ is non-abelian since dim$(\mathfrak{g})>1=$ dim$Z(\mathfrak{g})$.\\

\noindent Also using nilpotency of $\mathfrak{g}$ we may choose $y\in\mathfrak{g}$ such that $y\notin Z(\mathfrak{g})$ and $[y,\mathfrak{g}]\subseteq Z(\mathfrak{g})=Kz$. Thus the linear map $\ad(y):\mathfrak{g}\to\mathfrak{g}$ has rank 1, and hence it must have kernel of dimension dim$(\mathfrak{g})-1$. Let $\mathfrak{g}':=\ker(\ad(y))$, and it follows that $\mathfrak{g}'$ is an abelian ideal in $\mathfrak{g}$ of codimension 1.\\

\noindent So, $\mathfrak{g}=\mathfrak{g}'\oplus Kx$ for some $x\in\mathfrak{g}$ such that $[x,y]\neq 0$, and hence $[x,y]=\alpha z$ for some $\alpha\in K$ with $\alpha\neq 0$. Hence $(x,y,z,\mathfrak{g}')$ is a reducing quadruple as required.\\

\noindent $(ii)$ Since $\mathfrak{b}$ is a polarisation of $\mathfrak{g}'$ at $\mu$, it must contain $Z(\mathfrak{g}')$ by Lemma \ref{polarisation-properties}, hence $y,z\in\mathfrak{b}$.\\

\noindent Clearly $\mathfrak{b}\subseteq\mathfrak{g}'$ is a solvable subalgebra of $\mathfrak{g}$, so suppose that $\mathfrak{b}\subseteq V\subseteq\mathfrak{g}$ with $\lambda([V,V])=0$. We will prove that $V\subseteq\mathfrak{g}'$, and it will follow that $V=\mathfrak{b}$ by the definition of a polarisation, thus proving that $\mathfrak{b}$ is a polarisation of $\mathfrak{g}$ at $\lambda$.\\

\noindent Suppose that $\beta x+u\in V$ for some $u\in\mathfrak{g}'$, $\beta\in K$. Then since $y\in\mathfrak{b}\subseteq V$, it follows that $\lambda([\beta x+u,y])=0$, and hence $\beta\alpha\lambda(z)=0$. But since $\alpha,\lambda(z)\neq 0$, it follows that $\beta=0$ and hence $V\subseteq\mathfrak{g}'$ as required.\end{proof}

\begin{proposition}\label{small-pol}

Suppose that $\mathfrak{h}$ is nilpotent with $n:=\dim_F\mathfrak{h}>3$, and suppose that $\mathfrak{h}$ has a reducing quadruple $(x,y,z,\mathfrak{h}')$. We assume further that $\lambda\in\mathfrak{h}^*$ and $\lambda(\mathfrak{a})\neq 0$ for all non-zero ideals $\mathfrak{a}$ of $\mathfrak{h}$. Then for any polarisation $\mathfrak{b}$ of $\mathfrak{h}$ at $\lambda$, there exists a polarisation $\mathfrak{b}'$ at $\lambda$, contained in $\mathfrak{h}'$, and a proper subalgebra $\mathfrak{t}\subsetneq\mathfrak{h}$ such that $\mathfrak{b},\mathfrak{b}'\subseteq\mathfrak{t}$.

\end{proposition}

\begin{proof}

Firstly, note that $\mathfrak{h}^{\lambda}=\{u\in\mathfrak{h}:\lambda([u,\mathfrak{h}])=0\}$ is contained in $\mathfrak{b}$, otherwise $\mathfrak{b}\subsetneq\mathfrak{b}+\mathfrak{h}^{\lambda}$ and $\lambda([\mathfrak{b}+\mathfrak{h}^{\lambda},\mathfrak{b}+\mathfrak{h}^{\lambda}])=0$, contradicting the definition of a polarisation.\\

\noindent If $\mathfrak{b}\subseteq\mathfrak{h}'$, then taking $\mathfrak{b}'=\mathfrak{b}$, $\mathfrak{t}=\mathfrak{h}'$, the statement is trivially true, so we may assume that $\mathfrak{b}\not\subseteq\mathfrak{h}'$.\\

\noindent Since $y$ is central in $\mathfrak{h}'$ but not in $\mathfrak{h}$, it is clear that $\mathfrak{h}'=\ker(\ad(y))$ and $Fz=\im(\ad(y))$. So since there exists $u\in\mathfrak{b}\backslash\mathfrak{h}'$, we have that $[u,y]\neq 0$, i.e. $[u,y]=\beta z$ for some $0\neq\beta\in F$. But since $Fz$ is a non-zero ideal of $\mathfrak{h}$, $\lambda(z)\neq 0$, and thus $\lambda([u,y])\neq 0$. So since $\lambda([\mathfrak{b},\mathfrak{b}])=0$ and $u\in\mathfrak{b}$, this means that $y\notin\mathfrak{b}$.\\

\noindent Let $\mathfrak{b}':=(\mathfrak{b}\cap\mathfrak{h}')\oplus Fy$. This is a subalgebra of $\mathfrak{h}$, and clearly it is contained in $\mathfrak{h}'$. Also, $\mathfrak{b}\cap\mathfrak{h}'$ has codimension 1 in $\mathfrak{b}$, therefore $\dim_F\mathfrak{b}'=\dim_F\mathfrak{b}\cap\mathfrak{h}'+1=\dim_F\mathfrak{b}$. But it is clear that $\lambda([\mathfrak{b}',\mathfrak{b}'])=0$, so by Lemma \ref{polarisation-properties}, this means that $\mathfrak{b}'$ is a polarisation of $\mathfrak{h}$ at $\lambda$.\\

Now, let $\mathfrak{t}:=\mathfrak{b}\oplus Fy$. Since $Fz=[y,\mathfrak{h}]\subseteq\mathfrak{b}$, this is a subalgebra of $\mathfrak{h}$, and clearly it contains $\mathfrak{b}$ and $\mathfrak{b}'$, so we only need to prove that $\mathfrak{t}\neq\mathfrak{h}$, so assume for contradiction that $\mathfrak{t}=\mathfrak{h}$. This means that $\mathfrak{b}$ has codimension 1 in $\mathfrak{h}$, so $\dim_{F}\mathfrak{b}=n-1$. But $\dim_F\mathfrak{b}=\frac{1}{2}(n+\dim_F\mathfrak{h}^{\lambda})$ by Lemma \ref{polarisation-properties}, and thus $\dim_F\mathfrak{h}^{\lambda}=n-2$.\\

\noindent Furthermore, if $\beta x+\gamma y\in\mathfrak{h}^{\lambda}$, then $\beta\lambda([x,y])=\gamma\lambda([y,x])=0$, which is only possible if $\beta=\gamma=0$ since $\lambda(z)=\lambda(\alpha^{-1}[x,y])\neq 0$. So Span$_F\{x,y\}\cap\mathfrak{h}^{\lambda}=0$, and therefore, $\mathfrak{h}=Fx\oplus Fy\oplus\mathfrak{h}^{\lambda}$.\\

\noindent Let $\mathfrak{a}:=\ker(\lambda)\cap\mathfrak{h}^{\lambda}$. Then since $z\in\mathfrak{h}^{\lambda}$ and $\lambda(z)\neq 0$, it follows that $\mathfrak{a}$ has codimension 1 in $\mathfrak{h}^{\lambda}$, which means that $\dim_F\mathfrak{a}=n-3$.\\

\noindent It is clear that $\lambda(\mathfrak{a})=0$, so we will finish by proving that $\mathfrak{a}$ is an ideal of $\mathfrak{h}$, and this will imply that $\mathfrak{a}=0$, and hence $n-3=0$ and $n=\dim_F\mathfrak{h}=3$ -- contradicting our assumption.\\

\noindent By the definition of $\mathfrak{h}^{\lambda}$, it is clear that $\lambda([\mathfrak{h}^{\lambda},\mathfrak{h}])=0$ and so $[\mathfrak{h}^{\lambda},\mathfrak{a}]\subseteq\mathfrak{h}^{\lambda}\cap \ker(\lambda)=\mathfrak{a}$. So since $\mathfrak{h}=Fx\oplus Fy\oplus\mathfrak{h}^{\lambda}$, it remains to prove that $[y,\mathfrak{a}]\subseteq\mathfrak{a}$ and $[x,\mathfrak{a}]\subseteq\mathfrak{a}$.\\

\noindent Since $\mathfrak{a}\subseteq\mathfrak{h}'$, we have that $[y,\mathfrak{a}]=0\subseteq\mathfrak{a}$, and if we choose $u\in\mathfrak{b}$ such that $u\notin\mathfrak{h}'$, then $\mathfrak{b}=\mathfrak{h}^{\lambda}\oplus Fu$, so since $\mathfrak{h}$ is nilpotent and $\mathfrak{b}$ is a subalgebra, it follows that $[u,\mathfrak{h}^{\lambda}]\subseteq\mathfrak{h}^{\lambda}$, and hence $[u,\mathfrak{a}]\subseteq\mathfrak{h}^{\lambda}$. Also, since $\lambda([\mathfrak{b},\mathfrak{b}])=0$, it follows that $[u,\mathfrak{a}]\subseteq \ker(\lambda)$, hence $[u,\mathfrak{a}]\subseteq\mathfrak{a}$. 

But since $u\notin\mathfrak{h}'$, we have that $u=\beta x+\gamma y$, where $\beta\neq 0$, so it follows immediately that $[x,\mathfrak{a}]\subseteq\mathfrak{a}$ as required.\end{proof}

\noindent Reducing quadruples will play an important role in many of the proofs in this paper, since they allow us to use an inductive strategy commonly employed by Dixmier in \cite{Dixmier}, which we outline below, very roughly:\\

\noindent\textbf{Dixmier's Induction Strategy}: We have a statement $\mathcal{P}$ involving a nilpotent $k$-Lie algebra $\mathfrak{h}$ and a linear form $\lambda:\mathfrak{h}\to k$.

\begin{itemize}

\item Step 1: The base case is where dim$_k\mathfrak{h}=1$, this case should be straightforward. So we can assume that dim$_k\mathfrak{h}>1$ and $\mathcal{P}$ is true for all nilpotent Lie algebras of dimension less than dim$_k\mathfrak{h}$. 

\item Step 2: If there exists a non-zero ideal $\mathfrak{a}$ of $\mathfrak{h}$ such that $\lambda(\mathfrak{a})=0$, then we can replace $\mathfrak{h}$ by $\frac{\mathfrak{h}}{\mathfrak{a}}$ and apply induction. So we may assume that $\lambda(\mathfrak{a})\neq 0$ for all non-zero ideals $\mathfrak{a}$ of $\mathfrak{g}$.

\item Step 3: Applying Proposition \ref{sub-polarisation}, we can choose a reducing quadruple $(x,y,z,\mathfrak{h}')$ of $\mathfrak{h}$. Since we know that $\mathcal{P}$ holds for $\mathfrak{h}'$ and $\lambda|_{\mathfrak{h}'}$ by induction, we can hopefully induce up to $\mathfrak{h}$.

\end{itemize}

\subsection{Tate-Weyl algebras}

\noindent Let us first recall the definition of the Weyl algebra, given in \cite[Section 1.3]{McConnell}:

\begin{definition}\label{Weyl}

Let $d\in\mathbb{N}$. We define the \emph{$d$'th Weyl algebra} $A_d(k)$ to be the $k$-algebra in $2d$-variables $x_1,\cdots,x_d,y_1,\cdots,y_d$ satisfying the following relations for $1\leq i,j\leq d$, $i\neq j$: 

\begin{center}
$x_ix_j-x_jx_i=0,
y_iy_j-y_jy_i=0,
x_iy_j-y_jx_i=0,
x_iy_i-y_ix_i=1$.
\end{center}

\end{definition}

\begin{lemma}
$A_d(k)$ is isomorphic as a $k$-vector space to the polynomial ring $k[x_1,\cdots,x_d,\partial_1,\cdots,\partial_d]$, subject to the relations that $x_1,\cdots,x_d$ commute, $\partial_1,\cdots,\partial_d$ commute, $x_i\partial_j=\partial_jx_j$ if $i\neq j$ and $x_i\partial_i=\partial_ix_i+1$.
\end{lemma}

\noindent\textbf{Note:} It is well known and quite straightforward to prove that $A_n(k)$ is a simple $k$-algebra domain.\\ 

\noindent The Weyl algebra is of great importance within algebraic and differential geometry, since it can be realised as the ring of differential operators on the affine plane $\mathbb{A}^d_k$. The following result highlights the usefulness of this object within representation theory, since it can be explicitly realised as a ring of endomorphisms:

\begin{lemma}\label{endo}

Let $k[t_1,\cdots,t_d]$ be a polynomial ring in $d$-variables, and for each $i=1,\cdots,d$, let $x_i\in$\emph{ End}$_k k[t_1,\cdots,t_d]$ be left multiplication by $t_i$, and let $\partial_i:=\frac{d}{dt_i}\in$\emph{ End}$_k k[t_1,\cdots,t_d]$. Then the natural map $A_d(k)\to$\emph{ End}$_k k[t_1,\cdots,t_d]$ sending $x_i$ to $x_i$ and $\partial_i$ to $\partial_i$ is an injective ring homomorphism.

\end{lemma}

Now, if we fix $\mathfrak{h}$ a nilpotent $k$-Lie algebra, the following result of Dixmier (\cite[Theorem 4.7.9]{Dixmier}) shows us that all weakly rational ideals in $U(\mathfrak{h})$ are maximal, which completes the Dixmier-Moeglin equivalence in the nilpotent case:

\begin{theorem}\label{Dix3}

Let $I$ be a two-sided of $U(\mathfrak{h})$ such that $Z\left(\frac{U(\mathfrak{h})}{I}\right)=k$, then there exists $d\in\mathbb{N}$ such that $\frac{U(\mathfrak{h})}{I}\cong A_d(k)$. The integer $d$ is called the \emph{weight} of $I$.

\end{theorem}

\noindent The proof of this Theorem is inductive, not constructive, so we cannot usually explicitly write down this isomorphism. However, using Theorem \ref{Dix2}, we know that all weakly rational ideals of $U(\mathfrak{h})$ arise as annihilators of Dixmier modules, and using this notion there are cases where we can construct an isomorphism.\\

\noindent Specifically, if $I= $ Ann$_{U(\mathfrak{h})}D(\lambda)$ for some linear form $\lambda$ of $\mathfrak{h}$, with polarisation $\mathfrak{b}$ of codimension $d$, then we know that $\frac{U(\mathfrak{h})}{I}$ is a ring of $k$-linear endomorphisms of $D(\lambda)\cong k[u_1,\cdots,u_d]$. Since we know from Lemma \ref{endo} that $A_d(k)$ can be realised as a ring of endomorphisms of $k[u_1,\cdots,u_d]$, we want to show that these to rings can be identified.

Roughly speaking, we want to show that elements in the polarisation $\mathfrak{b}$ act by polynomials in $\partial_1,\cdots,\partial_d$, and that each basis vector $u_i$ of $\mathfrak{g}/\mathfrak{b}$ acts by $x_i$. In the next section, we will see how to do this explicitly.\\

\noindent Now, we want to try and generalise Theorem \ref{Dix3} to the affinoid case, so we need to define an affinoid version of the Weyl algebra:

\begin{definition}

Let $A_d(\mathcal{O})$ be the $\mathcal{O}$-lattice subalgebra of $A_d(K)$ consisting of polynomials in $x_1,\cdots,x_d,\partial_1,\cdots,\partial_d$ with coefficients in $\mathcal{O}$. We define the $d$'th \emph{Tate-Weyl algebra}, denoted $\widehat{A_d(K)}$ to be the completion of $A_d(K)$ with respect to $A_d(\mathcal{O})$

\end{definition}

\noindent Similarly to $A_d(K)$, the Tate-Weyl algebra $\widehat{A_d(K)}$ is also a simple domain.

\begin{lemma}\label{Tate-Weyl}

The $d$'th Tate-Weyl algebra $\widehat{A_d(K)}$ is isomorphic as a $K$-vector space to $K\langle x_1,\cdots,x_d,\partial_1,\cdots,\partial_d\rangle$, where $x_1,\cdots,x_d$ commute, $\partial_1,\cdots,\partial_d$ commute, $x_i\partial_j=\partial_jx_j$ if $i\neq j$ and $x_i\partial_i=\partial_ix_i+1$.

Moreover, the level $n$ deformation $\widehat{A_d(K)}_n$ of $\widehat{A_d(K)}$ is the subalgebra consisting the elements of $K\langle\pi^n  x_1,\cdots,\pi^n x_d,\pi^n \partial_1,\cdots,\pi^n\partial_d\rangle$.

\end{lemma}

\noindent Using this definition, we have the following affinoid version of Lemma \ref{endo}:

\begin{lemma}\label{aff-endo}

Let $K\langle t_1,\cdots,t_d\rangle$ be a Tate algebra in $d$-variables, and for each $i=1,\cdots,d$, let $x_i\in$\emph{ End}$_K K\langle t_1,\cdots,t_d\rangle$ be left multiplication by $t_i$, and let $\partial_i:=\frac{d}{dt_i}\in$\emph{ End}$_K K\langle t_1,\cdots,t_d\rangle$. Note that if $\lambda_{\alpha}\in K$ for each $\alpha\in\mathbb{N}^{2d}$, and $\lambda_{\alpha}\rightarrow 0$ as $\alpha\rightarrow\infty$, then the series $\underset{\alpha\in\mathbb{N}^{2d}}{\sum}{\lambda_{\alpha}x_1^{\alpha_1}\cdots x_d^{\alpha_d}\partial_1^{\alpha_{d+1}}\cdots\partial_d^{\alpha_d}}$ is a well defined $K$-linear endomorphism of $K\langle t_1,\cdots,t_d\rangle$. 

Then the natural map $A_d(k)\to$\emph{ End}$_K K\langle t_1,\cdots,t_d\rangle$ sending $x_i$ to $x_i$ and $\partial_i$ to $\partial_i$ is an injective ring homomorphism.

\end{lemma}

In a similar vein to Theorem \ref{Dix3}, we would like to prove that all weakly rational quotients of the affinoid enveloping algebra $\widehat{U(\mathcal{L})}_K$ are isomorphic to Tate-Weyl algebras. However, it was shown in \cite{Lewis} that this need not always be the case. In section 7 we will explore how to prove a related, weaker statement.

\subsection{The Coadjoint Action}

Now we will recall some Lie theory. Assume $\mathfrak{h}$ is nilpotent Lie algebra, and note that for every $u\in\mathfrak{h}$, the map $\ad(u)$ is a nilpotent derivation of $\mathfrak{h}$. So we can define:

\begin{equation}
\exp(\ad(u)):=\underset{n\geq 0}{\sum}{\frac{1}{n!}\ad(u)^n}:\mathfrak{h}\to\mathfrak{h}
\end{equation}

\noindent Since $\ad(u)$ is a derivation, it follows that $\exp(\ad(u))$ is a Lie-automorphism of $\mathfrak{h}$.

\begin{definition}

Define the \emph{adjoint algebraic group} of $\mathfrak{h}$ to be $\mathbb{H}(\mathfrak{h}):=\{\exp(\ad(u))\in Aut(\mathfrak{h}):u\in\mathfrak{h}\}$

\end{definition}

\noindent Then $\mathbb{H}(\mathfrak{h})$ is a subgroup of $Aut(\mathfrak{h})$, and if we define the functor $\mathbb{H}:k$-Alg$\to$ Grp,$R\mapsto\mathbb{H}(\mathfrak{h}\otimes_k R)$, then $\mathbb{H}$ is an affine algebraic group in the sense of \cite[Definition \rom{1}.2.1]{Jantzen}, and it is unipotent.\\

\noindent\textbf{Note:} If we view the space $\ad(\mathfrak{h})\subseteq End_k(\mathfrak{h})$ as an affine variety over $k$, then the map $\exp:\ad(\mathfrak{h})\to\mathbb{H}$ is an isomorphism of varieties, with inverse $\log$.\\

\noindent Now, let $\mathfrak{h}^*:k$-Alg$\to$ Set be the linear dual of $\mathfrak{h}$, i.e. $\mathfrak{h}^*(R)=Hom_R(\mathfrak{h}\otimes_k R,R)\cong\mathfrak{h}^*(k)\otimes_k R$. Then $\mathfrak{h}^*$ is an affine scheme in the sense of \cite[Definition \rom{1}.1.3]{Jantzen}.

\begin{definition}\label{coadjoint-action}

Define an action of $\mathbb{H}$ on $\mathfrak{g}^*$, i.e. a morphism of varieties $\mathbb{H}\times\mathfrak{h}^*\to\mathfrak{h}^*$, by $(g\cdot f)(u)=f(g^{-1}u)$. This is the \emph{coadjoint action}, and the orbits of this action in $\mathfrak{h}^*$ are called \emph{coadjoint orbits}.

\end{definition}

\noindent Given $\lambda\in\mathfrak{h}^*(k)$, let $X$ be the coadjoint orbit of $\lambda$ in $\mathfrak{h}^*$, and let $S$ be the stabiliser of $\lambda$ in $\mathbb{H}$, i.e. $S(R):=\{g\in\mathbb{H}(R):g\cdot \lambda=\lambda\}$, an affine algebraic subgroup of $\mathbb{H}$.

\begin{lemma}\label{product-decomposition}

There exists an isomorphism of varieties $\mathbb{H}\cong S\times X$ such that the map $\mathbb{H}\to X,g\mapsto g\cdot \lambda$ is just the natural projection $S\times X\to X$.

\end{lemma}

\begin{proof}

Since $\mathbb{H}$ is affine and $S$ is closed in $\mathbb{H}$, we see using \cite[Theorem \rom{2}.6.8]{Borel} that the quotient variety $\mathbb{H}/S$ exists. Since the orbit map $\mathbb{H}\to X$ is surjective and $k$ has characteristic 0, it follows that this map is \emph{separable} in the sense of \cite[\rom{1}.8.2]{Borel}, and hence using \cite[Theorem \rom{1}.17.3]{Borel} and \cite[Proposition \rom{2}.6.7]{Borel} it follows that $X=\mathbb{H}/S$ and that the map $\mathbb{H}\to X$ is the quotient map. Since $X$ is closed in $\mathfrak{h}^*$, it follows that $\mathbb{H}/S$ is affine.\\

\noindent Now, using \cite[\rom{1}.5.6(1)]{Jantzen} we see that $\mathbb{H}\times S\cong\mathbb{H}\times_{\mathbb{H}/S}\mathbb{H}$ as varieties, and using this isomorphism, it follows that the natural map $\mathbb{H}\to\mathbb{H}/S$ is an $S$-torsor in the sense of \cite[Ch.\rom{3} Definition 4.1.3]{torsor}.\\ 

\noindent Therefore, since $S$ is unipotent and $\mathbb{H}/S$ is affine, it follows from \cite[Ch.\rom{4} Proposition 3.7(b)]{torsor} that the torsor $\mathbb{H}\to\mathbb{H}/S$ is trivial, i.e. $\mathbb{H}\cong S\times\mathbb{H}/S=S\times X$ as varieties and the map $\mathbb{H}\to X$ is just the projection to the second factor.\end{proof}

\section{The Action of $\widehat{U(\mathcal{L})}_K$ on $\widehat{D(\lambda)}$}

In this section, we will prove our affinoid version of Theorem \ref{Dix1}, at least in the case where $\mathfrak{g}$ is nilpotent. Throughout, assume that $\mathfrak{g}$ is a finite dimensional $K$-Lie algebra, with $\mathcal{O}$-Lie lattice $\mathcal{L}$.

\subsection{Induced Modules}

\noindent Since the affinoid Dixmier module $\widehat{D(\lambda)}$ is an induced $\widehat{U(\mathcal{L})}_K$-module, we will first explore some general properties of induced modules.

\begin{lemma}\label{induced}

Let $\mathfrak{h}$ be a subalgebra of $\mathfrak{g}$, let $\mathfrak{a}$ be an ideal of $\mathfrak{g}$ such that $\mathfrak{a}\subseteq\mathfrak{h}$. Let $\mathfrak{g}_1:=\mathfrak{g}/\mathfrak{a}$, $\mathfrak{h}_1:=\mathfrak{h}/\mathfrak{a}$. Also, set $\mathcal{H}:=\mathcal{L}\cap\mathfrak{h}$, $\mathcal{A}:=\mathcal{L}\cap\mathfrak{a}$, $\mathcal{L}_1:=\mathcal{L}/\mathcal{A}$, $\mathcal{H}_1:=\mathcal{H}/\mathcal{A}$, which are Lie lattices in $\mathfrak{h}$, $\mathfrak{a}$, $\mathfrak{g}_1$ and $\mathfrak{h}_1$ respectively. Then:\\

\noindent $(i)$ There is a continuous surjection of $K$-algebras $\widehat{U(\mathcal{L})}_K\twoheadrightarrow\widehat{U(\mathcal{L}_1)}_K$ induced by the surjection $\mathcal{L}\twoheadrightarrow\mathcal{L}_1$. The kernel of this surjection is $\mathfrak{a}\widehat{U(\mathcal{L})}_K$\\

\noindent $(ii)$ If $M$ is a finitely generated $\widehat{U(\mathcal{H}_1)}_K$-module, then $M$ has the structure of a $\widehat{U(\mathcal{H})}_K$-module via the surjection in $(i)$, and $\widehat{U(\mathcal{L})}_K\otimes_{\widehat{U(\mathcal{H})}_K}M\cong\widehat{U(\mathcal{L}_1)}_K\otimes_{\widehat{U(\mathcal{H}_1)}_K}M$ as $\widehat{U(\mathcal{L})}_K$-modules.

\end{lemma}

\begin{proof}

$(i)$ It is clear that the surjection $\mathcal{L}\twoheadrightarrow\mathcal{L}_1$ induces a surjection $U(\mathcal{L})\twoheadrightarrow U(\mathcal{L}_1)$ sending $\pi^nU(\mathcal{L})$ to $\pi^nU(\mathcal{L}_1)$, so this yields a continuous map $\widehat{U(\mathcal{L})}\to\widehat{U(\mathcal{L}_1)}$.\\

\noindent If we fix a basis $\{x_1,\cdots,x_d\}$ for $\mathcal{L}$ such that $\{x_{r+1},\cdots,x_d\}$ is a basis for $\mathcal{A}$, then using Lemma \ref{PBW}, we see that every element of $\widehat{U(\mathcal{L}_1)}_K$ has the form $\underset{\alpha\in\mathbb{N}^r}{\sum}{\lambda_{\alpha}(x_1+\mathcal{A})^{\alpha_1}\cdots(x_r+\mathcal{A})^{\alpha_r}}$, where $\lambda_{\alpha}\in\mathcal{O}\rightarrow 0$ as $\alpha\rightarrow\infty$. Clearly under the map $\widehat{U(\mathcal{L})}\to\widehat{U(\mathcal{L}_1)}$, $x_i$ maps to $x_i+\mathcal{A}$ for each $i$, and hence the map is surjective.\\

\noindent Moreover, we can write any element of $\widehat{U(\mathcal{L})}$ as $\underset{\alpha\in\mathbb{N}^{r}}{\sum}{x_1^{\alpha_1}\cdots x_r^{\alpha_r}c_{\alpha}}$ for some $c_{\alpha}\in\widehat{U(\mathcal{A})}$ converging to zero, and this maps to 0 if and only if $c_{\alpha}$ maps to 0 for each $\alpha$. But each $c_{\alpha}$ has the form $\underset{\beta\in\mathbb{N}^{d-r}}{\sum}{\mu_{\beta}x_{r+1}^{\beta_{r+1}
}\cdots x_d^{\beta_d}}$, and this maps to zero if and only if $\mu_{0}=0$, i.e. $c_{\alpha}\in\widehat{U(\mathcal{A})}\mathcal{A}$. Hence the kernel of the surjection is $\widehat{U(\mathcal{L})}\mathcal{A}$ and part $(i)$ follows.\\

\noindent $(ii)$ Let $\phi:\widehat{U(\mathcal{L})}_K\to\widehat{U(\mathcal{L}_1)}_K$ be the surjection from part $(i)$, and define a map:

\begin{center}
$\Theta:\widehat{U(\mathcal{L})}_K\otimes_{\widehat{U(\mathcal{H})}_K}M\to\widehat{U(\mathcal{L}_1)}_K\otimes_{\widehat{U(\mathcal{H}_1)}_K}M,r\otimes m\mapsto \phi(r)\otimes m$.
\end{center}

\noindent It is clear that this is a well defined map of $\widehat{U(\mathcal{L})}_K$-modules, we want to prove that it is an isomorphism.\\

\noindent Every element $s\in\widehat{U(\mathcal{L}_1)}_K$ can be written uniquely in the form

\begin{center}
$s=\underset{\alpha\in\mathbb{N}^r}{\sum}{\lambda_{\alpha}(x_1+\mathcal{A})^{\alpha_1}\cdots (x_r+\mathcal{A})^{\alpha_r}}=\phi(\underset{\alpha\in\mathbb{N}^r}{\sum}{\lambda_{\alpha}x_1^{\alpha_1}\cdots x_r^{\alpha_r}})$,
\end{center} 

\noindent so there is a unique element in $K\langle x_1,\cdots,x_r\rangle$ that maps onto $s$ under $\phi$. We call this element $\phi^{-1}(s)$, and it is clear that this defines a $K$-linear map $\phi^{-1}:\widehat{U(\mathcal{L}_1)}\to K\langle x_1,\cdots,x_r\rangle$.\\

\noindent Therefore, we can define a $K$-linear map $\Psi:\widehat{U(\mathcal{L}_1)}_K\otimes_{\widehat{U(\mathcal{H}_1)}_K}M\to \widehat{U(\mathcal{L})}_K\otimes_{\widehat{U(\mathcal{H})}_K}M$ sending $s\otimes m$ to $\phi^{-1}(s)\otimes m$. We can show that this is well defined by choosing an appropriate basis for $\mathcal{H}_1$ that extends to a basis for $\mathcal{L}_1$, and clearly it is a right inverse to $\Theta$.\\

\noindent Using the fact that $\widehat{U(\mathcal{L})}_K$ is isomorphic as a $K$-vector space to $K\langle x_1,\cdots,x_r\rangle\langle x_{r+1},\cdots,x_d\rangle$, and $x_{r+1},\cdots,x_d\in\mathcal{A}\subseteq\mathcal{H}$, we see that every simple tensor $s\otimes m\in \widehat{U(\mathcal{L})}_K\otimes_{\widehat{U(\mathcal{H})}_K}M$ can be written as an infinite sum of simple tensors $s_n\otimes m_n$ converging to zero as $n\rightarrow\infty$, with $s_n\in K\langle x_1,\cdots,x_r\rangle$. We know this sum converges by Proposition \ref{fin-gen}.\\

\noindent Therefore, for any $s\in\widehat{U(\mathcal{L})}_K$, $m\in M$, $\Psi\Theta(s\otimes m)=\underset{n\in\mathbb{N}}{\sum}{\Psi\Theta(s_n\otimes m_n)}=\underset{n\in\mathbb{N}}{\sum}{\Psi(\phi(s_n)\otimes m_n)}$, and since $s_n\in K\langle x_1,\cdots,x_s\rangle$ for each $n$, $\phi^{-1}(\phi(s_n))=s_n$, and hence $\Psi\Theta(s\otimes m)=s\otimes m$. Thus $\Psi$ and $\Theta$ are mutually inverse bijections.\end{proof}

\vspace{0.2in}

\noindent Now, recall from \cite{DCap1} that if $A$ is a Banach $K$-algebra, and $M$ is a left $A$-module, $\pi$-adically complete with respect to some lattice $N\subseteq M$, then we may define the \emph{Tate module}: 

\begin{center}
$M\langle t_1,\cdots,t_d\rangle:=\{\underset{\alpha\in\mathbb{N}^d}{\sum}{t_1^{\alpha_1}\cdots t_d^{\alpha_d}s_{\alpha}}:s_{\alpha}\in M, s_{\alpha}\rightarrow 0$ as $\vert\alpha\vert\rightarrow\infty\}$.
\end{center}

\noindent Note that we don't necessarily give $M\langle t_1,\cdots,t_d\rangle$ the structure of an $A$-module, a priori it is just a $K$-vector space.

\begin{proposition}\label{Tate-action}

Let $\mathfrak{h}$ be a subalgebra of $\mathfrak{g}$, and let $\mathcal{H}:=\mathfrak{h}\cap\mathcal{L}$, so $\mathcal{H}$ is a Lie lattice in $\mathfrak{h}$. Suppose that $M$ is a finitely generated $\widehat{U(\mathcal{H})}_K$-module. Then if $r=\dim_K\mathfrak{g}/\mathfrak{h}$, there is an isomorphism of $K$-vector spaces $\widehat{U(\mathcal{L})}_K\otimes_{\widehat{U(\mathcal{H})}_K}M\cong M\langle t_1,\cdots,t_r\rangle$, where $t_iv$ corresponds to $x_i\otimes v$ for some $\mathcal{O}$-basis $\{x_1,\cdots,x_r\}$ for $\mathcal{L}/\mathcal{H}$. Thus $M\langle t_1,\cdots,t_r\rangle$ carries the structure of a $\widehat{U(\mathcal{L})}_K$ module.\\

\noindent Moreover, if $r=1$, so $\mathcal{L}=\mathcal{H}\oplus\mathcal{O}x$ for some $x\in\mathcal{L}$, then we can choose this isomorphism $\widehat{U(\mathcal{L})}_K\otimes_{\widehat{U(\mathcal{H})}_K}M\cong M\langle t\rangle$ such that:\\

\noindent $(i)$  $x$ acts by $t$ on $M\langle t\rangle$.\\

\noindent $(ii)$ If $y,z\in\mathcal{H}$ act on $M$ by scalars $\beta_y,\beta_z\in\mathcal{O}$, $[x,z]=0$ and $[y,x]=\alpha z$ for some $\alpha\in K$, then $y$ acts on $M\langle t\rangle$ by $\alpha\beta_z\frac{d}{dt}+\beta_y$.\\

\noindent $(iii)$ If $\alpha,\beta_z\neq 0$ and $M$ is irreducible over $\widehat{U(\mathcal{H})}_K$, then $M\langle t\rangle$ is irreducible over $\widehat{U(\mathcal{L})}_K$.

\end{proposition}

\begin{proof}

Let $\{x_1,\cdots,x_d\}$ be an $\mathcal{O}$-basis for $\mathcal{L}$ such that $\{x_{r+1},\cdots,x_d\}$ is a basis for $\mathcal{H}$. Then by Lemma \ref{PBW}, writing $\underline{x}^{\alpha}:=x_1^{\alpha_1}\cdots x_d^{\alpha_d}$, we have: $\widehat{U(\mathcal{L})}_K=\{\underset{\alpha\in\mathbb{N}^d}{\sum}{\lambda_{\alpha}\underline{x}^{\alpha}:\lambda_{\alpha}\in \mathcal{O}},\lambda_{\alpha}\rightarrow 0$ as $\vert\alpha\vert\rightarrow\infty\}$.\\

\noindent Define a map:

\begin{center}
$\Theta:\widehat{U(\mathcal{L})}_K\otimes_{\widehat{U(\mathcal{H})}_K}M\to M\langle t_1,\cdots,t_{r}\rangle,\underset{\alpha\in\mathbb{N}^d}{\sum}{\lambda_{\alpha}\underline{x}^{\alpha}}\otimes v\mapsto\underset{\beta\in\mathbb{N}^r}{\sum}{t_1^{\beta_1}\cdots t_r^{\beta_r}(\underset{\gamma\in\mathbb{N}^{d-r}}{\sum}{\lambda_{(\beta,\gamma)}\underline{x}^{\gamma}v})}$.
\end{center}

\noindent Note that here $(\beta,\gamma)$ refers to the $d$-tuple whose first $r$ terms are the terms of $\beta$, and the last $d-r$ terms are the terms of $\gamma$. It is straightforward but technical to show that this is a well defined $K$-linear map, so we need to prove that it is an isomorphism.\\

\noindent Firstly, $M=\widehat{U(\mathcal{H})}_Kv_1+\cdots+\widehat{U(\mathcal{H})}_Kv_t$, so any element of $M\langle t_1,\cdots,t_r\rangle$ will have the form $\underset{\beta\in\mathbb{N}^r}{\sum}{t_1^{\beta_1}\cdots t_r^{\beta_r}}(a_{1,\beta}v_1+\cdots+a_{t,\beta}v_t)$ for some $a_{i,\beta}\in\widehat{U(\mathcal{H})}_K$. This is the image of $\underset{\beta\in\mathbb{N}^r}{\sum}{x_1^{\beta_1}\cdots x_r^{\beta_r}a_{1,\beta}}\otimes v_1+\cdots+\underset{\beta\in\mathbb{N}^r}{\sum}{x_1^{\beta_1}\cdots x_r^{\beta_r}a_{t,\beta}}\otimes v_t$, so $\Theta$ is surjective.\\

\noindent Furthermore, if $\underset{\beta\in\mathbb{N}^r}{\sum}{t_1^{\beta_1}\cdots t_r^{\beta_r}}(a_{1,\beta}v_1+\cdots+a_{t,\beta}v_t)=0$ then $a_{1,\beta}v_1+\cdots+a_{t,\beta}v_t=0$ for all $\beta$. Since $\widehat{U(\mathcal{L})}_K\otimes_{\widehat{U(\mathcal{H})}_K}M$ is finitely generated, it is complete by Proposition \ref{fin-gen}, thus 

\begin{equation}
\begin{split}
{\underset{\beta\in\mathbb{N}^r}{\sum}{x_1^{\beta_1}\cdots x_r^{\beta_r}a_{1,\beta}}\otimes v_1+\cdots+\underset{\beta\in\mathbb{N}^r}{\sum}{x_1^{\beta_1}\cdots x_r^{\beta_r}a_{t,\beta}}\otimes v_t} \\ {=\underset{\beta\in\mathbb{N}^r}{\sum}{x_1^{\beta_1}\cdots x_r^{\beta_r}\otimes(a_{1,\beta}v_1+\cdots+a_{t,\beta}v_t})=0,}
\end{split}
\end{equation}

\noindent hence $\Theta$ is injective.\\

\noindent Hence $\Theta$ is an isomorphism of $K$-vector spaces, and $\Theta(x_i\otimes v)=t_iv$ for all $i\leq r$, $v\in M$. So clearly we can define an action of $\widehat{U(\mathcal{L})}_K$ on $M\langle t_1,\cdots,t_r\rangle$ making $\Theta$ into an isomorphism of $\widehat{U(\mathcal{L})}_K$-modules.\\

\noindent $(i)$ Since $\Theta(x^nv)=t^nv$ for all $v\in M$, $n\in\mathbb{N}$, it is clear that the action of $x$ on $M\langle t\rangle$ is given by multiplication by $t$.\\

\noindent $(ii)$ Since $y$ and $z$ act by scalars on $M$, their action on $M\langle t\rangle$ is determined entirely by their action on the powers of $t$.\\

\noindent Since $[x,z]=0$, it follows that $z$ commutes with all powers of $x$, and hence $z\cdot t^nv=z\Theta(x^n\otimes v)=\Theta(x^n\otimes z\cdot v)=\Theta(\beta_z(x^n\otimes v))=\beta_zt^nv$, so $z$ acts on $M\langle t\rangle$ via $\beta_z$.\\

\noindent Clearly $y\cdot v=\beta_y v$, so we will assume that for some $n\geq 0$, $y\cdot t^n v=n\alpha\beta_zt^{n-1}v+\beta_yt^nv$ and show that $y\cdot t^{n+1}v=(n+1)\alpha\beta_zt^nv+\beta_yt^{n+1}v$, and it will follow using induction that $y$ acts by $\alpha\beta_z\frac{d}{dt}+\beta_y$:

\begin{equation}
\begin{split}
y\cdot t^{n+1}v & =y\Theta(x^{n+1}\otimes v)=\Theta(yx^{n+1}\otimes v)=\Theta(([y,x]x^n+xyx^n)\otimes v)\\
                & =\alpha z\Theta(x^n\otimes v)+xy\Theta(x^n\otimes v)=\alpha z t^nv+xyt^nv\\
                & =\alpha\beta_z t^nv+xn\alpha\beta_zt^{n-1}v+x\beta_yt^nv=(n+1)\alpha\beta_z t^nv+\beta_y t^{n+1}v
\end{split}
\end{equation}

\noindent $(iii)$ Let $\partial:=\frac{d}{dt}$, and let $\rho:\widehat{U(\mathcal{L})}_K\to $ End$_K(M\langle t\rangle)$ be the action, then by part $(ii)$, $\rho(y)=\alpha\beta_z\partial+\beta_y$, so $\partial=(\alpha\beta_z)^{-1}(\rho(y)-\beta_y)=\rho((\alpha\beta_z)^{-1}(y-\beta_y))\in \im(\rho)$.\\ 

\noindent Hence for each $n\in\mathbb{N}$, $\partial^{[n]}=\frac{1}{n!}\partial^n\in \im(\rho)$.\\

\noindent So, suppose that $0\neq T\leq M\langle t\rangle$ is a submodule, i.e there exists $\underset{m\geq 0}{\sum}{t^ms_m}\in T$, $s_m\in M$, $s_m$ not all zero, $s_m\rightarrow 0$ as $m\rightarrow\infty$.\\

\noindent Then since $\partial^{[n]}(T)\subseteq T$ for all $n$, it follows that $\underset{m\geq n}{\sum}{\binom{m}{n}t^{m-n}s_m}\in T$, hence we may assume that $s_0\neq 0$.\\

\noindent Set $s:=s_0\in M\backslash\{0\}$, and define a sequence of elements in $T$ by $r_0:=\underset{m\geq 0}{\sum}{t^ms_m}$, and for $i>0$, $r_i:=r_{i-1}-t^i\partial^{[i]}(r_{i-1})$.\\ 

\noindent Now, if $r_i=s+\underset{m>i}{\sum}{t^ms_{i,m}}$, then $t^i\partial^{[i]}(r_i)=\underset{m>i}{\sum}{t^m\binom{m}{i}s_{i,m}}$, so 

\begin{center}
$r_{i+1}=r_i-t^i\partial^{[i]}(r_i)=s+\underset{m>i+1}{\sum}{t^m(s_{i,m}-\binom{m}{i}s_{i,m})}$.
\end{center}

\noindent So inductively, we get that for each $i\in\mathbb{N}$, $r_i=s+\underset{m>i}{\sum}{s_{i,m}t^m}$ for some $s_{i,m}\in M$ with $v(s_{i,m})\geq v(s_{i-1,m})$. It follows easily that $r_i\rightarrow s$ in $M\langle t\rangle$ as $i\rightarrow\infty$.\\

\noindent But since $M\langle t\rangle$ is finitely generated over $\widehat{U(\mathcal{L})}_K$, which is Noetherian, it follows that $T$ is finitely generated over $\widehat{U(\mathcal{L})}_K$, and hence $T$ is $\pi$-adically complete by Proposition \ref{fin-gen}. Therefore, since each $r_i\in T$, this means that $s\in T$. So $0\neq s\in M\cap T$, and thus $M\cap T\neq 0$. But since $M$ is irreducible, $M\cap T=0$ or $M$, hence $M\cap T=M$.\\

\noindent It follows that $M\langle t\rangle=\widehat{U(\mathcal{L})}_K\otimes_{\widehat{U(\mathcal{L'})}_K}M\subseteq T$, and hence $T=M\langle t\rangle$. Since our choice of $T$ was arbitrary, this implies that $M\langle t\rangle$ is irreducible as required.\end{proof}

\subsection{An explicit formula}

\noindent From now on, we will assume that $\mathfrak{g}$ is nilpotent. We will now examine the action of $\widehat{U(\mathcal{L})}_K$ on $\widehat{D(\lambda)}$ more closely, and write down an explicit formula for the action of elements of $\mathfrak{g}$.\\

\noindent First, let $\lambda:\mathfrak{g}\to K$ be a linear form such that $\lambda(\mathcal{L})\subseteq\mathcal{O}$, and let $\mathfrak{b}$ be a polarisation of $\mathfrak{g}$ at $\lambda$ with $\mathcal{B}=\mathfrak{b}\cap\mathcal{L}$. Fix a basis $\{u_1,\cdots,u_r\}$ for $\mathcal{L}/\mathcal{B}$, and it follows from Proposition \ref{Tate-action} that $\widehat{D(\lambda)}_{\mathfrak{b}}\cong K\langle u_1,\cdots,u_r\rangle$ as a $K$-vector space.\\

\noindent So, for each $i=1,\cdots,r$, recall from Section 2.4 that we can define the endomorphisms $x_i,\partial_i\in$ End$_K \widehat{D(\lambda)}$ by $x_i(f):=u_if$ and $\partial_i(f):=\frac{df}{du_i}$. Note that for each $u\in\mathfrak{g}$, we can write $u=v_{u}+\alpha_{1,u}u_1+\cdots+\alpha_{r,u}u_r$ for some unique $v_u\in\mathfrak{b}$, $\alpha_{i,u}\in K$. Using this, we define:

\begin{equation}
\mu:\mathfrak{g}\to\text{ End}_K\widehat{D(\lambda)}_{\mathfrak{b}},u\mapsto\lambda(v_u)+\alpha_{1,u}x_1+\cdots+\alpha_{r,u}x_r.
\end{equation}

\noindent Naively, one might think that this defines the action of $\mathfrak{g}$ on $\widehat{D(\lambda)}$, but this is not true. However, the following result gives the explicit formula which allows us to completely describe this action:

\begin{proposition}\label{explicit-formula}

Let $\rho:\widehat{U(\mathcal{L})}_K\to$\emph{ End}$_K\widehat{D(\lambda)}_{\mathfrak{b}}$ be the natural action of $\widehat{U(\mathcal{L})}_K$ on $\widehat{D(\lambda)}_{\mathfrak{b}}$, then we may choose a basis $\{u_1,\cdots,u_r\}$ for $\mathcal{L}/\mathcal{B}$ such that for every $u\in\mathfrak{g}$, the action of $u$ is given by:

\begin{center}
$\rho(u)=\underset{\alpha\in\mathbb{N}^r}{\sum}{\frac{1}{\alpha_1!}\cdots\frac{1}{\alpha_r!}\mu(\ad(u_r)^{\alpha_r}\cdots\ad(u_1)^{\alpha_1}(u))\partial_1^{\alpha_1}\cdots\partial_r^{\alpha_r}}$
\end{center}

\noindent\textbf{Note:} This is a finite sum, since $\mathfrak{g}$ is nilpotent. Also, we define $\ad(x)(y):=[y,x]$, as opposed to the more conventional $\ad(x)(y)=[x,y]$.\\

\noindent Moreover, if $u\in\mathfrak{b}$ and $\ad(\mathfrak{g})^n(u)\subseteq\mathfrak{b}$ for all $n\in\mathbb{N}$, then this formula holds for any choice of basis.

\end{proposition}

\begin{proof}

Since $\mathfrak{g}$ is nilpotent, we can choose a basis $\{u_1,\cdots,u_r\}$ for $\mathcal{L}/\mathcal{B}$ such that $[u_i,\mathfrak{g}]\subseteq\mathfrak{b}\oplus$ Span$_K\{u_i+1,\cdots,u_r\}$ for all $i$, and we will fix such  basis throughout the proof.\\

So, define $\rho':\mathfrak{g}\to$ End$_K\widehat{D(\lambda)},u\mapsto \underset{\alpha\in\mathbb{N}^r}{\sum}{\frac{1}{\alpha_1!}\cdots\frac{1}{\alpha_r!}\mu(\ad(u_r)^{\alpha_r}\cdots\ad(u_1)^{\alpha_1}(u))\partial_1^{\alpha_1}\cdots\partial_r^{\alpha_r}}$. We want to prove that $\rho(u)=\rho'(u)$ for all $u\in\mathfrak{g}$, i.e. that $\rho(u)(f)=\rho'(u)(f)$ for all $f\in\widehat{D(\lambda)}$.\\

\noindent Firstly, since by Lemma \ref{Tate-action}, every element of $\widehat{D(\lambda)}$ can be written as a Tate power series in the variables $u_1^{\beta_1}\cdots u_r^{\beta_r}$, it suffices to prove that $\rho(u)(u_1^{\beta_1}\cdots u_r^{\beta_r})=\rho'(u)(u_1^{\beta_1}\cdots u_r^{\beta_r})$ for every $\beta_1,\cdots,\beta_d\in\mathbb{N}$.\\

\noindent Note that $\rho'(u)(u_1^{\beta_1}\cdots u_r^{\beta_r})=\underset{\alpha\leq\beta}{\sum}{\binom{\beta}{\alpha}\mu(\ad(u_r)^{\alpha_r}\cdots\ad(u_1)^{\alpha_1}(u))u_1^{\beta_1-\alpha_1}\cdots u_r^{\beta_r-\alpha_r}}$, where $\alpha\leq\beta$ means that $\alpha_i\leq\beta_i$ for all $i$, and $\binom{\beta}{\alpha}:=\binom{\beta_1}{\alpha_1}\cdots\binom{\beta_r}{\alpha_r}$. So, we will prove by induction on $\vert\beta\vert:=\beta_1+\cdots+\beta_r$ that $\rho(u_1^{\beta_1}\cdots u_r^{\beta_r})=\underset{\alpha\leq\beta}{\sum}{\binom{\beta}{\alpha}\mu(\ad(u_r)^{\alpha_r}\cdots\ad(u_1)^{\alpha_1}(u))u_1^{\beta_1-\alpha_1}\cdots u_r^{\beta_r-\alpha_r}}$.\\

\noindent First, suppose that $\beta=0$, so $u_1^{\beta_1}\cdots u_r^{\beta_r}=1$, and hence $\rho'(u)(u_1^{\beta_1}\cdots u_r^{\beta_r})=\mu(u)(1)$.\\ 

\noindent Note that under the isomorphism $\widehat{D(\lambda)}\to K\langle u_1,\cdots,u_r\rangle$ given by Lemma \ref{Tate-action}, the inverse image of 1 is $1\otimes v$, where $v$ generates the one dimensional subspace $K_{\lambda}$ of $\widehat{D(\lambda)}=\widehat{U(\mathcal{L})}_K\otimes_{\widehat{U(\mathcal{B})}_K}K_{\lambda}$. Write $u=v_u+\alpha_{1,u}u_1+\cdots+\alpha_{r,u}u_r$ for $v_u\in\mathfrak{b}$ and $\alpha_{i,u}\in K$, so $\mu(u)=\lambda(v_u)+\alpha_{1,u}x_1+\cdots+\alpha_{r,u}x_r$ by definition. Then: \\

\noindent $\rho(u)(1)=\rho(u)(1\otimes v)=u\otimes v=v_u\otimes v+\alpha_{1,u}u_1\otimes v+\cdots +\alpha_{r,u}u_r\otimes v$\\

$=\lambda(v_u)\otimes v+\alpha_{1,u}u_1\otimes v+\cdots +\alpha_{r,u}u_r\otimes v$,\\

\noindent and the image of this under the isomorphism $D(\lambda)\to K\langle u_1,\cdots,u_r\rangle$ is $\lambda(v_u)+\alpha_{1,u}u_1+\cdots+\alpha_{r,u}u_r$ by Lemma \ref{Tate-action}, and clearly this is equal to $\mu(u)(1)=\rho'(u)(1)$, so $\rho(u)(1)=\rho'(u)(1)$ as required.\\

\noindent So, now assume that $n>0$ and the result holds whenever $\vert\beta\vert<n$. Choose $\gamma$ with $\vert\gamma\vert=n$, and choose $i\geq 1$ minimal such that $\gamma_i\neq 0$, so $u_1^{\gamma_1}\cdots u_r^{\gamma_r}=u_i^{\gamma_i}\cdots u_r^{\gamma_r}$.  Let $\beta_i:=\gamma_i-1$, $\beta_j=\gamma_j$ for all $j>i$, so that $u_i^{\gamma_i}\cdots u_r^{\gamma_r}=u_i^{\beta_i+1}u_{i+1}^{\beta_{i+1}}\cdots u_r^{\beta_r}$.\\

\noindent Then $\rho(u)(u_i^{\gamma_i}\cdots u_r^{\gamma_r})=\rho(u)(u_i^{\beta_i+1}\cdots u_r^{\beta_r}\otimes v)=uu_i^{\beta_i+1}\cdots u_r^{\beta_r}\otimes v=[u,u_i]u_i^{\beta_i}\cdots u_r^{\beta_r}\otimes v+u_iuu_i^{\beta_i}\cdots u_r^{\beta_r}\otimes v=\rho([u,u_i])u_i^{\beta_i}\cdots u_r^{\beta_r}+\rho(u_i)\rho(u)u_i^{\beta_i}\cdots u_r^{\beta_r}$.\\

\noindent Since $\vert\beta\vert=\vert\gamma\vert-1<n$, applying induction gives that \\

$\rho([u,u_i])u_i^{\beta_i}\cdots u_r^{\beta_r}=\underset{\alpha\leq\beta}{\sum}{\binom{\beta}{\alpha}\mu(\ad(u_r)^{\alpha_r}\cdots\ad(u_i)^{\alpha_i}([u,u_i]))u_i^{\beta_i-\alpha_i}\cdots u_r^{\beta_r-\alpha_r}}$\\

$=\underset{\alpha\leq\beta}{\sum}{\binom{\beta}{\alpha}\mu(\ad(u_r)^{\alpha_r}\cdots\ad(u_i)^{\alpha_i+1}(u))u_i^{\beta_i-\alpha_i}\cdots u_r^{\beta_r-\alpha_r}}$\\

$=\underset{\alpha_j\leq\beta_j}{\sum}{\binom{\beta_i}{\alpha_i}\cdots\binom{\beta_r}{\alpha_r}\mu(\ad(u_r)^{\alpha_r}\cdots\ad(u_i)^{\alpha_i+1}(u))u_i^{\beta_i-\alpha_i}\cdots u_r^{\beta_r-\alpha_r}}$.\\

\noindent and

\begin{center}
$\rho(u_i)\rho(u)u_i^{\beta_i}\cdots u_r^{\beta_r}=\rho(u_i)\underset{\alpha\leq\beta}{\sum}{\binom{\beta}{\alpha}\mu(\ad(u_r)^{\alpha_r}\cdots\ad(u_i)^{\alpha_i}(u))u_i^{\beta_i-\alpha_i}\cdots u_r^{\beta_r-\alpha_r}}$.
\end{center}

\noindent By our choice of basis, $\ad(u_r)^{\alpha_r}\cdots\ad(u_i)^{\alpha_i}(u)\in\mathfrak{b}\oplus$ Span$_K\{u_{i+1},\cdots,u_r\}$ for all $\alpha$, and thus $\mu(\ad(u_r)^{\alpha_r}\cdots\ad(u_i)^{\alpha_i}(u))\in K\oplus$ Span$_K\{x_r,\cdots,x_{i+1}\}$.\\ 

\noindent This means that for all $\alpha$, $\mu(\ad(u_r)^{\alpha_r}\cdots\ad(u_i)^{\alpha_i}(u))u_i^{\beta_i-\alpha_i}\cdots u_r^{\beta_r-\alpha_r}$ is linear combination of monomials of the form $u_i^{\delta_i}\cdots u_r^{\delta_r}$. 

So since $\rho(u_i)(u_i^{\delta_i}\cdots u_r^{\delta_r})=u_i^{\delta_i+1}u_{i+1}^{\delta_{i+1}}\cdots u_r^{\delta_r}=x_i (u_i^{\delta_i}\cdots u_r^{\delta_r})$ for all $\delta$, and $x_i$ commutes with $x_r,\cdots,x_{i+1}$, it follows that:\\ 

$\rho(u_i)\underset{\alpha\leq\beta}{\sum}{\binom{\beta}{\alpha}\mu(\ad(u_r)^{\alpha_r}\cdots\ad(u_i)^{\alpha_i}(u))u_i^{\beta_i-\alpha_i}\cdots u_r^{\beta_r-\alpha_r}}$\\

$=\underset{\alpha\leq\beta}{\sum}{\binom{\beta}{\alpha}\mu(\ad(u_r)^{\alpha_r}\cdots\ad(u_i)^{\alpha_i}(u))u_i^{\beta_i-\alpha_i+1}\cdots u_r^{\beta_r-\alpha_r}}$\\

$=\underset{\alpha_j\leq\beta_j}{\sum}{\binom{\beta_i}{\alpha_i}\cdots\binom{\beta_r}{\alpha_r}\mu(\ad(u_r)^{\alpha_r}\cdots\ad(u_i)^{\alpha_i}(u))u_i^{\beta_i-\alpha_i+1}\cdots u_r^{\beta_r-\alpha_r}}$.\\

\noindent Note that this is the only point in the proof where we use our particular choice of basis. Whereas if we assume that $\ad(\mathfrak{g})^n(u)\in\mathfrak{b}$ for all $n$ then, $\mu(\ad(u_r)^{\alpha_r}\cdots\ad(u_i)^{\alpha_i}(u))\in K$ for all $\alpha$, and clearly this commutes with $\rho(u_i)$ regardless of the choice of basis.\\

\noindent Now, collecting terms gives us:\\

\noindent $\rho(u)(u_i^{\beta_i+1}\cdots u_r^{\beta_r})=\underset{\alpha_j\leq\beta_j}{\sum}{\binom{\beta_i}{\alpha_i}\cdots\binom{\beta_r}{\alpha_r}\mu(\ad(u_r)^{\alpha_r}\cdots\ad(u_i)^{\alpha_i+1}(u))u_i^{\beta_i-\alpha_i}\cdots u_r^{\beta_r-\alpha_r}}$

\noindent $+\underset{\alpha_j\leq\beta_j}{\sum}{\binom{\beta_i}{\alpha_i}\cdots\binom{\beta_r}{\alpha_r}\mu(\ad(u_r)^{\alpha_r}\cdots\ad(u_i)^{\alpha_i}(u))u_i^{\beta_i-\alpha_i+1}\cdots u_r^{\beta_r-\alpha_r}}$\\

\noindent $=u_i^{\beta_i+1}\cdots u_r^{\beta_r}+\underset{1\leq \alpha_i\leq\beta_i}{\sum}{\left(\binom{\beta_i}{\alpha_i}+\binom{\beta_i}{\alpha_i-1}\right)\binom{\beta_{i+1}}{\alpha_{i+1}}\cdots\binom{\beta_r}{\alpha_r}\mu(\ad(u_r)^{\alpha_r}\cdots\ad(u_i)^{\alpha_i}(u))u_i^{\beta_i+1-\alpha_i}\cdots u_r^{\beta_r-\alpha_r}}$\\

\noindent $+\underset{\alpha_i=\beta_i}{\sum}{\binom{\beta_{i+1}}{\alpha_{i+1}}\cdots\binom{\beta_r}{\alpha_r}\mu(\ad(u_r)^{\alpha_r}\cdots\ad(u_i)^{\beta_i+1}(u))u_{i+1}^{\beta_{i+1}-\alpha_{i+1}}\cdots u_r^{\beta_r-\alpha_r}}$\\

\noindent $=\underset{\alpha\leq\gamma}{\sum}{\binom{\gamma}{\alpha}\mu(\ad(u_r)^{\alpha_r}\cdots\ad(u_i)^{\alpha_i}(u))u_i^{\gamma_i-\alpha_i}\cdots u_r^{\gamma_r-\alpha_r}}$\\

\noindent The last equality follows since $\binom{\beta_i}{\alpha_i}+\binom{\beta_i}{\alpha_i-1}=\binom{\beta_i+1}{\alpha_i}=\binom{\gamma_i}{\alpha_i}$.\end{proof}

\noindent\textbf{Note:} 1. This proof is purely classical, and defines a formula for the action of $\mathfrak{g}$ on the classical Dixmier module $D(\lambda)$. In fact, this result gives us a ring homomorphism $U(\mathfrak{g})\to A_r(K)$ whose kernel is $\Ann_{U(\mathfrak{g})}D(\lambda)$. We suspect this map is surjective in general, since we know from \cite[Proposition 6.2.2]{Dixmier} that $U(\mathfrak{g})/\Ann D(\lambda)\cong A_r(K)$.\\

\noindent 2. If $u\in\mathfrak{b}$ and $\ad(\mathfrak{g})^n(u)\subseteq\mathfrak{b}$ for all $n\in\mathbb{N}$, i.e. $u\in\mathfrak{a}$ for some ideal $\mathfrak{a}$ of $\mathfrak{g}$ with $\mathfrak{a}\subseteq\mathfrak{b}$, then $\mu(\ad(u_r)^{\alpha_r}\cdots\ad(u_1)^{\alpha_1}(u))=\lambda(\ad(u_r)^{\alpha_r}\cdots\ad(u_1)^{\alpha_1}(u))$ for all $\alpha\in\mathbb{N}^r$, so it follows from this proposition that $u$ acts on $\widehat{D(\lambda)}$ by a polynomial in $K[\partial_1,\cdots,\partial_r]$.

\begin{corollary}\label{completely-prime}

Assume that $\mathcal{L}$ is a \emph{powerful} Lie lattice in $\mathfrak{g}$, i.e. $[\mathcal{L},\mathcal{L}]\subseteq p\mathcal{L}$. Then for any linear form $\lambda:\mathfrak{g}\to K$ with $\lambda(\mathcal{L})\subseteq\mathcal{O}$, if $\mathfrak{g}$ has a polarisation $\mathfrak{b}$ at $\lambda$ of codimension $r$, and $I:=Ann_{\widehat{U(\mathcal{L})}_K}\widehat{D(\lambda)}_{\mathfrak{b}}$, then there exists an injective ring homomorphism $\widehat{U(\mathcal{L})}_K/I\to\widehat{A_r(K)}$, and thus $\widehat{U(\mathcal{L})}_K/I$ is a domain.

\end{corollary}

\begin{proof}

The natural action $\rho:\widehat{U(\mathcal{L})}_K\to$ End$_K\widehat{D(\lambda)}$ has kernel $I$, and since $\widehat{D(\lambda)}\cong K\langle u_1,\cdots,u_r\rangle$ by Lemma \ref{Tate-action}, the Tate-Weyl algebra $\widehat{A_r(K)}$ embeds as a subalgebra into End$_K\widehat{D(\lambda)}$ by Lemma \ref{aff-endo}. Therefore, it remains to prove that the image of $\rho$ is contained in $\widehat{A_r(K)}$.\\

\noindent Using Proposition \ref{explicit-formula}, we know that we can fix a basis $\{u_1,\cdots,u_r\}$ for $\mathcal{L}/\mathcal{B}$ such that for each $u\in\mathfrak{g}$, $\rho(u)=\underset{\alpha\in\mathbb{N}^r}{\sum}{\frac{1}{\alpha_1!}\cdots\frac{1}{\alpha_r!}\mu(\ad(u_r)^{\alpha_r}\cdots\ad(u_1)^{\alpha_1}(u))\partial_1^{\alpha_1}\cdots\partial_r^{\alpha_r}}$.\\ 

\noindent Since $\mu(\mathfrak{g})\subseteq K\oplus Kx_1\oplus\cdots\oplus Kx_r\subseteq A_r(K)$ and $\partial_1,\cdots,\partial_r\in A_r(K)$, it follows that $\rho(\mathfrak{g})\subseteq A_r(K)$, and hence $\rho(U(\mathfrak{g}))\subseteq A_r(K)$.

Therefore, it remains to prove that $\rho:U(\mathfrak{g})\to A_r(K)$ is continuous with respect to the $\pi$-adic topologies on $U(\mathfrak{g})$ and $A_r(K)$ respectively. And since $\rho$ is linear, this just means proving that the image of $U(\mathcal{L})$ under $\rho$ is contained in $A_r(\mathcal{O})=\mathcal{O}[x_1,\cdots,x_r,\partial_1,\cdots,\partial_r]$.\\

\noindent Since $[\mathcal{L},\mathcal{L}]\subseteq p\mathcal{L}$, it follows that for all $u\in\mathcal{L}$, $\alpha\in\mathbb{N}^r$, $\ad(u_r)^{\alpha_r}\cdots\ad(u_1)^{\alpha_1}(u)\in p^{\vert\alpha\vert}\mathcal{L}$. And since $\mu:\mathcal{L}\to\mathcal{O}\oplus\mathcal{O}x_1\oplus\cdots\mathcal{O}x_r$ is linear, $\mu(\ad(u_r)^{\alpha_r}\cdots\ad(u_1)^{\alpha_1}(u))\in p^{\vert\alpha\vert}A_r(\mathcal{O})$.\\

\noindent Therefore, since $\frac{1}{\alpha_1!}\cdots\frac{1}{\alpha_r!}p^{\vert\alpha\vert}=\frac{p^{\alpha_1}}{\alpha_1!}\cdots\frac{p^{\alpha_r}}{\alpha_r!}\in\mathcal{O}$, it follows that: 

\begin{center}
$\rho(u)=\underset{\alpha\in\mathbb{N}^r}{\sum}{\frac{1}{\alpha_1!}\cdots\frac{1}{\alpha_r!}\mu(\ad(u_r)^{\alpha_r}\cdots\ad(u_1)^{\alpha_1}(u))\partial_1^{\alpha_1}\cdots\partial_r^{\alpha_r}}\in A_r(\mathcal{O})$.
\end{center}

\noindent Therefore, the image of $\mathcal{L}$ under $\rho$ is contained in $A_r(\mathcal{O})$, and therefore so is the image of $U(\mathcal{L})$ as required.\end{proof}

\noindent Now, fix an ideal $\mathfrak{a}$ of $\mathfrak{g}$ with $\mathfrak{a}\subseteq\mathfrak{b}$, and define 

\begin{center}
$\mathfrak{a}^{\perp}:=\{u\in\mathfrak{g}:\lambda([u,\mathfrak{a}])=0\}$,
\end{center} 

\noindent then $\mathfrak{a}^{\perp}$ is a subalgebra of $\mathfrak{g}$ with $\mathfrak{b}\subseteq\mathfrak{a}^{\perp}$, so we can fix a basis $\{u_1,\cdots,u_r\}$ for $\mathcal{L}/\mathcal{B}$ such that $\{u_{s+1},\cdots,u_r\}$ is a basis for $(\mathfrak{a}^{\perp}\cap\mathcal{L})/\mathcal{B}$ for some $s\leq r$.

\begin{proposition}\label{perp}

Every element of $\mathfrak{a}$ acts on $\widehat{D(\lambda)}$ by a polynomial in $K[\partial_1,\cdots,\partial_s]$, and the image of $U(\mathfrak{a})$ under the action $\rho:\widehat{U(\mathcal{L})}_K\to$\emph{ End}$_K\widehat{D(\lambda)}$ is precisely $K[\partial_1,\cdots,\partial_s]$.

\end{proposition}

\begin{proof}

Firstly, using Proposition \ref{explicit-formula}, we know that for every $u\in\mathfrak{a}$, 

\begin{center}
$\rho(u)=\underset{\alpha\in\mathbb{N}^r}{\sum}{\frac{1}{\alpha_1!}\cdots\frac{1}{\alpha_r!}\lambda(\ad(u_r)^{\alpha_r}\cdots\ad(u_1)^{\alpha_1}(u))\partial_1^{\alpha_1}\cdots\partial_r^{\alpha_r}}\in K[\partial_1,\cdots,\partial_r]$. 
\end{center}

\noindent Also, if $\alpha_i\neq 0$ for any $r \geq i>s$, then assuming $i$ is maximal such that $\alpha_i\neq 0$,  $\lambda(\ad(u_r)^{\alpha_r}\cdots\ad(u_1)^{\alpha_1}(u))=\lambda(\ad(u_i)^{\alpha_i}\cdots\ad(u_1)^{\alpha_1}(u))=0$ since $u_i\in\mathfrak{a}^{\perp}$.

Therefore, $\rho(u)=\underset{\alpha\in\mathbb{N}^s}{\sum}{\frac{1}{\alpha_1!}\cdots\frac{1}{\alpha_s!}\lambda(\ad(u_s)^{\alpha_s}\cdots\ad(u_1)^{\alpha_1}(u))\partial_1^{\alpha_1}\cdots\partial_s^{\alpha_s}}\in K[\partial_1,\cdots,\partial_s]$ as required.\\

\noindent For the second statement, clearly $\rho(U(\mathfrak{a}))\subseteq K[\partial_1,\cdots,\partial_s]$, so we just need to show that $\partial_1,\cdots,\partial_s\in\rho(U(\mathfrak{a}))$.\\ 

\noindent We will first need to construct our basis appropriately. For convenience, set $\mathcal{A}=\mathfrak{a}\cap\mathcal{L}$ and $\mathcal{A}^{\perp}=\mathfrak{a}^{\perp}\cap\mathcal{L}$. First, write $\mathcal{L}/\mathcal{B}=(\mathcal{A}^{\perp}/\mathcal{B})\oplus V$, for some complement $V$ of $\mathcal{A}^{\perp}/\mathcal{B}$ in $\mathcal{L}/\mathcal{B}$.\\

\noindent We define the \emph{upper central series} by $Z_1(\mathcal{L}):=Z(\mathcal{L})$ and $Z_i(\mathcal{L}):=\{u\in\mathcal{L}:[u,\mathcal{L}]\subseteq Z_{i-1}(\mathcal{L})\}$ for all $i>1$. Since $\mathcal{L}$ is nilpotent, $Z_c(\mathcal{L})=\mathcal{L}$ for some $c\geq 1$. Therefore, since $V\cap\mathfrak{a}^{\perp}=0$, and hence $\lambda([\mathfrak{a},V])\neq 0$, we can choose $m_1>1$ minimal such that $\lambda([\mathcal{A}\cap Z_{m_1}(\mathcal{L}),V])\neq 0$.\\

\noindent So, choose $y_1\in\mathcal{A}\cap Z_{m_1}(\mathcal{L})$ such that $\lambda([y_1,V])\neq 0$. Then $\lambda\circ\ad(y_1):V\otimes K\to K$ is a non-zero linear form, so we may write $V=\mathcal{O}u_1\oplus V_1$ where $\lambda([y_1,V_1])=0$ and $\lambda([y_1,u_1])\neq 0$.\\

\noindent Now, let us suppose for induction, that for some $i<s=$ dim$_K\mathfrak{g}/\mathfrak{a}^{\perp}$, we have the following data: 

\begin{itemize}

\item Integers $1<m_1\leq m_2\leq\cdots\leq m_i\leq c$. 

\item Subspaces $V_i\subseteq V_{i-1}\subseteq\cdots\subseteq V_1\subseteq V_0=V$ with $\lambda([V_{j-1},\mathfrak{a}\cap Z_{m_j-1}(\mathcal{L})])=0$ for each $j\geq 1$. 

\item Elements $u_j\in V_{j-1}$ such that $V_{j-1}=\mathcal{O}u_j\oplus V_j$ for each $j\geq 1$. 

\item Elements $y_j\in\mathcal{A}\cap Z_{m_j}(\mathcal{L})$ such that $\lambda([y_j,V_j])=0$ and $\lambda([y_j,u_j])\neq 0$.

\end{itemize}

\noindent Note that for each $j$, dim$_K(V_j\otimes K)=s-j$. So since $i<s$, $V_i\cap\mathfrak{a}^{\perp}=0$,  so $\lambda([V_i,\mathfrak{a}])\neq 0$.\\ 

\noindent But since $V_i\subseteq V_{i-1}$, we know that $[V_i,\mathfrak{a}\cap Z_{m_i-1}(\mathcal{L})]=0$, so choose $m_{i+1}\geq m_i$ minimal such that $\lambda([V_i,\mathfrak{a}\cap Z_{m_{i+1}}(\mathcal{L})])\neq 0$, and choose $y_{i+1}\in\mathcal{A}\cap Z_{m_{i+1}}(\mathcal{L})$ such that $\lambda([y_{i+1},V_i])\neq 0$.\\

\noindent Again, $\lambda\circ\ad(y_{i+1}):V_i\otimes K\to K$ is a non-zero linear form, so $V_i=\mathcal{O}u_{i+1}\oplus V_{i+1}$ for some $u_{i+1}\in V_i$, where $\lambda([y_{i+1},V_{i+1}])=0$ and $\lambda([y_{i+1},u_{i+1}])\neq 0$.\\

\noindent Therefore, applying induction gives us a basis $\{u_1,\cdots,u_s\}$ for $V$ such that for each $i$, $\lambda([u_i,\mathcal{A}\cap Z_{m_i-1}(\mathcal{L})])=0$, and elements $y_1,\cdots,y_s\in\mathcal{A}$ with $y_i\in Z_{m_i}(\mathcal{L})$ such that $\lambda([y_i,u_j])=0$ for all $j<i$ and $\lambda([y_i,u_i])\neq 0$.\\

\noindent So, applying our explicit formula with the basis $\{u_1,\cdots,u_s\}$, we get that $\rho(y_i)=\underset{\alpha\in\mathbb{N}^s}{\sum}{\frac{1}{\alpha_1!}\cdots\frac{1}{\alpha_s!}\lambda(\ad(u_s)^{\alpha_s}\cdots\ad(u_1)^{\alpha_1}(y_i))\partial_1^{\alpha_1}\cdots\partial_s^{\alpha_s}}$ for each $i$. Let us suppose that for all $0\leq j<i$, $\partial_j\in\rho(U(\mathfrak{g}))$.\\

\noindent Then since $[\mathcal{L},y_{i}]\subseteq Z_{m_{i}-1}(\mathcal{L})$ and $\lambda([u_j, Z_{m_{i}-1}(\mathcal{L})])=0$ for all $j\geq i$, it follows that if $\lambda(\ad(u_s)^{\alpha_s}\cdots\ad(u_1)^{\alpha_1}(y_{i+1}))\neq 0$ then either $\alpha_j=0$ for all $j\geq i$, or $\alpha_j=0$ for all $j<i$ and $\alpha_{i}+\cdots+\alpha_s=1$. Therefore: 

\begin{center}
$\rho(y_{i})=f(\partial_1,\cdots,\partial_{i-1})+\lambda([y_{i},u_{i}])\partial_{i}+\cdots+\lambda([y_{i},u_{s}])\partial_s$ for some polynomial $f$.
\end{center} 

\noindent But since $\partial_1,\cdots,\partial_{i-1}\in\rho(U(\mathfrak{a}))$, $\lambda([y_i,u_j])=0$ for all $j>i$, and $\lambda([y_{i},u_{i}])\neq 0$, it follows that $\partial_{i}\in\rho(U(\mathfrak{a}))$. So applying induction, we see that $\partial_1,\cdots,\partial_s\in\rho(U(\mathfrak{a}))$ as required.\\

\noindent Finally, if $\{u_1',\cdots,u_s'\}$ is any other basis for $V$, then since each $u_i'$ is a non-zero linear combination in $u_1,\cdots,u_s$, it follows from the chain rule that $\partial_{u_i'}=\beta_i\partial_{u_i}$ for some non-zero $\beta_i\in K$, and hence $\partial_{u_1'},\cdots,\partial_{u_s'}\in\rho(U(\mathfrak{a}))$.\end{proof}

\subsection{Irreducibility}

\noindent Now we will prove the main theorem of this section.

\begin{theorem}\label{aff-Dix1}

Suppose that $\mathfrak{g}$ is nilpotent, $\lambda\in\mathfrak{g}^*$ with $\lambda(\mathcal{L})\subseteq\mathcal{O}$. Then there exists a polarisation $\mathfrak{b}$ of $\mathfrak{g}$ at $\lambda$ such that the affinoid Dixmier module $\widehat{D(\lambda)}_{\mathfrak{b}}$ of $\mathcal{L}$ at $\lambda$ with respect to $\mathfrak{b}$ is irreducible as a $\widehat{U(\mathcal{L})}_K$-module.

\end{theorem}

\begin{proof}

We will use induction on $n=$ dim$_K\mathfrak{g}$, so first suppose that $n=1$. Then $\mathfrak{g}$ is abelian, so $\lambda$ is a character of $\mathfrak{g}$, and $\widehat{D(\lambda)}=K$, which is clearly irreducible.\\

\noindent For the inductive step, we will assume that the result holds for all $m<n$:\\

\noindent Suppose first that there exists a non-zero ideal $\mathfrak{a}\trianglelefteq\mathfrak{g}$ such that $\lambda(\mathfrak{a})=0$. Let $\mathfrak{g}_1:=\mathfrak{g}/\mathfrak{a}$, so since $\mathfrak{a}\neq 0$, dim$_K\mathfrak{g}_1<n$, so we may apply the inductive hypothesis to $\mathfrak{g}_1$.\\

\noindent Let $\mathcal{A}:=\mathcal{L}\cap\mathfrak{a}$, and let $\mathcal{L}_1:=\frac{\mathcal{L}}{\mathcal{A}}$, then $\mathcal{A},\mathcal{L}_1$ are lattices in $\mathfrak{a}$, $\mathfrak{g}_1$ respectively.\\

\noindent Now, let $\lambda_1$ be the linear form of $\mathfrak{g}_1$ induced by $\lambda$, and clearly $\lambda_1(\mathcal{L}_1)\subseteq\mathcal{O}$. So by the inductive hypothesis, there exists a polarisation $\mathfrak{b}_1$ of $\mathfrak{g}_1$ at $\lambda_1$ such that $\widehat{D(\lambda_1)}_{\mathfrak{b}_1}$ is irreducible over $\widehat{U(\mathcal{L}_1)}_K$.\\

\noindent Since $\mathfrak{b}_1=\frac{\mathfrak{b}}{\mathfrak{a}}$ for some subalgebra $\mathfrak{b}$ of $\mathfrak{g}$, it follows that $\mathfrak{b}$ is a polarisation of $\mathfrak{g}$ at $\lambda$, so let $\mathcal{B}:=\mathfrak{b}\cap\mathcal{L}$, $\mathcal{B}_1:=\mathfrak{b}_1\cap\mathcal{L}_1$, and let $M:=\widehat{D(\lambda_1)}_{\mathfrak{b}_1}=\widehat{U(\mathcal{L}_1)}_K\otimes_{\widehat{U(\mathcal{B}_1)}_K}K_{\lambda_1}$.\\

Using the surjection $\widehat{U(\mathcal{L})}_K\twoheadrightarrow \widehat{U(\mathcal{L}_1)}_K$ given by Lemma \ref{induced}$(i)$, we see that $M$ has the structure of an irreducible $\widehat{U(\mathcal{L})}_K$-module, and using the fact that $\mathcal{B}_1=\frac{\mathcal{B}}{\mathcal{A}}$ and Lemma \ref{induced}($ii$) we see that $M\cong \widehat{U(\mathcal{L})}_K\otimes_{\widehat{U(\mathcal{B})}_K}K_{\lambda}=\widehat{D(\lambda)}_{\mathfrak{b}}$, and hence $\widehat{D(\lambda)}_{\mathfrak{b}}$ is irreducible as required.\\

\noindent So from now on, we may assume that $\lambda(\mathfrak{a})\neq 0$ for all non-zero ideals $\mathfrak{a}$ of $\mathfrak{g}$.\\ 

\noindent Using Proposition \ref{sub-polarisation}, we can find a reducing quadruple $(x,y,z,\mathfrak{g}')$ of $\mathfrak{g}$ such that any polarisation $\mathfrak{b}\subseteq\mathfrak{g}'$ of $\mathfrak{g}'$ at $\mu:=\lambda|_{\mathfrak{g}'}$ is in fact a polarisation of $\mathfrak{g}$ at $\lambda$. Note that since $y,z$ are central in $\mathfrak{g}'$, $y,z\in\mathfrak{b}$ by Lemma \ref{polarisation-properties}.\\

\noindent Set $\mathcal{L}':=\mathcal{L}\cap\mathfrak{g}'$. Then since dim$_K\mathfrak{g}'=n-1<n$, using the inductive hypothesis we can choose a polarisation $\mathfrak{b}$ of $\mathfrak{g}'$ at $\mu$ such that $\widehat{D(\mu)}=\widehat{U(\mathcal{L'})}_K\otimes_{\mathcal{B}} K_{\mu}$ is irreducible over $\widehat{U(\mathcal{L}')}_K$.\\

\noindent Now, since $\widehat{U(\mathcal{L})}_K=\widehat{U(\mathcal{L})}_K\otimes_{\widehat{U(\mathcal{L'})}_K}\widehat{U(\mathcal{L'})}_K$, it follows that $\widehat{D(\lambda)}_{\mathfrak{b}}=\widehat{U(\mathcal{L})}_K\otimes_{\widehat{U(\mathcal{L'})}_K}\widehat{D(\mu)}_{\mathfrak{b}}$.\\

\noindent Therefore, setting $M:=\widehat{D(\mu)}_{\mathfrak{b}}$, we have that $\widehat{D(\lambda)}_{\mathfrak{b}}\cong M\langle t\rangle$ as a $K$-vector space by Proposition \ref{Tate-action}$(i)$, where $x$ acts on $M\langle t\rangle$ by $t$.\\

\noindent Also, $y,z$ are central in $\mathfrak{g}'$ and $[y,x]=\alpha z$ for some $0\neq\alpha\in K$. Therefore, since $y$ and $z$ act on $M$ by scalars $\lambda(y)$, $\lambda(z)$ respectively, we see using Proposition \ref{Tate-action}$(ii)$ that $y$ acts on $M\langle t\rangle$ by $\alpha \lambda(z)\frac{d}{dt}+\lambda(y)$.\\

So, finally, since $Kz$ is an ideal of $\mathfrak{g}$, $\lambda(z)\neq 0$, so since $\alpha,\lambda(z)\neq 0$, it follows from Proposition \ref{Tate-action}$(iii)$ that $\widehat{D(\lambda)}_{\mathfrak{b}}=M\langle t\rangle$ is irreducible over $\widehat{U(\mathcal{L})}_K$ as required.\end{proof}

\section{Dixmier Annihilators}

From now on, we will always assume that $\mathfrak{g}$ is nilpotent. We are interested in the annihilators inside $\widehat{U(\mathcal{L})}_K$ of affinoid Dixmier modules. 

Since for any $\lambda\in\mathfrak{g}^*$ with $\lambda(\mathcal{L})\subseteq\mathcal{O}$, there exists a polarisation $\mathfrak{b}$ of $\mathfrak{g}$ at $\lambda$ such that $\widehat{D(\lambda)}_{\mathfrak{b}}$ is irreducible by Theorem \ref{aff-Dix1}, it follows by definition that Ann$_{\widehat{U(\mathcal{L})}_K}\widehat{D(\lambda)}_{\mathfrak{b}}$ is a primitive ideal of $\widehat{U(\mathcal{L})}_K$.\\

\noindent Ideally, we want to prove that all primitive ideals arise as annihilators of affinoid Dixmier modules. But in this section, we will first show that these Dixmier annihilators are always primitive, regardless of the choice of polarisation.

\subsection{Reducing Ideals}

\noindent First we need some preliminary results. The first is an affinoid version of \cite[Proposition 5.1.7]{Dixmier}:

\begin{lemma}\label{ind-annihilator}

Let $\mathfrak{h}\leq\mathfrak{g}$ be a subalgebra, let $\mathcal{H}:=\mathfrak{h}\cap\mathcal{L}$, and let $M$ be a finitely generated $\widehat{U(\mathcal{H})}_K$-module, with $J:=\Ann_{\widehat{U(\mathcal{H})}_K}M$. Then $\Ann_{\widehat{U(\mathcal{L})}_K}\left(\widehat{U(\mathcal{L})}_K\otimes_{\widehat{U(\mathcal{H})}_K}M\right)$ is the largest two-sided ideal in $\widehat{U(\mathcal{L})}_K$ contained in $\widehat{U(\mathcal{L})}_KJ$. It follows that if $M,N$ are finitely generated $\widehat{U(\mathcal{H})}_K$-modules such that $\Ann_{\widehat{U(\mathcal{H})}_K}M=\Ann_{\widehat{U(\mathcal{H})}_K}N$ then 

\begin{center}
$\Ann_{\widehat{U(\mathcal{L})}_K}(\widehat{U(\mathcal{L})}_K\otimes_{\widehat{U(\mathcal{H})}_K}M)=\Ann_{\widehat{U(\mathcal{L})}_K}(\widehat{U(\mathcal{L})}_K\otimes_{\widehat{U(\mathcal{H})}_K}N)$.
\end{center}

\end{lemma}

\begin{proof}

Fix an $\mathcal{O}$-basis $\{x_1,\cdots,x_d\}$ for $\mathcal{L}$ such that $\{x_1,\cdots,x_r\}$ is a basis for $\mathcal{H}$. Then by Lemma \ref{PBW}, every element $r\in\widehat{U(\mathcal{L})}_K$ can be written as $r=\underset{\alpha\in\mathbb{N}^d}{\sum}\lambda_{\alpha}x_1^{\alpha_1}\cdots x_d^{\alpha_d}$, for some $\lambda_{\alpha}\in K$ converging to zero as $\alpha\rightarrow\infty$, i.e. $r=\underset{\alpha\in\mathbb{N}^r}{\sum}{\underline{x}^{\alpha}s_{\alpha}}$ for some $s_{\alpha}\in\widehat{U(\mathcal{H})}_K$ such that $s_{\alpha}\rightarrow 0$ as $\alpha\rightarrow\infty$.\\

\vspace{0.1in}

\noindent Using Proposition \ref{Tate-action}, we see that $\widehat{U(\mathcal{L})}_K\otimes_{\widehat{U(\mathcal{H})}_K}M$ is isomorphic as a $K$-vector space to the Tate module $M\langle t_1,\cdots,t_r\rangle=\{\underset{\alpha\in\mathbb{N}^r}{\sum}t_1^{\alpha_1}\cdots t_r^{\alpha_r}s_{\alpha}:s_{\alpha}\in M, s_{\alpha}\rightarrow 0$ as $\vert\alpha\vert\rightarrow\infty\}$ via an isomorphism $\Psi$ sending $\underline{x}^{\alpha}\otimes m$ to $t_1^{\alpha_1}\cdots t_r^{\alpha_r}m$. It is clear that the set of all elements in $\widehat{U(\mathcal{L})}_K$ that annihilate the set $M$ inside $M\langle t_1,\cdots,t_r\rangle$ on the left contains the left ideal $\widehat{U(\mathcal{L})}_KJ$.\\

\noindent Moreover, if $rM=0$ for some $r=\underset{\alpha\in\mathbb{N}^r}{\sum}{\underline{x}^{\alpha}s_{\alpha}}\in\widehat{U(\mathcal{L})}_K$, then for all $m\in M$:

\begin{center}
$0=rm=\underset{\alpha\in\mathbb{N}^r}{\sum}{\underline{x}^{\alpha}s_{\alpha}m}=\Psi^{-1}(\underset{\alpha\in\mathbb{N}^r}{\sum}{t_1^{\alpha_1}\cdots t_r^{\alpha_r} s_{\alpha}m})$, and hence $s_{\alpha}m=0$ for all $\alpha\in\mathbb{N}^r$.
\end{center}

\noindent Thus $s_{\alpha}\in J$ for all $\alpha$, and hence $r\in\widehat{U(\mathcal{L})}_KJ$. Therefore the right ideal $\widehat{U(\mathcal{L})}_KJ$ is the set of all elements of $\widehat{U(\mathcal{L})}_K$ that annihilate the set $M$.\\

\noindent It follows that if $r\widehat{U(\mathcal{L})}_K\otimes_{\widehat{U(\mathcal{H})}_K}M=0$, then $r\widehat{U(\mathcal{L})}_K$ annihilates $M$, so the two-sided ideal generated by $r$ is contained in $\widehat{U(\mathcal{L})}_KJ$. Since our choice of $r$ was arbitrary, it follows that the annihilator of $\widehat{U(\mathcal{L})}_K\otimes_{\widehat{U(\mathcal{H})}_K}M$ is contained in $\widehat{U(\mathcal{L})}_KJ$.\\

\noindent Furthermore, if $I\subseteq\widehat{U(\mathcal{L})}_KJ$ is a two-sided ideal of $\widehat{U(\mathcal{L})}_K$, then $I$ annihilates $M$, so since $I\widehat{U(\mathcal{L})}_K=\widehat{U(\mathcal{L})}_KI$, it must also annihilate the submodule generated by $M$ inside $\widehat{U(\mathcal{L})}_K\otimes_{\widehat{U(\mathcal{H})}_K}M$, which is clearly the whole module, and the result follows.\end{proof}

\noindent The next result will be essential to several of our proofs, since it allows us to safely pass to and from a reducing quadruple when studying two-sided ideals in $\widehat{U(\mathcal{L})}_K$.

\begin{theorem}\label{control}

Suppose that $\mathfrak{g}$ has a reducing quadruple $(x,y,z,\mathfrak{g}')$ with $x,y,z\in\mathcal{L}$, and let $\mathcal{L}':=\mathcal{L}\cap\mathfrak{g}'$. Then if $I$ is a two-sided ideal of $\widehat{U(\mathcal{L})}_K$ such that $z+I$ is not a zero divisor in $\widehat{U(\mathcal{L})}_K/I$, then $I$ is controlled by $\mathcal{L}'$, i.e.:

\begin{center}
$I=\left(I\cap\widehat{U(\mathcal{L}')}_K\right)\widehat{U(\mathcal{L})}_K=\widehat{U(\mathcal{L})}_K\left(I\cap\widehat{U(\mathcal{L}')}_K\right)$.
\end{center}

\end{theorem}

\begin{proof}

Using Lemma \ref{PBW}, we see that every element of $\widehat{U(\mathcal{L})}_K$ can be written as $g(x)$ for some Tate power series $g$ with coefficients in $\widehat{U(\mathcal{L}')}_K$. We will prove that if $g(x)\in I$ then the coefficients of $g$ all lie in $I$, and the result follows.\\

It will suffice to show that if $g(x)=c_0+c_1x+c_2x^2+\cdots\in I$ then the formal derivative $g'(x)=c_1+2c_2x+3c_3x^2+\cdots$ also lies in $I$. Then using an argument similar to the proof of Proposition \ref{Tate-action}$(iii)$, we can construct a sequence of elements in $I$ converging to $c_0$. By closure of $I$ in $\widehat{U(\mathcal{L})}_K$, it follows that $c_0\in I$, so repeating the argument for $\frac{1}{n!}g^{(n)}(x)$ for each $x$, it follows that all coefficients of $g(x)$ lie in $I$ as required.\\

\noindent To prove that $g'(x)$ lies in $I$, consider the action of $y$ on $\widehat{U(\mathcal{L})}_K/I$:\\

\noindent Since $y$ is central in $\mathcal{L}'$, $y$ commutes with everything in $\widehat{U(\mathcal{L}')}_K$. Also, since $[x,y]=\alpha z$, clearly $y\cdot x=xy-\alpha z$, and an easy induction shows that $y\cdot x^n=x^n y-n\alpha x^{n-1}z$. So if $l_y$ is the left action of $y$ on $\widehat{U(\mathcal{L})}_K/I$, $r_y$ is the right action, then $l_y-r_y=-\alpha z\frac{d}{dx}$. Therefore, since $z$ is not a zero divisor modulo $I$, and $\alpha\neq 0$, it follows that if $g(x)\in I$ then $\frac{d}{dx}(g(x))\in I$ as required.\end{proof}

\subsection{Independence Theorem}

\noindent We will now prove the main result of this section, namely that Dixmier annihilator ideals are independent of the choice of polarisation. First, we deal with a special case.

\begin{lemma}\label{Heisenberg}

Suppose that $\mathcal{L}$ has an $\mathcal{O}$-basis $\{x,y,z\}$ such that $z$ is central and $[x,y]=\alpha z$ for some $0\neq\alpha\in\mathcal{O}$. Then for any $0\neq\beta\in\mathcal{O}$, the ideal $(z-\beta)\widehat{U(\mathcal{L})}_K$ is a maximal two-sided ideal of $\widehat{U(\mathcal{L})}_K$.

\end{lemma}

\begin{proof}

Let $I$ be a proper ideal of $\widehat{U(\mathcal{L})}_K$ containing $z-\beta$. Then since $\beta\neq 0$, $z+I$ is not a zero divisor in $\widehat{U(\mathcal{L})}_K/I$. So setting $\mathfrak{g}':=$ Span$_K\{y,z\}$, since $(x,y,z,\mathfrak{g}')$ is a reducing quadruple, it follows from Theorem \ref{control} that $I$ is controlled by $\mathcal{L}'=\mathfrak{g}'\cap\mathcal{L}$. Therefore, if we can prove that $I\cap\widehat{U(\mathcal{L}')}_K=(z-\beta)\widehat{U(\mathcal{L}')}_K$ then it follows that $I=(z-\beta)\widehat{U(\mathcal{L})}_K$.\\

\noindent Let $\bar{y}$ be the image of $y$ in $\widehat{U(\mathcal{L}')}_K/(z-\beta)\widehat{U(\mathcal{L}')}_K$, then given $r\in\widehat{U(\mathcal{L}')}_K$, by Lemma \ref{PBW}, the image of $r$ in $\widehat{U(\mathcal{L}')}_K/(z-\beta)\widehat{U(\mathcal{L}')}_K$ has the form $\bar{r}=\underset{n\geq 0}{\sum}{\lambda_n\bar{y}^n}$ for some $\lambda_n\in K$, $\lambda_n\rightarrow 0$ as $n\rightarrow\infty$.\\

\noindent Since $[x,y]=\alpha z$, we have that $x\cdot\bar{y}=\bar{y}x+\alpha\beta$, and an easy induction shows that $x\cdot\bar{y}^n=\bar{y}^nx+\alpha\beta\bar{y}^{n-1}$ for all $n$, i.e. if $l_x$ and $r_x$ are the respective left and right actions of $x$ on $\widehat{U(\mathcal{L}')}_K/(z-\beta)$, then $l_x-r_x=\alpha\beta\frac{d}{d\bar{y}}$. Since $\alpha,\beta\neq 0$ and $I$ is a two-sided ideal, $\frac{d}{d\bar{y}}$ preserves $I\cap\widehat{U(\mathcal{L}')}_K/(z-\beta)$.\\

\noindent So if $g(\bar{y})=\lambda_0+\lambda_1\bar{y}+\lambda_2\bar{y}^2+\cdots\in I\cap\widehat{U(\mathcal{L}')}_K/(z-\beta)\widehat{U(\mathcal{L}')}_K$, it follows that $\frac{1}{n!}g^{(n)}(\bar{y})\in I$ for all $n\in\mathbb{N}$, and using an argument similar to the proof of Proposition \ref{Tate-action}$(iii)$, we can construct a sequence of elements in $I\cap\widehat{U(\mathcal{L}')}_K$ converging to $\lambda_0\in K$.\\ 

\noindent By closure of $I\cap\widehat{U(\mathcal{L})}_K$, this implies that $\lambda_0\in I$, and hence $\lambda_0=0$, and it follows after replacing $g(\bar{y})$ by $\frac{1}{n!}g^{(n)}(\bar{y})$ that $\lambda_n=0$ for all $n$, i.e. $g(\bar{y})=0$. Therefore $I\cap\widehat{U(\mathcal{L}')}_K=(z-\beta)\widehat{U(\mathcal{L}')}_K$, and $I=(z-\beta)\widehat{U(\mathcal{L})}_K$. So since our choice of $I$ was arbitrary, $(z-\beta)\widehat{U(\mathcal{L})}_K$ is maximal as required.\end{proof}

\vspace{0.1in}

\begin{theorem}\label{independence}

Suppose $\mathfrak{g}$ is nilpotent, and let $\lambda\in\mathcal{L}^*$. Then for any polarisations $\mathfrak{b}_1$,$\mathfrak{b}_2$ of $\mathfrak{g}$ at $\lambda$, $\Ann_{\widehat{U(\mathcal{L})}_K}\widehat{D(\lambda)}_{\mathfrak{b}_1}=\Ann_{\widehat{U(\mathcal{L})}_K}\widehat{D(\lambda)}_{\mathfrak{b}_2}$.

\end{theorem}

\begin{proof}

If $\mathfrak{g}$ is abelian then $\mathfrak{b}_1=\mathfrak{b}_2=\mathfrak{g}$ so the statement is obvious. Since all nilpotent Lie algebras of dimension 1 and 2 are abelian, we may assume that dim$_K\mathfrak{g}\geq 3$.\\

\noindent If $\mathfrak{g}$ is non-abelian and dim$_K\mathfrak{g}=3$ then it is straightforward to show that $\mathcal{L}$ has basis $\{x,y,z\}$ with $z$ central and $[x,y]=\alpha z$ for some $\alpha\in\mathcal{O}\backslash 0$. If $\lambda(z)=0$ then $\lambda$ is a character of $\mathfrak{g}$, so $\mathfrak{g}$ is the only polarisation and the statement is trivially true. If $\lambda(z)\neq 0$, then for any polarisation $\mathfrak{b}$, $z$ acts on $\widehat{D(\lambda)}_{\mathfrak{b}}$ by $\lambda(z)$, and so the $\widehat{U(\mathcal{L})}_K$-annihilator must contain $(z-\lambda(z))$, which is a maximal ideal by Lemma \ref{Heisenberg}, hence this must be the annihilator in all cases as we require.\\

\noindent So from now on, we may assume that $n=$ dim$_K\mathfrak{g}\geq 4$ and we will proceed by Dixmier's induction strategy on $n$:\\

\noindent Suppose first that there exists a non-zero ideal $\mathfrak{a}\trianglelefteq\mathfrak{g}$ such that $\lambda(\mathfrak{a})=0$, so $\lambda$ induces a linear form $\lambda_0$ of $\mathfrak{g}_0:=\mathfrak{g}/\mathfrak{a}$. Setting $\mathcal{A}:=\mathfrak{a}\cap\mathcal{L}$, $\mathcal{L}_0:=\frac{\mathcal{L}}{\mathcal{A}}$, it is clear that $\mathcal{L}_0$ is a Lie lattice in $\mathfrak{g}_0$ and $\lambda_0(\mathcal{L}_0)\subseteq\mathcal{O}$.\\

\noindent Note that $\mathfrak{a}\subseteq\mathfrak{b}_1,\mathfrak{b}_2$ by Lemma \ref{polarisation-properties}, so set $\mathfrak{b}_{i,0}:=\mathfrak{b}_i/\mathfrak{a}$ for $i=1,2$, and $\mathfrak{b}_{1,0}$,$\mathfrak{b}_{2,0}$ are polarisations of $\mathfrak{g}_0$ at $\lambda_0$.\\

\noindent Since dim$_K\mathfrak{g}_0<n$, it follows from induction that $\Ann_{\widehat{U(\mathcal{L}_0)}_K}\widehat{D(\lambda_0)}_{\mathfrak{b}_{1,0}}=\Ann_{\widehat{U(\mathcal{L}_0)}_K}\widehat{D(\lambda_0)}_{\mathfrak{b}_{2,0}}$.\\

\noindent Using Lemma \ref{induced}, we see that for $i=1,2$, $\widehat{D(\lambda_0)}_{\mathfrak{b}_{i,0}}=\widehat{U(\mathcal{L}_0)}_K\otimes_{\widehat{U(\mathcal{B}_{i,0})}_K}K_{\lambda}$ is naturally a $\widehat{U(\mathcal{L})}_K$-module, and that it is isomorphic to $\widehat{U(\mathcal{L})}_K\otimes_{\widehat{U(\mathcal{B}_i)}_K}K_{\lambda}=\widehat{D(\lambda)}_{\mathfrak{b}_i}$.\\

\noindent If $\phi:\widehat{U(\mathcal{L})}_K\twoheadrightarrow\widehat{U(\mathcal{L}_0)}_K$ is the natural surjection, then it is clear that $\Ann_{\widehat{U(\mathcal{L}_0)}_K}\widehat{D(\lambda_0)}_{\mathfrak{b}_{i,0}}=\frac{\Ann_{\widehat{U(\mathcal{L})}_K}\widehat{D(\lambda_0)}_{\mathfrak{b}_{i,0}}}{\ker(\phi)}$ for $i=1,2$. Thus $\Ann_{\widehat{U(\mathcal{L})}_K}\widehat{D(\lambda_0)}_{\mathfrak{b}_{1,0}}=\Ann_{\widehat{U(\mathcal{L})}_K}\widehat{D(\lambda_0)}_{\mathfrak{b}_{2,0}}$, and hence:\\

\noindent $\Ann_{\widehat{U(\mathcal{L})}_K}\widehat{D(\lambda)}_{\mathfrak{b}_1}=\Ann_{\widehat{U(\mathcal{L})}_K}\widehat{D(\lambda_0)}_{\mathfrak{b}_{1,0}}=\Ann_{\widehat{U(\mathcal{L})}_K}\widehat{D(\lambda_0)}_{\mathfrak{b}_{2,0}}=\Ann_{\widehat{U(\mathcal{L})}_K}\widehat{D(\lambda)}_{\mathfrak{b}_2}$ as required.\\

\noindent So from now on, we may assume that $\lambda(\mathfrak{a})\neq 0$ for all non-zero ideals $\mathfrak{a}$ of $\mathfrak{g}$. Using Proposition \ref{sub-polarisation}, we see that this means that there exists a reducing quadruple $(x,y,z,\mathfrak{g}')$ for $\mathfrak{g}$. Since we are assuming dim$_K\mathfrak{g}>3$, we may apply Proposition \ref{small-pol} to get that for each $i=1,2$ there exists a proper subalgebra $\mathfrak{h}_i$ of $\mathfrak{g}$ containing $\mathfrak{b}_i$, and a polarisation $\mathfrak{b}_i'$ of $\mathfrak{g}$ at $\lambda$ contained in $\mathfrak{h}_i$ and $\mathfrak{g}'$.

By induction, since dim$_K\mathfrak{g}'<n$, we get that $\Ann_{\widehat{U(\mathcal{L}')}_K}\widehat{D(\lambda|_{\mathfrak{g}'})}_{\mathfrak{b}_1'}=\Ann_{\widehat{U(\mathcal{L}')}_K}\widehat{D(\lambda|_{\mathfrak{g}'})}_{\mathfrak{b}_2'}$, so by Lemma \ref{ind-annihilator}, $\Ann_{\widehat{U(\mathcal{L})}_K}\widehat{D(\lambda)}_{\mathfrak{b}_1'}=\Ann_{\widehat{U(\mathcal{L})}_K}\widehat{D(\lambda)}_{\mathfrak{b}_2'}$.\\

\noindent Similarly, since $\mathfrak{h}_1,\mathfrak{h}_2$ are proper subalgebras of $\mathfrak{g}$, we also have that $\Ann_{\widehat{U(\mathcal{H}_i)}_K}\widehat{D(\lambda|_{\mathfrak{h_i}})}_{\mathfrak{b}_i}=\Ann_{\widehat{U(\mathcal{H}_i)}_K}\widehat{D(\lambda|_{\mathfrak{h_i}})}_{\mathfrak{b}_i'}$ for $i=1,2$ by induction, and applying Lemma \ref{ind-annihilator} again gives that $\Ann_{\widehat{U(\mathcal{L})}_K}\widehat{D(\lambda)}_{\mathfrak{b}_i}=\Ann_{\widehat{U(\mathcal{L})}_K}\widehat{D(\lambda)}_{\mathfrak{b}_i'}$, and it follows that $\Ann_{\widehat{U(\mathcal{L})}_K}\widehat{D(\lambda)}_{\mathfrak{b}_1}=\Ann_{\widehat{U(\mathcal{L})}_K}\widehat{D(\lambda)}_{\mathfrak{b}_2}$ as required.\end{proof}

\noindent Now that we have established that the annihilator of a Dixmier module does not depend on the choice of polarisation, we can unambiguously make the following definition:

\begin{definition}

Let $F/K$ be a finite extension, and let $\lambda\in\mathcal{L}_F^*:=(\mathcal{L}\otimes_{\mathcal{O}}\mathcal{O}_F)^*$ be a linear form. Define the \emph{Dixmier annihilator} in $\widehat{U(\mathcal{L})}_K$ associated to $\lambda$ to be the two sided ideal $I(\lambda):=\Ann_{\widehat{U(\mathcal{L})}_K}\widehat{D(\lambda)}_F$ (or $I(\lambda)_F$ if it is unclear what the base field is).

\end{definition}

\noindent\textbf{Note:} This definition makes sense because there is a natural embedding $\widehat{U(\mathcal{L})}_K\to\widehat{U(\mathcal{L}_F)}_F$ for any finite extension $F/K$. Using Theorem \ref{aff-Dix1}, we see that $I(\lambda)$ is a primitive ideal of $\widehat{U(\mathcal{L})}_K$ whenever $F=K$.

\section{Locally Closed Ideals}

Now we will study some general ring theoretic properties of the affinoid enveloping algebra.

\subsection{Prime ideals in $\widehat{U(\mathcal{L})}_K$}

\begin{proposition}\label{power-converge}

Let $P$ be a prime ideal of $\widehat{U(\mathcal{L})}_K$, and let $J$ be a two-sided ideal of $R:=\widehat{U(\mathcal{L})}_K/P$. Then if $J\neq 0$ there exists an element $a\in J$ such that $a^n$ does not converge to 0 as $n\rightarrow\infty$.

\end{proposition}

\begin{proof}

Let $w$ be the $\pi$-adic filtration on $\widehat{U(\mathcal{L})}_K$ corresponding to the lattice $\widehat{U(\mathcal{L})}$, and let $\bar{w}$ be the quotient filtration on $R:=\widehat{U(\mathcal{L})}_K/P$. Then since $R$ is complete with respect to $\bar{w}$ and gr$_{\bar{w}}$ $R\cong\frac{\text{gr}_w\text{ }\widehat{U(\mathcal{L})}_K}{\text{gr}_w\text{ }P}$, it follows from \cite[Ch \rom{2} Theorem 2.2.1]{LVO} that $\bar{w}$ is a Zariskian filtration on $R$.\\

\noindent Also, since gr$_w$ $\widehat{U(\mathcal{L})}_K\cong U(\mathcal{L}/\pi\mathcal{L})[t,t^{-1}]$ by \cite[Lemma 3.1]{annals}, and it is well known that $U(\mathcal{L}/\pi\mathcal{L})$ is finitely generated over its centre, it follows that gr$_{\bar{w}}$ $R$ is also finitely generated over its centre.\\

Furthermore, $t=$ gr $\pi$ is central of positive degree in gr $R$, and it is non-nilpotent, so it follows that gr $R$ is finitely generated over a central Noetherian subring whose positive part is non-nilpotent. Hence after applying \cite[Theorem 3.3]{APB}, we can find a filtration $v$ on the ring of quotients $Q(R)$ such that $v$ restricts to a valuation of the centre, and the inclusion $(R,\bar{w})\to (Q(R),v)$ is continuous.\\

\noindent Suppose for contradiction that for every element $a\in J$, $a^n\rightarrow 0$ as $n\rightarrow\infty$. Choose an arbitrary $a\in J$, and following \cite[Definition 3.2]{APB}, define the growth rate function $\rho:Q(R)\to\mathbb{R}\cup\{\infty\},q\mapsto\underset{n\rightarrow\infty}{\lim}{\frac{v(q^n)}{n}}$, and let $m:=\ceil{\rho(a)}$.\\

\noindent If we assume that $m<\infty$, then set $b:=\pi^{-(m+1)}a$, and since $\pi$ is central in $R$, we see using \cite[Lemma 3.7($v$)]{APB} that $\rho(b)=\rho(a)-(m+1)v(\pi)\leq \rho(a)-(\rho(a)+1)v(\pi)<0$, and hence $b^n$ does not converge to 0 as $n\rightarrow\infty$ -- contradiction since $b\in J$.\\

\noindent Therefore $m=\rho(a)=\infty$, so since $Q(R)$ is simple and artinian, it follows from \cite[Lemma 3.7($iv$)]{APB} that $a$ is nilpotent.\\

\noindent Since our choice of $a$ was arbitrary, this means that every element of $J$ is nilpotent, and using \cite[Lemma 3.1.14]{Dixmier} it follows that $J$ is a nilpotent ideal of $R$. Since $R$ is prime, this means that $J=0$ as required.\end{proof}

\noindent The following result is the affinoid version of \cite[Proposition 3.1.15]{Dixmier}:

\begin{theorem}\label{semiprimitive}

Let $I$ be a two sided ideal of $\widehat{U(\mathcal{L})}_K$. Then $I$ is semiprime if and only if $I$ is an intersection of primitive ideals.

\end{theorem}

\begin{proof}

Clearly if $I$ is an intersection of primitive ideals, then it is semiprime, so it remains only to prove the converse, i.e. that if $I$ is semiprime then it is the intersection of primitive ideals.\\

\noindent Since semiprime ideals arise as an intersection of primes, we can assume that $I$ is prime in $\widehat{U(\mathcal{L})}_K$, and we will show that $J(\widehat{U(\mathcal{L})}_K/I)=0$, from which the result follows.\\

\noindent Assume for contradiction that $J:=J(\widehat{U(\mathcal{L})}_K/I)\neq 0$, then since $I$ is prime it follows from Proposition \ref{power-converge} that we can choose an element $a\in J$ such that $a^n$ does not converge to zero as $n\rightarrow\infty$.\\

\noindent Let $R:=\frac{\widehat{U(\mathcal{L})}_K}{I}$, and let $C:=R\langle X\rangle$ be the Tate algebra in one variable over $R$. Then if we set $\mathfrak{g}_0:=\mathfrak{g}\times K$, $\mathcal{L}_0:=\mathcal{L}\times\mathcal{O}$, it is clear that $\mathcal{L}_0$ is a Lie lattice in $\mathfrak{g}_0$ and it follows from Lemma \ref{PBW} that $C\cong\widehat{U(\mathcal{L}_0)}_K/I\widehat{U(\mathcal{L}_0)}_K$.\\

\noindent Consider the element $1-aX\in C$. If this element is a unit, its inverse must have the form $a_0+a_1X+a_2X^2+\cdots$ for some $a_i\in \widehat{U(\mathcal{L})}_K/I$ with $a_n\rightarrow 0$ as $n\rightarrow\infty$. But since $1=(1-aX)(a_0+a_1X+a_2X^2+\cdots)$, it follows that $a_0=1$, $a_1=a$, $a_2=a^2$, $\cdots$, $a_n=a^n$, and hence $a^n\rightarrow 0$ as $n\rightarrow\infty$ -- contradiction.\\

\noindent Therefore $1-aX$ is not a unit in $C$, so there exists a maximal left ideal of $C$ containing $1-aX$, i.e. there exists an irreducible $C$-module $M$ and an element $0\neq m\in M$ such that $(1-aX)m=0$.\\

\noindent Now, $X$ does not act by zero on $M$, otherwise $1-aX$ would act by 1, and we would have $(1-aX)m=m\neq 0$. So using Schur's Lemma, the action of $X$ is invertible, and using \cite[Theorem 6.4.6]{ioan} we see that the action of $X^{-1}$ is algebraic over $K$, i.e. there exists $f(t)=a_0+a_1t+\cdots+a_nt^n$ for some $a_i\in K$ such that $f(X^{-1})=0$, and we may assume that $a_0\neq 0$. So let $g(t):=a_0^{-1}f(t)=1+b_1t+\cdots+b_nt^n$.\\

\noindent Since $aXm=m$, we have that $am=X^{-1}aXm=X^{-1}m$, hence $a^rm=X^{-r}m$ for all $r\in\mathbb{N}$, and thus $g(a)m=g(X^{-1})m=0$.\\

\noindent But $g(a)=1+(b_1+b_2a+\cdots+b_na^{n-1})a$, so since $a\in J(\widehat{U(\mathcal{L})}_K/I)$, this means that $g(a)$ is a unit in $\widehat{U(\mathcal{L})}_K/I$. Therefore, since $m\neq 0$, $g(a)m\neq 0$ -- contradiction.\\

\noindent Therefore $J(\widehat{U(\mathcal{L})}_K/I)=0$ as we require.\end{proof}

\subsection{Towards the Dixmier-Moeglin equivalence}

\noindent Now, recall from the introduction that a prime ideal $P$ of $\widehat{U(\mathcal{L})}_K$ is \emph{locally closed} if 

\begin{center}
$P\neq\cap\{Q\trianglelefteq\widehat{U(\mathcal{L})}_K: Q$ prime,$ P\subsetneq Q\}$.
\end{center}

\begin{proposition}\label{nullstellensatz}

Let $P$ be a prime ideal of $\widehat{U(\mathcal{L})}_K$. Then:

\begin{center}
$P$ is locally closed $\implies$ $P$ is primitive $\implies$ $P$ is weakly rational.
\end{center}

\end{proposition}

\begin{proof}

First, suppose that $P$ is primitive, i.e. $P=$ Ann$_{\widehat{U(\mathcal{L})}_K}M$ for some simple $\widehat{U(\mathcal{L})}_K$-module $M$. Then it follows from \cite[Theorem 6.4.6]{ioan} that every element of  End$_{\widehat{U(\mathcal{L})}_K}M$ is algebraic over $K$. So since $Z(\widehat{U(\mathcal{L})}_K/P)$ is a domain, and embeds into End$_{\widehat{U(\mathcal{L})}_K}M$, it follows that $Z(\widehat{U(\mathcal{L})}_K/P)$ is an algebraic field extension of $K$, and hence $P$ is weakly rational.\\

\noindent Using Theorem \ref{semiprimitive}, we see that if $P$ is equal to an intersection of primitive ideals. So if $P$ is locally closed, i.e. not equal to the intersection of all prime ideals properly containing it, then $P$ must be primitive.\end{proof}

\noindent\textbf{Note:} These implications also hold for primes $P$ in $U(\mathfrak{g})$, and in the case where $\mathfrak{g}$ is nilpotent, $P$ weakly rational $\implies$ $P$ maximal \cite[Proposition 4.7.4]{Dixmier}. We suspect that this is also true in the affinoid case, but this is only known in the case where $\mathfrak{g}$ contains an abelian ideal of codimension 1 \cite[Corollary 1.6.2]{Lewis}.\\

\noindent Using this result, we see that to prove the Dixmier-Moeglin equivalence for $\widehat{U(\mathcal{L})}_K$, it remains to show that weakly rational ideals are locally closed. Taking steps in this direction, the aim of this section is to prove that all locally closed prime ideals $P$ of $\widehat{U(\mathcal{L})}_K$ have the form of a Dixmier annihilator $I(\lambda)_F$ for some finite extension $F$ of $K$.

\begin{proposition}\label{prime-intersection}

Let $\mathfrak{g}$ be a nilpotent $K$-Lie algebra, and let $\mathcal{L}$ be an $\mathcal{O}$-Lie lattice in $\mathfrak{g}$ such that every locally closed prime ideal in $\widehat{U(\mathcal{L})}_K$ has the form $I(\lambda)_F$ for some finite extension $F/K$ and some $K$-linear map $\lambda:\mathfrak{g}\to F$ such that $\lambda(\mathcal{L})\subseteq\mathcal{O}_F$.\\

\noindent Then given any prime ideal $P$ in $\widehat{U(\mathcal{L})}_K$, $P$ arises as an intersection of Dixmier annihilators.

\end{proposition}

\begin{proof}

Note that since gr $\widehat{U(\mathcal{L})}_K\cong U(\mathcal{L}/\pi\mathcal{L})[t,t^{-1}]$ has finite left and right Krull dimension, it follows  from \cite[Ch.\rom{1} Theorem 7.1.3]{LVO} that $\widehat{U(\mathcal{L})}_K$ has finite left and right Krull dimension. Therefore, using \cite[Lemma 6.4.5]{McConnell}, it follows that $\widehat{U(\mathcal{L})}_K$ has finite \emph{classical Krull dimension}, i.e. there is a finite upper bound on the length of chains of prime ideals in $\widehat{U(\mathcal{L})}_K$.\\

\noindent So, given a prime ideal $P$ of $R$, define the \emph{dimension} $\dim(P)$ of $P$ to be the largest integer $n\geq 0$ such that there exists a chain of prime ideals $P=P_0\subsetneq P_1\subsetneq\cdots\subsetneq P_n$ of $\widehat{U(\mathcal{L})}_K$. We will proceed by induction on $\dim(P)$.\\

\noindent If $\dim(P)=0$, then $P$ is maximal, and hence locally closed, so $P=I(\lambda)_F$ for some finite extension $F$, $\lambda:\mathfrak{g}\to F$ as required. So suppose the result holds whenever $\dim(P)<n$.\\

\noindent If $\dim(P)=n$ then for every prime ideal $Q$ of $\widehat{U(\mathcal{L})}_K$ with $P\subsetneq Q$, $Q$ arises as an intersection of Dixmier annihilators by the inductive hypothesis.\\

\noindent If $P$ is locally closed, then $P$ is a Dixmier annihilator by assumption, otherwise $P$ is equal to the intersection of all prime ideals properly containing it, and hence it is an intersection of Dixmier annihilators as required.\end{proof}

\subsection{Classification}

\noindent Now we will complete the classification of locally closed ideals of $\widehat{U(\mathcal{L})}_K$ in terms of Dixmier annihilators. First, we need some technical results:

\begin{lemma}\label{rational}

Suppose that \emph{dim}$_K\mathfrak{g}>1$, and let $I$ be an ideal of $\widehat{U(\mathcal{L})}_K$. Suppose further that $Z(\widehat{U(\mathcal{L})}_K/I)=K$, and $I\cap\mathfrak{g}=0$. Then $\mathfrak{g}$ has a reducing quadruple $(x,y,z,\mathfrak{g}')$.

\end{lemma}

\begin{proof}

Firstly, suppose $u,v\in\mathfrak{g}$ are central, then $u+I,v+I\in Z(\widehat{U(\mathcal{L})}_K/I)=K$, hence they are $K$-linearly dependent. So there exist non-zero $\alpha,\beta\in K$ such that $\alpha u+\beta v\in I\cap\mathfrak{g}=0$, hence $u,v$ are $K$-linearly dependent in $\mathfrak{g}$.

Since $\mathfrak{g}$ is nilpotent, $Z(\mathfrak{g})\neq 0$, so it follows that $Z(\mathfrak{g})$ has dimension 1, i.e. $Z(\mathfrak{g})=Kz$ for some $z\in Z(\mathfrak{g})$.\\

\noindent So, since dim$_K\mathfrak{g}>1$, $\mathfrak{g}$ is non-abelian, and again using nilpotence of $\mathfrak{g}$, we can find $y\in\mathfrak{g}$ such that $y$ is not central, but $[y,\mathfrak{g}]\subseteq Z(\mathfrak{g})=Kz$.\\

\noindent Let $\mathfrak{g}':=\ker(\ad(y))$. Then $\mathfrak{g}'$ is an ideal of $\mathfrak{g}$, and since $ad(y):\mathfrak{g}\to Kz$ is non-zero, $\mathfrak{g}'$ must have codimension 1 in $\mathfrak{g}$, so $\mathfrak{g}=\mathfrak{g}'\oplus Kx$, and it is clear that $(x,y,z,\mathfrak{g}')$ is a reducing quadruple for $\mathfrak{g}$.\end{proof}

\begin{proposition}\label{ind-centre}

Let $I$ be a two sided ideal of $\widehat{U(\mathcal{L})}_K$ such that $F=Z(\widehat{U(\mathcal{L})}_K/I)$ is a finite field extension of $K$. Then $\widehat{U(\mathcal{L}\otimes_{\mathcal{O}}\mathcal{O}_F)}_F\cong \widehat{U(\mathcal{L})}_K\otimes_KF$, and there exists a surjection $\widehat{U(\mathcal{L}\otimes_{\mathcal{O}}\mathcal{O}_F)}_F\to\widehat{U(\mathcal{L})}_K/I$ of $F$-algebras with kernel containing $I\otimes_KF$.

\end{proposition}

\begin{proof}

To see that $\widehat{U(\mathcal{L}\otimes_{\mathcal{O}}\mathcal{O}_F)}_F\cong \widehat{U(\mathcal{L})}_K\otimes_KF$, note that:

\begin{center}
$U(\mathcal{L})\otimes_{\mathcal{O}}\mathcal{O}_F\leftrightarrow U(\mathcal{L}\otimes_{\mathcal{O}}\mathcal{O}_F)$
\end{center}

\begin{center}
$a\otimes\alpha\mapsto \alpha a$
\end{center}

\begin{center}
$u\otimes\alpha\mapsfrom u\otimes\alpha$
\end{center}

\noindent Are isomorphisms of $\mathcal{O}_F$ algebras, preserving the $\pi$-adic filtration. Hence they induce an isomorphism $\widehat{U(\mathcal{L}\otimes_{\mathcal{O}}\mathcal{O}_F)}\cong \widehat{U(\mathcal{L})}\otimes_{\mathcal{O}}\mathcal{O}_F$, and the result follows.\\

\noindent Since $F=Z(\widehat{U(\mathcal{L})}_K/I)\subseteq \widehat{U(\mathcal{L})}_K/I$, it is clear that $\widehat{U(\mathcal{L})}_K/I$ is an $F$-algebra, and the map $\widehat{U(\mathcal{L})}_K\otimes_KF\to\widehat{U(\mathcal{L})}_K/I,r\otimes (\alpha+I)\mapsto \alpha r+I$ is clearly a surjection of $F$-algebras sending $I\otimes F$ to $0$ as required.\end{proof}

\begin{lemma}\label{Dix-quotient}

Let $\mathfrak{a}$ be an ideal of $\mathfrak{g}$ nilpotent, let $\mathcal{A}:=\mathfrak{a}\cap\mathcal{L}$ and let $\mathcal{L}_0:=\mathcal{L}/\mathcal{A}$. Let $P$ be a prime ideal of $\widehat{U(\mathcal{L})}_K$, containing $\mathfrak{a}$, such that the image $P_0\trianglelefteq\widehat{U(\mathcal{L}_0)}_K$ of $P$ under the surjection $\widehat{U(\mathcal{L})}_K\to\widehat{U(\mathcal{L}_0)}_K$ is a Dixmier annihilator. Then $P$ is a Dixmier annihilator.

\end{lemma}

\begin{proof}

We know that $P_0=Ann_{\widehat{U(\mathcal{L}_0)}_K}\widehat{D(\mu)}_F$ for some finite extension $F/K$, $\mu\in (\mathcal{L}/\mathcal{A})_F^*$. Clearly $\mu$ is induced from a linear form $\lambda$ of $\mathfrak{g}\otimes_KF$ such that $\lambda(\mathcal{L})\subseteq\mathcal{O}_F$ and $\lambda(\mathfrak{a})=0$. We will prove that $P=Ann_{\widehat{U(\mathcal{L})}_K}\widehat{D(\lambda)}_F$.\\

\noindent Choose a polarisation $\mathfrak{b}$ of $\mathfrak{g}\otimes_KF$ at $\lambda$, and since the annihilator is independent of the choice of polarisation by Theorem \ref{independence}, we may assume that $\mathfrak{a}\subseteq\mathfrak{b}$, i.e. $\mathfrak{b}/\mathfrak{a}$ is a polarisation of $\mathfrak{g}/\mathfrak{a}$ at $\mu$. Using Lemma \ref{induced}$(iii)$, we see that $\widehat{D(\lambda)}_F=\widehat{U(\mathcal{L})}_F\otimes_{\widehat{U(\mathcal{B})}_F}F\cong\widehat{U(\mathcal{L}/\mathcal{A})_F}\otimes_{\widehat{U(\mathcal{B}/\mathcal{A})}_F}F=\widehat{D(\mu)}_F$.\\

\noindent Using Lemma \ref{induced}$(i)$, we know that $\widehat{U(\mathcal{L})}/\widehat{\mathfrak{a}U(\mathcal{L})}\cong\widehat{U(\mathcal{L}/\mathcal{A})}_K$, and hence $P_0=P/\mathfrak{a}\widehat{U(\mathcal{L})}_K$. Therefore, since $P_0=\Ann_{\widehat{U(\mathcal{L}_0)}_K}\widehat{D(\mu)}_F$, and hence $P\widehat{D(\mu)}_F=0$, it follows that $P\widehat{D(\lambda)}_F=0$, i.e. $P\subseteq \Ann_{\widehat{U(\mathcal{L})}_K}\widehat{D(\lambda)}_F$. 

Moreover, if $x\widehat{D(\lambda)}_F=0$ then $x\widehat{D(\mu)}_F=0$ so $x+\mathfrak{a}\widehat{U(\mathcal{L})}_K\in P_0$ and hence $x\in P$. Therefore  $P=\Ann_{\widehat{U(\mathcal{L})}_K}\widehat{D(\lambda)}_F$ as required.\end{proof}

\noindent Now we can prove the main theorem of this section, classifying locally closed ideals in terms of Dixmier annihilators.

\begin{theorem}\label{aff-Dix2}

Let $\mathfrak{g}$ be a nilpotent $K$-lie algebra, with $\mathcal{O}$-Lie lattice $\mathcal{L}$, and let $P$ be a locally closed prime ideal of $\widehat{U(\mathcal{L})}_K$. Then there exists a finite extension $F\backslash K$ and a $K$-linear map $\lambda:\mathfrak{g}\to F$ with $\lambda(\mathcal{L})\subseteq\mathcal{O}_F$ such that $P=I(\lambda)_F$.

\end{theorem}

\begin{proof}

We will use induction on $n=$ dim$_{K}\mathfrak{g}$:\\

\noindent First suppose that $n=1$, and hence $\widehat{U(\mathcal{L})}_K\cong K\langle u\rangle$ by Lemma \ref{PBW}. So if $P$ is a locally closed ideal, then it is primitive, and hence maximal since $\widehat{U(\mathcal{L})}_K$ is commutative. So let $F:=\widehat{U(\mathcal{L})}_K/P$, then $F$ is a field.\\

\noindent Furthermore, using \cite[Corollary 2.2.12]{Bosch}, we see that $F$ is a finite extension of $K$, so define $\lambda:\mathfrak{g}\to F,x\mapsto x+P$, and clearly this map is $K$-linear. Also, $\widehat{U(\mathcal{L})}/P\cap \widehat{U(\mathcal{L})}=\mathcal{O}\langle u\rangle/P\cap\mathcal{O}\langle u\rangle$ is a lattice in $F=K\langle u\rangle/P$. Thus $\widehat{U(\mathcal{L})}/P\cap \widehat{U(\mathcal{L})}\subseteq\mathcal{O}_F$ so clearly $\lambda(\mathcal{L})\subseteq\mathcal{O}_F$.\\

\noindent So $\widehat{D(\lambda)}_F=F$, where $x\in\widehat{U(\mathcal{L})}_K$ acts by zero if and only if $\lambda(x)=0$, i.e. if and only if $x\in P$, so $P=\Ann_{\widehat{U(\mathcal{L})}_K}\widehat{D(\lambda)}_F=I(\lambda)_F$ as required.\\

\noindent So now suppose that the result holds whenever $\dim_K\mathfrak{g}<n$.\\

\noindent Again, suppose that $P$ is a locally closed ideal of $\widehat{U(\mathcal{L})}_K$, and let $\mathfrak{a}:=P\cap\mathfrak{g}$, $\mathcal{A}:=\mathfrak{a}\cap\mathcal{L}$. Clearly $\mathfrak{a}$ is an ideal of $\mathfrak{g}$, contained in $P$, and $\mathcal{A}$ is a Lie lattice in $\mathfrak{a}$. We will suppose first that $\mathfrak{a}\neq 0$.\\

\noindent Let $P_0$ be the image of $P$ under the surjection $\widehat{U(\mathcal{L})}_K\to\widehat{U(\mathcal{L}/\mathcal{A})}_K$, then $P_0$ is a locally closed ideal of $\widehat{U(\mathcal{L}/\mathcal{A})}_K$. Since dim$_K\mathfrak{g}/\mathfrak{a}<n$, it follows from induction that $P_0$ is a Dixmier annihilator. Therefore, using Lemma \ref{Dix-quotient}, $P$ is a Dixmier annihilator as required.\\

\noindent So from now on we may assume that $\mathfrak{a}=P\cap\mathfrak{g}=0$.\\

\noindent Since we know by Proposition \ref{nullstellensatz} that $P$ is primitive, it follows from \cite[Theorem 6.4.6]{ioan} that $F=Z(\widehat{U(\mathcal{L})}_K/P)$ is an algebraic field extension of $K$. Since the centre of $\widehat{U(\mathcal{L})}_K/P$ is closed and $\widehat{U(\mathcal{L})}_K/P$ is complete, it follows that $F$ is complete, so it must in fact be a finite extension of $K$.\\ 

\noindent We will assume for now that $F=K$, so applying Lemma \ref{rational}, we see that $\mathfrak{g}$ has a reducing quadruple $(x,y,z,\mathfrak{g}')$. So let $\mathcal{L}':=\mathfrak{g}'\cap\mathcal{L}$, then since $z\notin P$, it is clear that $z+P\in Z(\widehat{U(\mathcal{L})}_K/P)=K$ is not a zero divisor, so using Theorem \ref{control}, we see that $P$ is controlled by $\mathcal{L}'$, i.e. $P=\widehat{U(\mathcal{L})}_K(P\cap\widehat{U(\mathcal{L}')}_K)$.\\

\noindent Let $Q:=P\cap\widehat{U(\mathcal{L}')}_K$, then $Q$ is a semiprime ideal of $\widehat{U(\mathcal{L}')}_K$, so since all locally closed prime ideals in $\widehat{U(\mathcal{L}')}_K$ are Dixmier annihilators by induction, it follows from Proposition \ref{prime-intersection} that all semiprime ideals arise as an intersection of Dixmier annihilators, i.e. there exist finite extensions $F_j/K$, $\mu_j\in(\mathcal{L}')_{F_j}^*$, as $j$ ranges over some indexing set $X$, and 

\begin{center}
$Q=\underset{j\in X}{\bigcap}{I(\mu_j)_{F_j}}$.
\end{center}

\noindent Since $z\notin Q$ and $Z(\widehat{U(\mathcal{L})}_K/P)=K$, there exists $0\neq\beta\in K$ such that $z-\beta\in Q$. Therefore $z-\beta\in I(\mu_j)_{F_j}$ for each $j$. Since $\beta\neq 0$, this means that $z\notin I(\mu_j)_{F_j}$, i.e. $\mu_j(z)\neq 0$.\\

\noindent Now, it is clear that $(x\otimes 1,y\otimes 1,z\otimes 1,\mathfrak{g}'\otimes_KF_j)$ is a reducing quadruple for $\mathfrak{g}\otimes_KF_j$, so applying Lemma \ref{extension} gives that if $\mathfrak{b}$ is a polarisation of $\mathfrak{g}'\otimes_KF_j$ at $\mu_j$ and $\lambda_j$ is an extension of $\mu_j$ to $\mathfrak{g}\otimes_KF_{j}$, then $\mathfrak{b}$ is a polarisation of $\mathfrak{g}\otimes_KF_j$ at $\lambda_j$.\\

\noindent Therefore, $\widehat{D(\lambda_j)}_{F_j}\cong\widehat{U(\mathcal{L})}_{F_j}\otimes_{\widehat{U(\mathcal{L}')}_{F_j}}\widehat{D(\mu_j)}_{F_j}$, so by Lemma \ref{ind-annihilator}, $I(\lambda_j)_{F_j}=\Ann_{\widehat{U(\mathcal{L})}_{F_j}}\widehat{D(\lambda_j)}_{F_j}$ is the largest two-sided ideal of $\widehat{U(\mathcal{L})}_{F_j}$ contained in $\widehat{U(\mathcal{L})}_{F_j}\Ann_{\widehat{U(\mathcal{L}')}_F}\widehat{D(\mu_j)}_{F_j}$.\\

\noindent But $P=\widehat{U(\mathcal{L})}_KQ\subseteq \widehat{U(\mathcal{L})}_KI(\mu_j)_{F_j}$, and by Proposition \ref{ind-centre}, $\widehat{U(\mathcal{L})}_{F_j}=\widehat{U(\mathcal{L})}_K\otimes_K{F_j}$, hence $P\otimes_K{F_j}\subseteq\widehat{U(\mathcal{L})}_{F_j}I(\mu_j)_{F_j}\subseteq\widehat{U(\mathcal{L})}_{F_j}\Ann_{\widehat{U(\mathcal{L}')}_{F_j}}\widehat{D(\mu_j)}_{F_j}$. 

Thus $P\otimes_KF_j\subseteq \Ann_{\widehat{U(\mathcal{L})}_{F_j}}\widehat{D(\lambda_j)}_{F_j}$ and $P\subseteq \Ann_{\widehat{U(\mathcal{L})}_K}\widehat{D(\lambda_j)}_{F_j}=I(\lambda_j)_{F_j}$.\\

\noindent Furthermore, given $r\in\underset{j\in X}{\bigcap}{I(\lambda_j)_{F_j}}$, we have that $r=\underset{i\geq 0}{\sum}{x^ir_i}$ for some $r_i\in\widehat{U(\mathcal{L}')}_K$ by Lemma \ref{PBW}, with $r_i\rightarrow 0$ as $i\rightarrow\infty$. Then since each $I(\lambda_j)_{F_j}$ is a prime ideal of $\widehat{U(\mathcal{L})}_K$, and $z\notin I(\lambda_j)_{F_j}$, it follows from Theorem \ref{control} that each $r_i$ lies in $I(\lambda_j)_{F_j}$ for every $j$.\\

\noindent This means that $r_i\widehat{D(\lambda_j)}_{F_j}=0$ for all $i,j$, so $r_i\widehat{D(\mu_j)}_{F_j}=0$ and thus $r_i\in\underset{j\in X}{\bigcap}{I(\mu_j)_{F_j}}=Q$ for every $i$. Therefore $r\in \widehat{U(\mathcal{L})}_K Q=P$. Since our choice of $r$ was arbitrary, it follows that:

\begin{center}
$P=\underset{j\in X}{\bigcap}{I(\lambda_j)_{F_j}}$.
\end{center}

\noindent Since $P$ is locally closed and each $I(\lambda_j)_{F_j}$ is a prime ideal of $\widehat{U(\mathcal{L})}_K$ containing $P$, it follows that $P=I(\lambda_j)_{F_j}$ for some $j\in X$ as we require.\\

\noindent Finally, take $P$ to be a general locally closed prime ideal. Then  $F=Z(\widehat{U(\mathcal{L})}_K/P)$ is a finite extension of $K$, so let $\mathfrak{g}_0:=\mathfrak{g}\otimes_KF$, $\mathcal{L}_0:=\mathcal{L}\otimes_{\mathcal{O}}\mathcal{O}_F$. Then dim$_F\mathfrak{g}_0=$ dim$_K\mathfrak{g}=n$, $\mathcal{L}_0$ is a Lie lattice in $\mathfrak{g}_0$, and by Proposition \ref{ind-centre}, there exists a surjection of $F$-algebras $\widehat{U(\mathcal{L}_0)}_F=\widehat{U(\mathcal{L})}_K\otimes_KF\twoheadrightarrow\widehat{U(\mathcal{L})}_K/P$ whose kernel contains $P\otimes_KF$. Let $J$ be this kernel.\\

\noindent Then $J$ is a locally closed prime ideal of $\widehat{U(\mathcal{L}_0)}_F$ and $\widehat{U(\mathcal{L}_0)}_F/J\cong\widehat{U(\mathcal{L})}_K/P$. But $Z(\widehat{U(\mathcal{L})}_F/J)\cong Z(\widehat{U(\mathcal{L})}_K/P)=F$ so it follows from the above discussion that $J=\Ann_{\widehat{U(\mathcal{L}_0)}_F}\widehat{D(\lambda)}_{F'}$ for some finite extension $F'/F$ and some linear form $\lambda$ of $\mathfrak{g}_0\otimes_FF'$ such that $\lambda(\mathcal{L}_0)\subseteq\mathcal{O}_{F'}$.\\

\noindent It is clear that $J\cap\widehat{U(\mathcal{L})}_K=P$, and hence $P=\Ann_{\widehat{U(\mathcal{L})}_K}\widehat{D(\lambda)}_{F'}=I(\lambda)_{F'}$ as required.\end{proof}

\begin{corollary}\label{Dix-intersection}

Let $\mathcal{L}$ be a Lie lattice in $\mathfrak{g}$ nilpotent. Then given a prime ideal $P$ of $\widehat{U(\mathcal{L})}_K$, $P$ arises as an intersection of Dixmier annihilators.

\end{corollary}

\begin{proof}

This is immediate from Theorem \ref{aff-Dix2} and Proposition \ref{prime-intersection}.\end{proof}

\section{Weakly Rational Ideals}

In this section, we will prove our main result, Theorem \ref{A}, which allows us to describe all weakly rational ideals in $\widehat{U(\mathcal{L})}_K$ in terms of Dixmier annihilators.

\subsection{Dixmier Annihilators under Isomorphisms}

In section 2, we saw how the adjoint algebraic group $\mathbb{G}$ of $\mathfrak{g}$ acts on $\mathfrak{g}^*$ via the coadjoint action. More generally, if $\sigma$ is any Lie automorphism of $\mathfrak{g}$, then for any $\lambda$ in $\mathfrak{g}^*$, we can similarly define $\sigma\cdot\lambda:\mathfrak{g}\to K,u\mapsto\lambda(\sigma^{-1}(u))$. Note that if $\mathcal{L}$ is a Lie lattice in $\mathfrak{g}$ then $\sigma(\mathcal{L})$ also is.\\

\noindent More generally, let us suppose that $\mathcal{L}_1,\mathcal{L}_2$ are Lie lattices in $\mathfrak{g}$ and $\sigma:\mathcal{L}_1\to\mathcal{L}_2$ is an $\mathcal{O}$-linear Lie isomorphism. Then $\sigma$ extends to a $K$-linear isomorphism $\sigma:\widehat{U(\mathcal{L}_1)}_K\to\widehat{U(\mathcal{L}_2)}_K$ of affinoid enveloping algebras. 

Given a linear form $\lambda\in$ Hom$_{\mathcal{O}}(\mathcal{L}_1,\mathcal{O})$, $\sigma\cdot\lambda\in$ Hom$_{\mathcal{O}}(\mathcal{L}_2,\mathcal{O})$, and if $\mathcal{B}_1$ is a polarisation at $\lambda$ then $\mathcal{B}_2=\sigma\mathcal{B}_1$ is a polarisation at $\sigma\cdot\lambda$.

\begin{lemma}\label{transport-structure}

Let $I(\lambda):=$ Ann$_{\widehat{U(\mathcal{L}_1)}_K}\widehat{D(\lambda)}_{\mathcal{B}_1}\trianglelefteq\widehat{U(\mathcal{L}_1)}_K$, $I(\sigma\cdot\lambda):=$ Ann$_{\widehat{U(\mathcal{L}_2)}_K}\widehat{D(\sigma\cdot\lambda)}_{\mathcal{B}_2}\trianglelefteq\widehat{U(\mathcal{L}_2)}_K$. Then $\sigma(I(\lambda))=I(\sigma\cdot\lambda)$

\end{lemma}

\begin{proof}

By definition, $\widehat{D(\lambda)}_{\mathcal{B}_1}=\widehat{U(\mathcal{L}_1)}_K\otimes_{\widehat{U(\mathcal{B}_1)}_K}K_{\lambda}$ and $\widehat{D(\sigma\cdot\lambda)}_{\mathcal{B}_2}=\widehat{U(\mathcal{L}_2)}_K\otimes_{\widehat{U(\mathcal{B}_2)}_K}K_{\sigma\cdot\lambda}=\widehat{U(\sigma\mathcal{L}_1)}_K\otimes_{\widehat{U(\sigma\mathcal{B}_1)}_K}K_{\sigma\cdot\lambda}$. So consider the map $\Theta:\widehat{D(\lambda)}_{\mathcal{B}_1}\to \widehat{D(\sigma\cdot\lambda)}_{\mathcal{B}_2},x\otimes v\mapsto \sigma(x)\otimes v$.\\

\noindent We will show that $\Theta$ is a $K$-linear isomorphism such that $\Theta(xm)=\sigma(x)\Theta(m)$ for all $x\in\widehat{U(\mathcal{L})}_K$, $m\in\widehat{D(\lambda)}_{\mathcal{B}_1}$. It will follow from this that $x\widehat{D(\lambda)}_{\mathcal{B}_1}=0$ if and only if $\sigma(x)\widehat{D(\sigma\cdot\lambda)}_{\mathcal{B}_2}=0$, and hence $\sigma(I(\lambda))=I(\sigma\cdot\lambda)$ as required.\\

\noindent It is clear that $\Theta$ is $K$-linear, and that it has an inverse defined by $x\otimes v\mapsto \sigma^{-1}(x)\otimes v$, hence it is an isomorphism of vector spaces.\\

\noindent Finally, $\Theta(x(y\otimes v))=\Theta(xy\otimes v)=\sigma(xy)\otimes v=\sigma(x)(\sigma(y)\otimes v)=\sigma(x)\Theta(y\otimes v)$.\end{proof}

\noindent This result becomes particularly useful when comparing Dixmier annihilators, particularly using the following lemma.

\begin{lemma}\label{orbit}

Suppose $\lambda,\mu\in$ Hom$_{\mathcal{O}}(\mathcal{L},\mathcal{O})$ such that $I(\lambda)\cap U(\mathfrak{g})= I(\mu)\cap U(\mathfrak{g})$, e.g. if  $I(\lambda)\subseteq I(\mu)$. Then there exists $g\in\mathbb{G}(K)$ such that $\mu=g\cdot\lambda$.

\end{lemma}

\begin{proof}

Since the classical Dixmier module $D(\lambda)$ is dense in $\widehat{D(\lambda)}$, it follows that $r\widehat{D(\lambda)}=0$ if and only if $rD(\lambda)=0$, and hence $I(\lambda)\cap U(\mathfrak{g})=$ Ann$_{U(\mathfrak{g})}D(\lambda)$. Therefore, if $I(\lambda)\cap U(\mathfrak{g})= I(\mu)\cap U(\mathfrak{g})$ then  Ann$_{U(\mathfrak{g})}D(\lambda)=$ Ann$_{U(\mathfrak{g})}D(\mu)$, and using \cite[Proposition 6.2.3]{Dixmier}, it follows that $\lambda$ and $\mu$ lie in the same coadjoint orbit as required.\\

\noindent Also, if $I(\lambda)\subseteq I(\mu)$ then Ann$_{U(\mathfrak{g})}D(\lambda)\subseteq$ Ann$_{U(\mathfrak{g})}D(\mu)$. But Ann$_{U(\mathfrak{g})}D(\lambda)$ is a weakly rational ideal of $U(\mathfrak{g})$, and hence it is maximal by \cite[Proposition 4.7.4]{Dixmier}, therefore $I(\lambda)\cap U(\mathfrak{g})=$ Ann$_{U(\mathfrak{g})}D(\lambda)=$ Ann$_{U(\mathfrak{g})}D(\mu)=I(\mu)\cap U(\mathfrak{g})$. \end{proof} 

\subsection{The Coadjoint action}

Now, let $\lambda,\mu:\mathfrak{g}\to F$ be linear forms such that $\lambda(\mathcal{L}),\mu(\mathcal{L})\subseteq\mathcal{O}_F$ for some finite extension $F/K$. We want to compare the Dixmier annihilators $I(\lambda)$ and $I(\mu)$ in $\widehat{U(\mathcal{L})}_K$ in the case where $\lambda$ and $\mu$ lie in the same coadjoint orbit, i.e. $\mu=g\cdot\lambda$ for some $g\in\mathbb{G}(F)$. Explicitly, $g=\exp(\ad(u))$ for some $u\in\mathfrak{g}\otimes_K F$.\\ 

\noindent Our first results ensure that it is sufficient to consider the case where $F=K$.

\begin{proposition}\label{scalar-extension}

Let $F/K$ be a finite extension, and let $\lambda:\mathfrak{g}\to K$ be $K$-linear. Then there exists a polarisation $\mathfrak{b}$ of $\mathfrak{g}$ at $\lambda$ such that $\mathfrak{b}\otimes_K F$ is a polarisation for $\mathfrak{g}\otimes_K F$ at the extension $\lambda_F:\mathfrak{g}\otimes_K F\to F$. 

\end{proposition}

\begin{proof}

Using induction on $\dim(\mathfrak{g})$. If $\dim(\mathfrak{g})=1$ then it is obvious, because $\mathfrak{g}$ and $\mathfrak{g}\otimes F$ are the only polarisations. So suppose the result holds whenever $\dim(\mathfrak{g})<n$.\\

\noindent If $\lambda(\mathfrak{a})=0$ for some non-zero ideal $\mathfrak{a}$ of $\mathfrak{g}$, then using induction we may choose a polarisation $\mathfrak{a}\subseteq\mathfrak{b}$ such that $\frac{\mathfrak{b}}{\mathfrak{a}}\otimes_K F=\frac{\mathfrak{b}\otimes_K F}{\mathfrak{a}\otimes_K F}$ is a polarisation for $\frac{\mathfrak{g}\otimes_K F}{\mathfrak{a}\otimes_K F}$ at $\lambda_F$. Hence $\mathfrak{b}\otimes_K F$ is a polarisation for $\mathfrak{g}\otimes_K F$ at $\lambda_F$.\\

\noindent So from now on, we may assume that $\lambda(\mathfrak{a})=0$ for all non-zero ideals $\mathfrak{a}$ of $\mathfrak{g}$. Then it follows from Proposition \ref{sub-polarisation} that $\mathfrak{g}$ has a reducing quadruple $(x,y,z,\mathfrak{g}')$, and $\lambda(z)\neq 0$. Clearly $(x\otimes 1,y\otimes 1,z\otimes 1,\mathfrak{g}'\otimes_K F)$ is a reducing quadruple for $\mathfrak{g}\otimes_K F$.\\

\noindent Let $\mathfrak{b}$ be a polarisation for $\mathfrak{g}'$ at $\lambda|_{\mathfrak{g}'}$ such that $\mathfrak{b}\otimes_K F$ is a polarisation for $\mathfrak{g}'\otimes_K F$. Then using Lemma \ref{sub-polarisation} we see that $\mathfrak{b}$ is a polarisation for $\mathfrak{g}$ at $\lambda$, and $\mathfrak{b}\otimes_K F$ is a polarisation for $\mathfrak{g}\otimes_K F$ at $\lambda_F$ as required.\end{proof}

\begin{corollary}\label{ind-extension}

Let $F/K$ be a finite extension, and let $\lambda:\mathfrak{g}\to F$ be $K$-linear such that $\lambda(\mathcal{L})\subseteq\mathcal{O}_F$. Then for any finite extension $L/F$, $I(\lambda)_F=I(\lambda)_L$.

\end{corollary}

\begin{proof} 

This is immediate from Proposition \ref{scalar-extension} and Theorem \ref{independence}.\end{proof}

\noindent So, from now on, we will assume that $\lambda$ and $\mu$ take values in $K$, and $\mu=\exp(\ad(u))\cdot\lambda$ for some $u\in\mathfrak{g}$. Let $\sigma:=\exp(\ad(u))\in\mathbb{G}(K)$, and fix a natural number $N\in\mathbb{N}$ such that $u\in p^{-N}\mathcal{L}$. Also let $c$ be the nilpotency class of $\mathfrak{g}$, i.e. $c$ is minimal such that $\ad(\mathfrak{g})^c=0$.

Since $\sigma$ is a Lie automorphism of $\mathfrak{g}$, it follows that $\sigma\mathcal{L}$ is an $\mathcal{O}$-Lie lattice in $\mathfrak{g}$, hence there exists a natural number $n\in\mathbb{N}$ such that $p^n\mathcal{L}\subseteq \sigma\mathcal{L}$ and $p^n\sigma\mathcal{L}\subseteq\mathcal{L}$.

\begin{lemma}\label{bound}

For any $n\geq cN+v_p(c!)$, $p^n\mathcal{L}\subseteq \sigma\mathcal{L}$ and $p^n\sigma\mathcal{L}\subseteq\mathcal{L}$.

\end{lemma}

\begin{proof}

Since $\sigma=\exp(\ad(u))$, where $u=p^{-N}v$ for some $v\in\mathcal{L}$, it follows that for all $w\in\mathcal{L}$:

\begin{equation}
\sigma(w)=w+p^{-N}[v,w]+\frac{1}{2}p^{-2N}[v,[v,w]]+\cdots+\frac{1}{c!}p^{-cN}(ad(v))^c(w)
\end{equation}

\noindent But for each $0\leq i\leq c$, $(\ad(v))^i(w)\in\mathcal{L}$, $v_p(\frac{1}{i!}p^{-iN})=-iN-v_p(i!)\geq -cN-v_p(c!)\geq -n$, so $\frac{1}{i!}p^{-iN}(\ad(v))^i(w)\in p^{-n}\mathcal{L}$. Hence $\sigma\mathcal{L}\subseteq p^{-n}\mathcal{L}$, and $p^n\sigma\mathcal{L}\subseteq\mathcal{L}$.\\

\noindent Also, $\sigma$ is an isomorphism, and $\sigma^{-1}=\exp(\ad(-u))$, with $-u\in p^{-N}\mathcal{L}$. So since $\sigma^{-1}:\sigma\mathcal{L}\to\mathcal{L}$ is a Lie-isomorphism, it follows from the above discussion that $p^n\mathcal{L}\subseteq \sigma\mathcal{L}$.\end{proof}

\noindent It is clear that since $\sigma:\mathcal{L}\to \sigma\mathcal{L}$ is a continuous isomorphism of $\mathcal{O}$-Lie lattices, it extends to a continuous isomorphism $\sigma:\widehat{U(\mathcal{L})}_K\to\widehat{U(\sigma\mathcal{L})}_K$ of $K$-algebras. Moreover, for any $n\in\mathbb{N}$, $\sigma$ induces an isomorphism $\sigma:\widehat{U(p^n\mathcal{L})}_K\to\widehat{U(p^n\sigma\mathcal{L})}_K$, and thus using Lemma \ref{bound}, for $n\geq cN+v_p(c!)$, there is an injective $K$-algebra homomorphism $\sigma:\widehat{U(p^n\mathcal{L})}_K\to\widehat{U(\mathcal{L})}_K$.

\begin{proposition}\label{preserve}

Given $n\in\mathbb{N}$ such that $n\geq cN+v_p(c!)$, if $I$ is a two-sided ideal of $\widehat{U(\mathcal{L})}_K$, then $\sigma:\widehat{U(p^n\mathcal{L})}_K\to\widehat{U(\mathcal{L})}_K$ maps $I\cap\widehat{U(p^n\mathcal{L})}_K$ into $I$.

\end{proposition}

\begin{proof}

Consider the sequence of continuous $\mathcal{O}$-linear maps $\sigma_i:=\underset{0\leq j\leq i}{\sum}{\frac{1}{j!}(\ad(u))^j}:\widehat{U(p^n\mathcal{L})}_K\to\widehat{U(\mathcal{L})}_K$. Clearly each of these sends $I\cap\widehat{U(p^n\mathcal{L})}_K$ into $I$.\\

\noindent We will show that $\sigma_i$ converges pointwise to $\sigma$ as $i\rightarrow\infty$, and since all ideals in $\widehat{U(\mathcal{L})}_K$ are closed, it will follow that $\sigma(I\cap\widehat{U(p^n\mathcal{L})}_K)\subseteq I$.\\

\noindent Let $\delta:=\ad(u)$, and let $v$ be the $p$-adic filtration on $\widehat{U(\mathcal{L})}_K$ induced from $\widehat{U(\mathcal{L})}$. Then for all $u\in\mathcal{L}$, $v(\delta(u))\geq v(u)-N$.\\

\noindent Since $\delta$ is a derivation, a standard inductive argument shows that for all $a_1,\cdots,a_r\in\widehat{U(\mathcal{L})}_K$:

\begin{equation}\label{der-expansion}
\underset{0\leq j\leq i}{\sum}{\frac{1}{j!}\delta^j(a_1a_2\cdots a_r)}=\underset{0\leq j\leq i}{\sum}{\left(\underset{j_1+\cdots+j_r=j}{\sum}{\left(\underset{1\leq m\leq r}{\prod}{\frac{1}{j_m!}\delta^{j_m}(a_j)}\right)}\right)}
\end{equation}

\noindent So, if $x\in\widehat{U(\mathcal{L})}_K$, then $x=\underset{(r,u)}{\sum}{\lambda_u u_1\cdots u_r}$, where the sum is taken over all $r\geq 0$, $u=u_1\cdots u_r$ for $u_i\in\mathcal{L}$, and $v_p(\lambda_u)-nr\rightarrow\infty$ as $r\rightarrow\infty$. Therefore, fixing $t\in\mathbb{N}$, we have:

\begin{equation}
(\sigma-\sigma_t)(x)=\underset{(r,u)}{\sum}{\lambda_u(\sigma-\sigma_t)(u_1\cdots u_r)}=\underset{(r,u)}{\sum}{\lambda_u\left(\underset{j>t}{\sum}{\left(\underset{j_1+\cdots+j_r=j}{\sum}{\left(\underset{1\leq m\leq r}{\prod}{\frac{1}{j_m!}\delta^{j_m}(u_m)}\right)}\right)}\right)}
\end{equation}

\noindent For each $r\geq 0$, let $A_r:=\{\alpha\in [c]^r:\vert\alpha\vert>t\}$, where $[c]=\{0,\cdots,c-1\}$. Note that $\delta^c(u)=0$ for all $u\in\mathcal{L}$.\\

\noindent Then $(\sigma-\sigma_t)(x)=\underset{(r,u)}{\sum}{\lambda_u\left(\underset{\alpha\in A_r}{\sum}{\underset{1\leq m\leq r}{\prod}{\frac{1}{\alpha_m!}\delta^{\alpha_m}(u_m)}}\right)}$, and since $A_r$ is finite, $\underset{\alpha\in A_r}{\sum}{\underset{1\leq m\leq r}{\prod}{\frac{1}{\alpha_m!}\delta^{\alpha_m}(u_m)}}$ is a finite sum.\\

\noindent Since $\alpha_m<c$ for all $\alpha\in A_r$, we have that $v_p(\alpha_m!)\leq v_p(c!)$. Also, since $v(\delta(u))\geq v(u)-N$, it follows that $v(\delta^{\alpha_m}(u))\geq v(u)-\alpha_m N$. Therefore $v(\frac{1}{\alpha_m!}\delta^{\alpha_m}(u_m))\geq v(u_m)-\alpha_mN-v_p(c!)$ for all $m\leq r$.\\

\noindent Thus for each pair $(r,u)$, $v(\underset{\alpha\in A_r}{\sum}{\underset{1\leq m\leq r}{\prod}{\frac{1}{\alpha_m!}\delta^{\alpha_m}(u_m)}})\geq v(u_1)+\cdots+v(u_r)-(\vert\alpha\vert N+v_p(c!)r)\geq -(\vert\alpha\vert N+rv_p(c!))\geq -r(cN+v_p(c!))$, the last inequality follows since $\vert\alpha\vert\leq rc$.\\

\noindent Therefore, $v\left(\lambda_u\left(\underset{\alpha\in A_r}{\sum}{\underset{1\leq m\leq r}{\prod}{\frac{1}{\alpha_m!}\delta^{\alpha_m}(u_m)}}\right)\right)\geq v_p(\lambda_u)-r(cN+v_p(c!))\geq v_p(\lambda_u)-nr\rightarrow\infty$ as $r\rightarrow\infty$.\\

\noindent Moreover, for $r\leq\frac{t}{c}$, $A_r=\emptyset$, so we have:

\begin{equation}
(\sigma-\sigma_t)(x)=\underset{(r,u),r>\frac{t}{c}}{\sum}{\lambda_u\left(\underset{\alpha\in A_r}{\sum}{\underset{1\leq m\leq r}{\prod}{\frac{1}{\alpha_m!}\delta^{\alpha_m}(u_m)}}\right)}
\end{equation}

\noindent Therefore, $v((\sigma-\sigma_t)(x))\geq\inf\{v_p(\lambda_u)-nr:u=u_1\cdots u_r$ with $r>\frac{t}{c}\}$, and this tends to infinity as $t\rightarrow\infty$. Hence $(\sigma-\sigma_t)(x)\rightarrow 0$ as $t\rightarrow\infty$.\\

\noindent So $\sigma(x)=\underset{t\rightarrow\infty}{\lim}{\sigma_t(x)}$, so if $x\in I$ then $\sigma(x)\in I$ as required.\end{proof}

\vspace{0.2in}

\noindent Now, let $\mathfrak{b}$ be a polarisation for $\mathfrak{g}$ at $\lambda$, and let $\mathfrak{b}'$ be a polarisation for $\mathfrak{g}$ at $\mu$. Since $\mu=\sigma\cdot\lambda$, it follows that $\sigma\mathfrak{b}$ is also a polarisation for $\mathfrak{g}$ at $\mu$. Also, it is clear that $\sigma\mathfrak{b}\cap \sigma\mathcal{L}=\sigma(\mathfrak{b}\cap\mathcal{L})$, so let $\mathcal{B}:=\mathfrak{b}\cap\mathcal{L}$ and $\mathcal{B}':=\mathfrak{b}'\cap\mathcal{L}$.\\

\noindent Consider the Dixmier modules $\widehat{D(\lambda)}_{\mathcal{B}}:=\widehat{U(\mathcal{L})}_K\otimes_{\widehat{U(\mathcal{B})}_K}K_{\lambda}$, $\widehat{D(\mu)}_{\sigma\mathcal{B}}:=\widehat{U(\sigma\mathcal{L})}_K\otimes_{\widehat{U(\sigma\mathcal{B})}_K}K_{\mu}$, $\widehat{D(\mu)}_{\mathcal{B}'}:=\widehat{U(\mathcal{L})}_K\otimes_{\widehat{U(\mathcal{B}')}_K}K_{\mu}$.\\

\noindent Then $\widehat{D(\lambda)}_{\mathcal{B}}$ and $\widehat{D(\mu)}_{\mathcal{B}'}$ are $\widehat{U(\mathcal{L})}_K$-modules, topological completions of the $U(\mathfrak{g})$-modules $D(\lambda)_{\mathfrak{b}}$ and $D(\mu)_{\mathfrak{b}'}$ respectively, while $\widehat{D(\mu)}_{\sigma\mathcal{B}}$ is a $\widehat{U(\sigma\mathcal{L})}_K$-module, a topological completion of $D(\mu)_{\sigma\mathfrak{b}}$.\\

\noindent Let $I(\mu):=Ann_{\widehat{U(\mathcal{L})}_K}\widehat{D(\mu)}_{\mathcal{B}'}\trianglelefteq \widehat{U(\mathcal{L})}_K$, and let $I'(\mu):=Ann_{\widehat{U(\sigma\mathcal{L})}_K}\widehat{D(\mu)}_{\sigma\mathcal{B}}\trianglelefteq\widehat{U(\sigma\mathcal{L})}_K$.

\begin{lemma}\label{I=I'}

$I'(\mu)=\sigma(I(\lambda))$, and given $n\in\mathbb{N}$ such that $p^n\mathcal{L}\subseteq\sigma\mathcal{L}$, $I(\mu)\cap\widehat{U(p^n\mathcal{L})}_K=I'(\mu)\cap\widehat{U(p^n\mathcal{L})}_K$.

\end{lemma}

\begin{proof}

Since $\sigma:\mathcal{L}\to\sigma\mathcal{L}$ is a Lie isomorphism, it follows from Lemma \ref{transport-structure} that $I'(\mu)=\sigma(I(\lambda))$.\\ 

\noindent Let $\mathcal{C}:=\sigma\mathfrak{b}\cap p^n\mathcal{L}$, then the $\widehat{U(p^n\mathcal{L})}_K$-affinoid Dixmier module $\widehat{D(\mu)}_{\mathcal{C}}:=\widehat{U(p^n\mathcal{L})}_K\otimes_{\widehat{U(\mathcal{C})}_K}K_{\mu}$ embeds densely into $\widehat{D(\mu)}_{\sigma\mathcal{B}}$. So if $x\in \widehat{U(p^n\mathcal{L})}_K$ then $x\widehat{D(\mu)}_{\mathcal{C}}=0$ if and only if $x\widehat{D(\mu)}_{\sigma\mathcal{B}}=0$.\\

\noindent But since $\sigma\mathfrak{b}$ and $\mathfrak{b}'$ are polarisations of $\mathfrak{g}$ at $\mu$ with $\sigma\mathfrak{b}\cap p^n\mathcal{L}=C$ and $\mathfrak{b}'\cap p^n\mathcal{L}=p^n\mathcal{B}'$, we can apply Theorem \ref{independence} to get that $Ann_{\widehat{U(p^n\mathcal{L})}_K}\widehat{D(\mu)}_{\mathcal{C}}=Ann_{\widehat{U(p^n\mathcal{L})}_K}\widehat{D(\mu)}_{p^n\mathcal{B}'}$. Therefore, given $x\in\widehat{U(p^n\mathcal{L})}_K$:

\begin{center}
$x\widehat{D(\mu)}_{\sigma\mathcal{B}}=0 \iff x\widehat{D(\mu)}_{\mathcal{C}}=0 \iff \widehat{D(\mu)}_{p^n\mathcal{B}'}=0 \iff x\widehat{D(\mu)}_{\mathcal{B}'}=0$
\end{center}

\noindent Therefore $I(\mu)\cap\widehat{U(p^n\mathcal{L})}_K=I'(\mu)\cap\widehat{U(p^n\mathcal{L})}_K$ as required.\end{proof}

\noindent Now we can prove the main result of this section, which allows us to compare Dixmier annihilators for $\lambda,\mu\in\mathfrak{g}^*$ in the same coadjoint orbit:

\begin{theorem}\label{orbit-restriction}

Let $\mathfrak{g}$ be a nilpotent $K$-Lie algebra, with nilpotency class $c$, and let $\mathcal{L}$ be an $\mathcal{O}$-Lie lattice in $\mathfrak{g}$. Let $\lambda,\mu:\mathfrak{g}\to\overline{K}$ be $K$-linear maps such that $\lambda(\mathcal{L})\subseteq\mathcal{O}_{\overline{K}}$, and suppose that $\mu=\exp(\ad(u))\cdot\lambda$ for some $u\in p^{-N}(\mathcal{L}\otimes_{\mathcal{O}}\mathcal{O}_{\overline{K}})$. Then given $n\in\mathbb{N}$ such that $n\geq Nc+v_p(c!)$, $I(\lambda)\cap\widehat{U(p^{2n}\mathcal{L})}_K=I(\mu)\cap\widehat{U(p^{2n}\mathcal{L})}_K$.

\end{theorem}

\begin{proof}

Firstly, note that $u$ lies in $p^{-N}(\mathcal{L}\otimes_{\mathcal{O}}\mathcal{O}_F)$ for some finite extension $F$ of $K$, and we may assume further that $\lambda$ and $\mu$ take values in $F$, possibly after extending $F$.\\

\noindent Using Corollary \ref{ind-extension}, we see that $I(\lambda)=I(\lambda)_F$, i.e. $I(\lambda)=\Ann_{\widehat{U(\mathcal{L})}_K}\widehat{D(\lambda)}_F$, and similarly $I(\mu)=\Ann_{\widehat{U(\mathcal{L})}_K}\widehat{D(\mu)}_F$, so we may safely assume that $F=K$. So $\lambda$ and $\mu$ are $K$-linear forms of $\mathfrak{g}$, $u\in p^{-N}\mathcal{L}$, and setting $\sigma:=\exp(\ad(u))$, since $n\geq Nc+v(c!)$ we see using Lemma \ref{bound} that $p^n\mathcal{L}\subseteq \sigma\mathcal{L}$ and $p^n\sigma\mathcal{L}\subseteq\mathcal{L}$.\\

\noindent We will prove that $I(\mu)\cap\widehat{U(p^{2n}\mathcal{L})}_K\subseteq I(\lambda)$, and after replacing $\sigma$ by $\sigma^{-1}=\exp(\ad(-u))$, it will follow that $I(\lambda)\cap\widehat{U(p^{2n}\mathcal{L})}_K\subseteq I(\mu)$ as required.\\

\noindent By Lemma \ref{I=I'}, we see that $I(\mu)\cap\widehat{U(p^{2n}\mathcal{L})}_K=\sigma(I(\lambda))\cap\widehat{U(p^{2n}\mathcal{L})}_K\subseteq \sigma(I(\lambda))\cap \widehat{U(p^{n}\sigma\mathcal{L})}_K=\sigma(I(\lambda)\cap \widehat{U(p^{n}\mathcal{L})}_K)$.\\

\noindent But since $I(\lambda)$ is a two-sided ideal of $\widehat{U(\mathcal{L})}_K$, it follows from Proposition \ref{preserve} that $\sigma(I(\lambda)\cap\widehat{U(p^{n}\mathcal{L})}_K)\subseteq I(\lambda)$, and hence $I(\mu)\cap\widehat{U(p^{2n}\mathcal{L})}_K\subseteq I(\lambda)$ as required.\end{proof}

\subsection{Proof of Theorem \ref{A}}

\noindent Now, let $P$ be a weakly rational ideal of $\widehat{U(\mathcal{L})}_K$. Using Theorem \ref{aff-Dix2} and Proposition \ref{prime-intersection}, we see that $P$ arises an an intersection of Dixmier annihilators:

\begin{center}
$P=\underset{j\in X}{\bigcap}{I(\lambda_j)}$
\end{center}

\noindent for some $\lambda_j:\mathfrak{g}\to\overline{K}$ $K$-linear, such that $\lambda_j(\mathcal{L})\subseteq\mathcal{O}_{\overline{K}}$ for each $j$.\\

\noindent Also, $P\cap U(\mathfrak{g})$ is a weakly rational ideal of $U(\mathfrak{g})$ and hence it is maximal by \cite[Proposition 4.7.4]{Dixmier}. So since $P\cap U(\mathfrak{g})\subseteq I(\lambda_j)\cap U(\mathfrak{g})$ for all $j$, it follows that $I(\lambda_j)\cap U(\mathfrak{g})=I(\lambda_k)\cap U(\mathfrak{g})$ for all $j,k\in X$. 

Therefore, using Lemma \ref{orbit}, this means that for every $j,k\in X$, there exists $a_{j,k}\in\mathbb{G}(\overline{K})$ such that $a_{j,k}\cdot\lambda_j=\lambda_k$, i.e. all $\lambda_j$ lie in the same coadjoint orbit.

\begin{proposition}\label{final}

Let $\lambda:\mathfrak{g}\to\overline{K}$ be a $K$-linear map such that $\lambda(\mathcal{L})\subseteq\mathcal{O}_{\overline{K}}$. Then there exists an integer $N\geq 0$ such that for any linear form $\mu:\mathfrak{g}\to\overline{K}$ in the $\mathbb{G}$-coadjoint orbit of $\lambda$ with $\mu(\mathcal{L})\subseteq\mathcal{O}_{\overline{K}}$, $\mu=\exp(\ad(u))\cdot\lambda$ for some $u\in p^{-N}(\mathcal{L}\otimes_{\mathcal{O}}\mathcal{O}_{\overline{K}})$.

\end{proposition}

\begin{proof}

Let $Y$ be the coadjoint orbit in $\mathfrak{g}^*=$ Hom$_K(\mathfrak{g},\overline{K})$ containing $\lambda$, and recall from Lemma \ref{product-decomposition} that there exists an affine algebraic subgroup $S$ of $\mathbb{G}$ such that $\mathbb{G}\cong S\times Y$ as varieties, where the natural morphism $\mathbb{G}\to Y,g\mapsto g\cdot\lambda$ is just the projection to the second factor. Consider the following sequence in the category of $\overline{K}$-varieties defined over $K$:

\begin{equation}\label{variety'}
\ad(\mathfrak{g})\underset{\exp}{\rightarrow}\mathbb{G}\cong S\times Y\rightarrow Y.
\end{equation}

\noindent Let $U$ be the set of all $\mu\in Y$ such that $\mu(\mathcal{L})\subseteq\mathcal{O}_{\overline{K}}$. Then $U$ is an affinoid subdomain of $Y$ (when $Y$ is considered as a rigid variety) isomorphic to Sp $\widehat{S(\mathfrak{g})}$, where $\widehat{S(\mathfrak{g})}$ is the $\pi$-adic completion of the symmetric algebra $S(\mathfrak{g})$ with respect to the lattice $S(\mathcal{L})$. Since $\exp$ is an isomorphism, we may take the inverse image $V:=\exp^{-1}(1\times U)$ of $1\times U$, which will be an affinoid subdomain of $\ad(\mathfrak{g})$.\\

\noindent But $\ad(\mathfrak{g})\cong\mathfrak{g}/Z(\mathfrak{g})$ is a union of open discs containing $p^{-n}(\mathcal{L}/Z(\mathcal{L}))$ for $n\in\mathbb{N}$. So since $V$ is affinoid, it follows that $V$ is contained in $p^{-N}(\mathcal{L}/Z(\mathcal{L})\otimes_{\mathcal{O}}\mathcal{O}_{\overline{K}})$ for some $N\in\mathbb{N}$.

Therefore, for any $\mu\in U$, we can choose $u\in p^{-N}(\mathcal{L}\otimes_{\mathcal{O}}\mathcal{O}_{\overline{K}})$ such that $\mu$ is the image of $\ad(u)$ under the composition $\ad(\mathfrak{g})\to \mathbb{G}\to Y$, i.e. $\mu=\exp(\ad(u))\cdot\lambda$ as required.\end{proof}

\vspace{0.2in}

\noindent Now we can finally prove our main theorem:\\

\noindent\emph{Proof of Theorem \ref{A}.} Let $P$ be a weakly rational ideal of $\widehat{U(\mathcal{L})}_K$, then since $P$ is prime, we see using Corollary \ref{Dix-intersection} that $P=\underset{j\in X}{\bigcap}{I(\lambda_j)}$ for linear forms $\lambda_j$ all lying in the same coadjoint orbit. 

Since $\lambda_j(\mathcal{L})\subseteq\mathcal{O}_{\overline{K}}$ for each $j$, it follows from Proposition \ref{final} that we can choose $N\in\mathbb{N}$, $u_{j,k}\in p^{-N}(\mathcal{L}\otimes_{\mathcal{O}}\mathcal{O}_{\overline{K}})$ for each $j,k\in X$ such that $\lambda_k=\exp(\ad(u_{j,k}))\cdot\lambda_j$.\\

\noindent Therefore, let $c$ be the nilpotency class of $\mathfrak{g}$, and choose $n\in\mathbb{N}$ with $n\geq 2Nc+2v(c!)$. Then using Theorem \ref{orbit-restriction}, we see that $I(\lambda_j)\cap\widehat{U(p^n\mathcal{L})}_K=I(\lambda_k)\cap\widehat{U(p^n\mathcal{L})}_K$ for each $j,k\in X$.\\

\noindent Therefore, $P\cap\widehat{U(p^n\mathcal{L})}_K=\underset{j\in X}{\bigcap}{I(\lambda_j)\cap\widehat{U(p^n\mathcal{L})}_K}=I(\lambda_j)\cap\widehat{U(p^n\mathcal{L})}_K$ for any $j\in X$. Hence $P\cap\widehat{U(p^n\mathcal{L})}_K=\Ann_{\widehat{U(p^n\mathcal{L})}_K}\widehat{D(\lambda_j)}$ is a Dixmier annihilator as required.\qed\\

\noindent We suspect that we can always take $n=0$ in the statement of Theorem \ref{A}, but we do not have a proof of this.

\section{Special Dixmier Annihilators}

As outlined in the introduction, our aim is to prove the \emph{deformed Dixmier-Moeglin equivalence} for $\widehat{U(\mathcal{L})}_K$, and using Proposition \ref{nullstellensatz}, we see that this just means proving that if $P$ is a weakly rational ideal of $\widehat{U(\mathcal{L})}_K$ then $P\cap\widehat{U(\pi^n\mathcal{L})}_K$ is locally closed in $\widehat{U(\pi^n\mathcal{L})}_K$ for $n$ sufficiently high. In fact, in the case where $\mathcal{L}$ is nilpotent, we suspect moreover that $P\cap\widehat{U(\pi^n\mathcal{L})}_K$ is maximal.

So now that we have established that weakly rational ideals become Dixmier annihilators after passing to a suitably high deformation $\widehat{U(\pi^n\mathcal{L})}_K$, it remains to prove that Dixmier annihilator ideals $I(\lambda)$ are maximal. In this section, we will establish this in some important cases, and in particular prove Theorem \ref{B}.

\subsection{Base Change}

Fix $F/K$ a finite, Galois extension, with $G:=Gal(F/K)$, let $R$ be a $K$-algebra and let $S:=R\otimes_K F$. There is a natural action of $G$ on $S$ by automorphisms via $\sigma\cdot(r\otimes\alpha)=r\otimes \sigma(\alpha)$, and clearly $R$ is invariant under this action.

We will make a further assumption on $S$, namely that if $\alpha_1,\cdots,\alpha_n\in F$ are linearly independent over $K$ and $r_1\otimes\alpha_1+\cdots+r_n\otimes\alpha_n=0$ in $S$ then $r_i=0$ for all $i$. Using Lemma \ref{PBW}, it is easily verified that this property is satisfied if we take $R$ to be the affinoid enveloping algebra $\widehat{U(\mathcal{L})}_K$.

\begin{proposition}\label{Galois-ext}

Let $I$ be a $G$-invariant ideal of $S$, and let $J:=I\cap R$. Then $I=J\otimes_KF$. In fact, if $r_1\otimes\alpha_1+\cdots+r_n\otimes\alpha_n\in I$ with $\alpha_1,\cdots,\alpha_n$ linearly independent over $K$, then $r_1,\cdots,r_n\in J$.

\end{proposition}

\begin{proof}

For each $s\in S$, let $l(s)$ be the minimal number of simple tensors required to sum to $s$, so $l(s)=0$ if and only if $s=0$. If $s\in I$ and $l(s)=1$ then $s=r\otimes\alpha$ is a simple tensor, with $\alpha\neq 0$, thus $r=\alpha^{-1}s\in I\cap R=J$. Hence $s\in J\otimes_K F$.

So, we will assume, for induction, that for all $s\in I$ with $l(s)<n$, $s\in J\otimes_K F$, and moreover that if $s=r_1\otimes\alpha_1+\cdots+r_m\otimes\alpha_m$ with $m<n$ and $\alpha_1,\cdots,\alpha_m$ linearly independent over $K$, then $r_i\in J$ for all $i$.\\

\noindent Now, suppose $s\in  I$ with $l(s)=n>1$. Then choose a subfield $F'\subseteq F$ with $K\subseteq F'$ such that $F'$ is minimal with respect to the property that there exist $r_1,\cdots,r_n\in R$, $\alpha_1,\cdots,\alpha_n\in F'$ such that $s=r_1\otimes\alpha_1+\cdots+r_n\otimes\alpha_n$. Since $l(s)=n$, it follows immediately that $\alpha_1,\cdots,\alpha_n$ are linearly independent over $K$, and note that since $\alpha_1^{-1}s\in I$ and $l(s)=l(\alpha_1^{-1}s)$, we may assume that $\alpha_1=1$.\\

\noindent If $F'=K$ then $s\in R$ so $l(s)=1$, a contradiction. Hence $F'\neq K$, so since $F/K$ is a Galois extension, there exists $\sigma\in G$ such that the set $F'':=\{\alpha\in F':\sigma(\alpha)=\alpha\}\neq F'$. Clearly $F''$ is a subfield of $F'$ with $K\subseteq F''$.\\

\noindent Since $\alpha_1=1\in F''$, choose $m\geq 1$ maximal such that there exists $t_1,\cdots,t_n\in R$, $\beta_1,\cdots,\beta_n\in F'$ with $\beta_1,\cdots,\beta_m\in F''$ and $s=t_1\otimes\beta_2+\cdots+t_n\otimes\beta_n$. Clearly $m<n$ by minimality of $F'$ and the fact that $F''\subsetneq F'$.\\

\noindent Now, since $I$ is $G$-invariant, $\sigma(s)-s\in I$. But $\sigma(s)-s=t_1\otimes (\sigma(\beta_1)-\beta_1)+\cdots+t_n\otimes(\sigma(\beta_n)-\beta_n)=t_{m+1}\otimes (\sigma(\beta_{m+1})-\beta_{m+1})+\cdots+t_n\otimes(\sigma(\beta_n)-\beta_n)$.\\

\noindent If $t_i\in J$  for all $i>m$ then $t_1\otimes\beta_1+\cdots+t_m\otimes\beta_m\in I$, and hence since $m<n$ and $\beta_1,\cdots,\beta_m$ are linearly independent over $K$, $t_i\in J$ for all $i$ by induction, and hence $s\in J\otimes_K F$. 

Therefore, we will assume for contradiction that $t_i\notin J$ for some $i>m$, so since $\sigma(s)-s\in I$, it follows from induction that $\sigma(\beta_{m+1})-\beta_{m+1},\cdots,\sigma(\beta_n)-\beta_n$ are linearly dependent over $K$.\\

\noindent So, there exist $\gamma_{m+1},\cdots,\gamma_n\in K$, not all zero, such that $\sigma(\gamma_{m+1}\beta_{m+1}+\cdots+\gamma_n\beta_n)=\gamma_{m+1}\beta_{m+1}+\cdots+\gamma_n\beta_n$. So let $\beta':=\gamma_{m+1}\beta_{m+1}+\cdots+\gamma_n\beta_n\in F'$, and clearly $\beta'\in F''$ by the definition of $F''$. We may assume without loss of generality that $\gamma_{m+1}\neq 0$. Thus:\\ 

\noindent $s=t_1\otimes\beta_2+\cdots+t_n\otimes\beta_n$\\

$=t_1\otimes\beta_1+\cdots+t_m\otimes\beta_m+\gamma_{m+1}^{-1}t_{m+1}\otimes(\beta'-\gamma_{m+2}\beta_{m+2}-\cdots-\gamma_n\beta_n)+t_{m+2}\otimes\beta_{m+2}+$

$\cdots+t_n\otimes\beta_n$\\

$=t_1\otimes \beta_1+\cdots+t_m\otimes\beta_m+\gamma_{m+1}^{-1}t_{m+1}\otimes\beta'+(t_{m+2}-\gamma_{m+1}^{-1}\gamma_{m+2}t_{m+1})\otimes\beta_{m+2}+\cdots+$

$(t_n-\gamma_{m+1}^{-1}\gamma_nt_n)\otimes\beta_n$.\\

\noindent But since $\beta_1,\cdots,\beta_m,\beta'\in F''$, this contradicts the maximality of $m$. Therefore $s\in J\otimes_K F$.\\

\noindent Finally, to complete the induction, suppose that $s=r_1\otimes\alpha_1+\cdots+r_n\otimes\alpha_n\in I$ with $\alpha_1,\cdots,\alpha_n$ linearly independent over $K$. Then $l(s)\leq n$ so we have proved that $s\in J\otimes_K F$, so there exist $t_1,\cdots, t_m\in J$, $\beta_1,\cdots,\beta_m\in F$ such that $s=t_1\otimes\beta_1+\cdots+t_m\otimes\beta_m$.\\

\noindent So, extend $\{\alpha_1,\cdots,\alpha_n\}$ to a basis $\{\alpha_1,\cdots,\alpha_d\}$ of $F/K$, and set $r_{n+1}=\cdots=r_d=0$. Then for each $i=1\cdots,m$, $\beta_i=\underset{1\leq j\leq d}{\sum}{\gamma_{i,j}\alpha_j}$ for some $\gamma_{i,j}\in K$, and:

$s=\underset{1\leq i\leq m}{\sum}{t_i\otimes\beta_i}=\underset{1\leq i\leq m}{\sum}{\underset{1\leq j\leq d}{\sum}{\gamma_{i,j}t_i\otimes\alpha_j}}=\underset{1\leq j\leq d}{\sum}{\left(\underset{1\leq i\leq m}{\sum}{\gamma_{i,j}t_i}\right)\otimes\alpha_j}$.\\

\noindent But we know that $s=\underset{1\leq j\leq d}{\sum}{r_j\otimes\alpha_j}$, and hence $\underset{1\leq j\leq d}{\sum}{\left(r_j-\underset{1\leq i\leq m}{\sum}{\gamma_{i,j}t_i}\right)}\otimes\alpha_j=0$, and this implies that $r_j=\underset{1\leq i\leq m}{\sum}{\gamma_{i,j}t_i}\in J$ for each $j$ as required.\end{proof}

\begin{corollary}\label{min-prime}

If we assume $S$ has finite left and right Krull dimension, and $P$ be is prime ideal of $S$, then $\{\sigma(P):\sigma\in G\}$ form a complete set of minimal prime ideals above $(P\cap R)\otimes_K F$.

\end{corollary}

\begin{proof}

Let $Q=P\cap R$. Then for every $\sigma\in G$, $\sigma(Q)=Q$, and hence $Q\subseteq\sigma(P)$. So, let $I:=\underset{\sigma\in G}{\cap}{\sigma(P)}$, then $I$ is a $G$-invariant two-sided ideal of $S$, and $I\cap R=Q$, thus $I=Q\otimes_KF$ by Proposition \ref{Galois-ext}.\\

\noindent If $P'$ is a prime ideal of $S$ with $I\subseteq P'\subseteq P$, then $\underset{\sigma\in G}{\prod}{\sigma(P)}\subseteq I\subseteq P'$, which implies that $\sigma(P)\subseteq P'\subseteq P$ for some $\sigma\in G$, so $P\subseteq \sigma^{-1}(P)$.\\

\noindent But $S$ has finite Krull dimension, so it follows from \cite[Lemma 6.4.5]{McConnell} that $S$ has finite classical Krull dimension, and hence there exists $n\geq 0$ maximal such that there exists a chain of prime ideals $P=P_0\subsetneq P_1\subsetneq\cdots\subsetneq P_n$ of $S$. But $P\subseteq\sigma^{-1}(P)\subsetneq\sigma^{-1}(P_1)\subsetneq\cdots\subsetneq\sigma^{-1}(P_n)$ is also a chain of prime ideals, and thus $P=\sigma^{-1}(P)$ and $\sigma(P)=P$.

But since $\sigma(P)\subseteq P'\subseteq P$, this implies that $P=P'$, and hence $P$ is minimal prime above $I$ as required. The same argument shows that $\sigma(P)$ is minimal prime above $I$ for every $\sigma\in G$.\end{proof}

\noindent In particular, we could take $R=\widehat{U(\mathcal{L})}_K$, in which case $S=\widehat{U(\mathcal{L})}_F$ has finite Krull dimension as required. Therefore, we can apply this result in the case where $P$ is a Dixmier annihilator.\\

\noindent So, let $\lambda:\mathfrak{g}\to F$ be a linear form such that $\lambda(\mathcal{L})\subseteq\mathcal{O}_F$, and $P:=$ Ann$_{\widehat{U(\mathcal{L})}_F}\widehat{D(\lambda)}_F$, then for every $\sigma\in G$, $\sigma(P)=$ Ann$_{\widehat{U(\mathcal{L})}_F}\widehat{D(\sigma\cdot\lambda)}_F$ by Lemma \ref{transport-structure}.

\begin{proposition}\label{Dix-extension}

Let $\lambda,\mu:\mathfrak{g}\to F$ be linear forms such that $\lambda(\mathcal{L}),\mu(\mathcal{L})\subseteq\mathcal{O}_F$. Let $P=I(\lambda)\trianglelefteq\widehat{U(\mathcal{L})}_K$, $Q=I(\mu)\trianglelefteq\widehat{U(\mathcal{L})}_K$.

Then if $P\subseteq Q$, there exists $\sigma\in G$ such that if $P_{\sigma}:=$\emph{ Ann}$_{\widehat{U(\mathcal{L})}_F}\widehat{D(\sigma\cdot\lambda)}_F$ and $Q_1:=$\emph{ Ann}$_{\widehat{U(\mathcal{L})}_F}\widehat{D(\mu)}_F$ then $P_{\sigma}\subseteq Q_1$.

\end{proposition}

\begin{proof}

Firstly, let $P_1=$ Ann$_{\widehat{U(\mathcal{L})}_F}\widehat{D(\lambda)}_F$.  Then since $P_1,Q_1$ are primitive ideals of Ann$_{\widehat{U(\mathcal{L})}_F}\widehat{D(\lambda)}_F$ by Theorem \ref{aff-Dix2} and Theorem \ref{independence}, it follows from Corollary \ref{min-prime} that the set $\{\sigma(P_1):\sigma\in G\}$ consists of all minimal prime ideals above $(P_1\cap\widehat{U(\mathcal{L})}_K)\otimes_K F=P\otimes_K F$, and also that $Q_1$ is minimal prime above $Q\otimes_KF$.\\

\noindent Therefore, if $P\subseteq Q$ then $P\otimes_K F\subseteq Q\otimes_K F\subseteq Q_1$, and hence there exists $\sigma\in G$ such that $\sigma(P_1)\subseteq Q_1$. But using Lemma \ref{transport-structure} we see that $\sigma(P_1)=$ Ann$_{\widehat{U(\mathcal{L})}_F}\widehat{D(\sigma\cdot\lambda)}_F=P_{\sigma}$ as required.\end{proof}

\subsection{Dimension Theory}

Now, fix linear forms $\lambda,\mu:\mathfrak{g}\to K$ such that $\lambda(\mathcal{L}),\mu(\mathcal{L})\subseteq\mathcal{O}$, and suppose that $I(\lambda)\subseteq I(\mu)$. Let us assume further that $\mathcal{L}$ is powerful, i.e. $[\mathcal{L},\mathcal{L}]\subseteq p\mathcal{L}$.\\ 

\noindent Let $\mathfrak{a}$ be an abelian ideal of $\mathfrak{g}$, and recall from section 3.2 that we define $\mathfrak{a}^{\perp}=\{u\in\mathfrak{g}:\lambda([u,\mathfrak{a}])=0\}$. Also, using Proposition \ref{standard-polarisation}, we may fix a polarisation $\mathfrak{b}$ of $\mathfrak{g}$ at $\lambda$ such that $\mathfrak{a}\subseteq\mathfrak{b}$, and construct the affinoid Dixmier module $\widehat{D(\lambda)}_{\mathfrak{b}}$ with respect to this polarisation. Note that since $\lambda([\mathfrak{a},\mathfrak{b}])=0$, $\mathfrak{b}\subseteq\mathfrak{a}^{\perp}$.\\

\noindent Suppose $\mathfrak{a}^{\perp}$ has codimension $s$ in $\mathfrak{g}$, and fix a basis $\{u_1,\cdots,u_s\}$ for $\mathcal{L}/(\mathfrak{a}^{\perp}\cap\mathcal{L})$. Then if $\mathcal{A}:=\mathfrak{a}\cap\mathcal{L}$, it follows from Corollary \ref{completely-prime} and Proposition \ref{perp} that the action $\widehat{U(\mathcal{A})}_K\to$ End$_K\widehat{D(\lambda)}$ has image contained in $K\langle \partial_1,\cdots,\partial_s\rangle$, where $\partial_i=\frac{d}{du_i}$, and this image contains $\partial_1,\cdots,\partial_s$.

\begin{proposition}\label{dimension-theory}

The quotient $\frac{\widehat{U(\mathcal{A})}_K}{I(\lambda)\cap\widehat{U(\mathcal{A})}_K}$ has Krull dimension $s$. 

\end{proposition}

\begin{proof}

We will prove that the image $R$ of $\widehat{U(\mathcal{A})}_K$ under the action $\rho:\widehat{U(\mathcal{A})}_K\to$ End$_K\widehat{D(\lambda)}_{\mathfrak{b}}$ has Krull dimension $s$. Since the Dixmier annihilator $I(\lambda)$ does not depend on the choice of polarisation by Theorem \ref{independence}, the result will follow.\\

Note that since $\mathcal{A}$ is abelian, $\widehat{U(\mathcal{A})}_K$ is isomorphic to a commutative Tate-power series ring, and hence $R$ is a commutative affinoid algebra. Also, since $\partial_1,\cdots,\partial_s\in\rho(\widehat{U(\mathcal{A})}_K)$, it follows that there exist $m\geq 0$ such that $\pi^{m}\partial_1,\cdots,\pi^{m}\partial_s\in\rho(\widehat{U(\mathcal{A})})$, and hence every power series in $K\langle\pi^{m}\partial_1,\cdots,\pi^{m}\partial_s\rangle$ lies in $\rho(\widehat{U(\mathcal{A})}_K)=R$.\\

\noindent Therefore we have inclusions of commutative affinoid algebras, $K\langle\pi^{m}\partial_1,\cdots,\pi^{m}\partial_s\rangle\xhookrightarrow{}R\xhookrightarrow{}K\langle\partial_1,\cdots,\partial_s\rangle$, which gives rise to a chain of open embeddings of the associated affinoid spectra: Sp $K\langle\partial_1,\cdots,\partial_s\rangle\xhookrightarrow{}$ Sp $R\xhookrightarrow{}$ Sp $K\langle\pi^{m}\partial_1,\cdots,\pi^{m}\partial_s\rangle$.\\

\noindent Now, in \cite{dimension} the notion of the analytic dimension dim $\mathcal{X}$ of a rigid variety $\mathcal{X}$ is defined, and it is proved to be equal to the supremum of the Krull dimensions of every affinoid algebra $A$ such that Sp $A$ is an affinoid subdomain of $\mathcal{X}$. In particular, if Sp $A$ is an affinoid subdomain of Sp $B$ in the sense of Definition \ref{affinoid-subdomain}, then K.dim$(A)\leq$ K.dim$(B)$.

Therefore, since the Tate algebras $K\langle\partial_1,\cdots,\partial_s\rangle$ and $K\langle\pi^{m}\partial_1,\cdots,\pi^{m}\partial_s\rangle$ both have dimension $s$, it remains to prove that the embeddings Sp $R\to$ Sp $K\langle\pi^{m}\partial_1,\cdots,\pi^{m}\partial_s\rangle$ and Sp $K\langle\partial_1,\cdots,\partial_s\rangle\to$ Sp $R$ define affinoid subdomains.\\

\noindent For convenience, set $D:=$ Sp $K\langle\partial_1,\cdots,\partial_s\rangle$ and $D_1:=$ Sp $K\langle\pi^{m}\partial_1,\cdots,\pi^{m}\partial_s\rangle$. Then $D$ can be realised as the unit disc in $s$-dimensional rigid $K$-space, while $D_1$ is a larger disc containing $D$, so explicitly $D=\{(x_1,\cdots,x_s)\in D_1:\vert x_i\vert\leq 1$ for all $i\}$. 

But since Sp $R$ contains $D$, we could instead write $D=\{(x_1,\cdots,x_s)\in$ Sp $R:\vert x_i\vert\leq 1$ for all $i\}$, and this is a Weierstrass subdomain of Sp $R$ in the sense of \cite[Definition 3.3.7]{Bosch}, and hence $D$ is an affinoid subdomain of Sp $R$ by \cite[Proposition 3.3.11]{Bosch}. Therefore K.dim$(R)\geq$ dim$D=s$.\\

\noindent Now, choose a basis $\{v_1,\cdots,v_r\}$ for $\mathcal{A}$, so that $\widehat{U(\mathcal{A})}_K\cong K\langle v_1,\cdots,v_r\rangle$. Using Proposition \ref{perp} we see that for each $i$, $\rho(v_i)=f_i(\partial_1,\cdots,\partial_s)$ for some polynomial $f_i$. So, let $T:=K\langle\pi^{m_1}\partial_1,\cdots,\pi^{m_s}\partial_s\rangle$, then since $\rho(\widehat{U(\mathcal{A})}_K)$ contains $T$, it follows that the affinoid algebra $B:=T\langle\zeta_1,\cdots,\zeta_r\rangle/(\zeta_i-f_i(\partial_1,\cdots,\partial_s))$ surjects onto $R=\rho(\widehat{U(\mathcal{A})}_K)$, where each $a\in T$ is sent to $a$, and each $\zeta_i$ is sent to $\rho(v_i)$. Therefore, K.dim$(R)\leq$ K.dim$(B)$.\\

\noindent But clearly there is a map $T\to B$, inducing a map of affinoid varieties Sp $B\to$ Sp $T$, and the proof of \cite[Proposition 3.3.11]{Bosch} shows that this corresponds to the embedding of the Weierstrass subdomain $\{x\in$ Sp $T:\vert f_i(x)\vert\leq 1$ for all $i\}$, into Sp $T$, and hence Sp $B$ is an affinoid subdomain of Sp $T=D_1$ by \cite[Proposition 3.3.11]{Bosch}, and hence K.dim$(B)\leq$ dim$(D_1)=s$.

Therefore $s\leq$ K.dim$(R)\leq$ K.dim$(B)\leq s$, forcing equality, so K.dim$(R)=s$ as required.\end{proof}

\begin{corollary}\label{reduction}

Suppose that $\mathcal{L}$ is a powerful Lie lattice in $\mathfrak{g}$, $\lambda,\mu:\mathfrak{g}\to K$ are linear forms with $\lambda(\mathcal{L}),\mu(\mathcal{L})\subseteq\mathcal{O}$, and suppose that $I(\lambda)\subseteq I(\mu)$. Then for any abelian ideal $\mathfrak{a}$ of $\mathfrak{g}$, if $\mathcal{A}:=\mathfrak{a}\cap\mathcal{L}$ then $I(\lambda)\cap\widehat{U(\mathcal{A})}_K=I(\mu)\cap\widehat{U(\mathcal{A})}_K$.

\end{corollary}

\begin{proof}

Let $\mathfrak{a}_1^{\perp}=\{u\in\mathfrak{g}:\lambda([u,\mathfrak{a}])=0\}$, and let $\mathfrak{a}_2^{\perp}=\{u\in\mathfrak{g}:\mu([u,\mathfrak{a}])=0\}$. Then if $s_i$ is the codimension of $\mathfrak{a}_i^{\perp}$ in $\mathfrak{g}$, for $i=1,2$, then it follows from Proposition \ref{dimension-theory} that $\widehat{U(\mathcal{A})}_K/I(\lambda)\cap\widehat{U(\mathcal{A})}_K$ has Krull dimension $s_1$ and $\widehat{U(\mathcal{A})}_K/I(\mu)\cap\widehat{U(\mathcal{A})}_K$ has Krull dimension $s_2$.

Since $I(\lambda)\subseteq I(\mu)$, clearly there is a surjection $\widehat{U(\mathcal{A})}_K/I(\lambda)\cap\widehat{U(\mathcal{A})}_K\to \widehat{U(\mathcal{A})}_K/I(\mu)\cap\widehat{U(\mathcal{A})}_K$, and by Corollary \ref{completely-prime} both these algebras are commutative domains. Therefore $s_2\leq s_1$, with equality if and only if $I(\lambda)\cap\widehat{U(\mathcal{A})}_K=I(\mu)\cap\widehat{U(\mathcal{A})}_K$, so we will prove that $s_1=s_2$.\\

\noindent Since $I(\lambda)\subseteq I(\mu)$, it follows from Lemma \ref{orbit} that there exists $g\in\mathbb{G}(K)$ such that $\mu=g\cdot\lambda$. We will prove that the image of $\mathfrak{a}_1^{\perp}$ under $g$ is $\mathfrak{a}_2^{\perp}$, and hence $\mathfrak{a}_1^{\perp}$ and $\mathfrak{a}_2^{\perp}$ must have the same dimension, and hence the same codimension in $\mathfrak{g}$ as required.\\

\noindent Firstly, if $u\in\mathfrak{a}_1^{\perp}$, then $\lambda([u,\mathfrak{g}])=0$, so given $v\in\mathfrak{g}$, $\mu([g(u),v])=(g\cdot\lambda)([g(u),g(g^{-1}(v))])=\lambda(g^{-1}g[u,g^{-1}(v)])=\lambda([u,g^{-1}(v)])=0$, and hence $\mu([g(u),\mathfrak{g}])=0$ and $g(u)\in\mathfrak{a}_2^{\perp}$. Thus $g(\mathfrak{a}_1^{\perp})\subseteq\mathfrak{a}_2^{\perp}$

Conversely, since $\lambda=g^{-1}\cdot\mu$, the same argument shows that $g^{-1}(\mathfrak{a}_2^{\perp})\subseteq\mathfrak{a}_1^{\perp}$, and hence $g(\mathfrak{a}_1^{\perp})=\mathfrak{a}_2^{\perp}$ as required.\end{proof}

\subsection{Control Theorem}

A particularly useful technique, which we hope should ultimately help us complete the proof of the deformed Dixmier-Moeglin equivalence in full generality, is to look for a suitable generating set for Dixmier annihilators.

Specifically, recall that an ideal $I$ of $\widehat{U(\mathcal{L})}_K$ is \emph{controlled} by a closed ideal $\mathcal{A}$ of $\mathcal{L}$ if $I=\widehat{U(\mathcal{L})}_K(I\cap\widehat{U(\mathcal{A})}_K)$, i.e. $I$ is generated as a left ideal by a subset of $\widehat{U(\mathcal{A})}_K$. If $\mathcal{A}=\mathfrak{a}\cap\mathcal{L}$ for some ideal $\mathfrak{a}$ of $\mathfrak{g}$, we may also say that $I$ is controlled by $\mathfrak{a}$ for convenience.

\begin{lemma}\label{control-properties}

Let $I$ be a two-sided ideal of $\widehat{U(\mathcal{L})}_K$, and let $\mathfrak{a}$ be an ideal of $\mathfrak{g}$:

\begin{itemize}

\item If $I$ is controlled by $\mathfrak{a}$, then $I$ is controlled by any ideal $\mathfrak{a}'$ of $\mathfrak{g}$ with $\mathfrak{a}\subseteq\mathfrak{a}'$.

\item If $I=\underset{i\in X}{\cap}{J_i}$ for some two-sided ideals $J_i$ of $\widehat{U(\mathcal{L})}_K$, and each $J_i$ is controlled by $\mathfrak{a}$, then $I$ is controlled by $\mathfrak{a}$.

\item If $F/K$ is a finite extension, and $J$ is a two-sided ideal of $\widehat{U(\mathcal{L})}_F$ such that $J\cap\widehat{U(\mathcal{L})}_K=I$, then if $J$ is controlled by $\mathfrak{a}\otimes_KF$, $I$ is controlled by $\mathfrak{a}$.

\end{itemize}

\end{lemma}

\begin{proof}

For the first statement, it is clear that if $I$ is controlled by $\mathcal{A}=\mathfrak{a}\cap\mathcal{L}\subseteq\mathfrak{a}'\cap\mathcal{L}=\mathcal{A}'$, then $I=\widehat{U(\mathcal{L})}_K(I\cap\widehat{U(\mathcal{A})}_K)\subseteq \widehat{U(\mathcal{L})}_K(I\cap\widehat{U(\mathcal{A}')}_K)\subseteq I$, forcing equality.\\

\noindent Now, let $\mathcal{A}=\mathfrak{a}\cap\mathcal{L}$, and let $\{u_1,\cdots,u_d\}$ be a basis for $\mathcal{L}$ such that $\{u_{r+1},\cdots,u_d\}$ is a basis for $\mathcal{A}$ for some $r\geq 0$. Using Lemma \ref{PBW}, we see that every element $s\in\widehat{U(\mathcal{L})}_K$ can be written as $s=\underset{\alpha\in\mathbb{N}^d}{\sum}{\lambda_{\alpha}u_1^{\alpha_1}\cdots u_d^{\alpha_d}}$ for some unique $\lambda_{\alpha}\in K$ with $\lambda_{\alpha}\rightarrow 0$ and $\alpha\rightarrow\infty$. For each $\beta\in\mathbb{N}^{r}$, set $s_{\beta}=\underset{\gamma\in\mathbb{N}^{d-r}}{\sum}{\lambda_{(\beta,\gamma)}u_{r+1}^{\gamma_{r+1}}\cdots u_d^{\gamma_d}}\in\widehat{U(\mathcal{A})}_K$, and it follows that $s=\underset{\beta\in\mathbb{N}^{r}}{\sum}{u_1^{\beta_1}\cdots u_r^{\beta_r}s_{\beta}}$.

We will first prove that $I$ is controlled by $\mathcal{A}$ if and only if for every $s\in I$, $s_{\beta}\in I$ for all $\beta\in\mathbb{N}^r$. One direction is clear since if $I=\widehat{U(\mathcal{L})}_K(I\cap\widehat{U(\mathcal{A})}_K)$ then every element of $I$ is a sum of elements of the form $su$ with $u\in I\cap\widehat{U(\mathcal{A})}_K$, so $su=\underset{\beta\in\mathbb{N}^{r}}{\sum}{u_1^{\beta_1}\cdots u_r^{\beta_r}s_{\beta}u}$ and $s_{\beta}u\in I$ as required.\\

\noindent Conversely, if $s_{\beta}\in I$ for all $s\in I$ then $u_1^{\beta_1}\cdots u_r^{\beta_r}s_{\beta}\in\widehat{U(\mathcal{L})}_K(I\cap\widehat{U(\mathcal{A})}_K)$ for each $\beta$, so since $\widehat{U(\mathcal{L})}_K(I\cap\widehat{U(\mathcal{A})}_K)$ is closed in $\widehat{U(\mathcal{L})}_K$, it follows that $s\in\widehat{U(\mathcal{L})}_K(I\cap\widehat{U(\mathcal{A})}_K)$ and hence $I$ is controlled by $\mathcal{A}$.\\

\noindent Now, if $I=\underset{i\in X}{\cap}{J_i}$ for some two-sided ideals $J_i$ of $\widehat{U(\mathcal{L})}_K$, then given $s=\underset{\beta\in\mathbb{N}^{r}}{\sum}{u_1^{\beta_1}\cdots u_r^{\beta_r}s_{\beta}}\in I$, we know that $s\in J_i=\widehat{U(\mathcal{L})}_K(J_i\cap\widehat{U(\mathcal{A})}_K)$ for every $i$, and hence $s_{\beta}\in J_i$ for every $i$, $\beta$, i.e. $s_{\beta}\in\underset{i\in X}{\cap}{J_i}=I$ for all $\beta$, and hence $I$ is controlled by $\mathcal{A}$.\\

\noindent For the final statement, it is clear that $\{u_{r+1},\cdots,u_d\}$ is an $F$-basis for $\mathcal{A}\otimes_{\mathcal{O}}\mathcal{O}_F$. So if $J$ is controlled by $\mathcal{A}\otimes_{\mathcal{O}}\mathcal{O}_F$ then given $s\in I\subseteq J$, $s_{\beta}\in J$ for every $\beta$. Therefore $s_{\beta}\in J\cap\widehat{U(\mathcal{L})}_K=I$ and hence $I$ is controlled by $\mathcal{A}$ as required.\end{proof}

\noindent Now, given a linear form $\lambda:\mathfrak{g}\to K$, recall that we define $\mathfrak{g}^{\lambda}:=\{u\in\mathfrak{g}:\lambda([u,\mathfrak{g}])=0\}$.

\begin{lemma}\label{twist}

Let $\sigma$ be a Lie automorphism of $\mathfrak{g}$, and let $\sigma\cdot\lambda$ be the linear form defined by $(\sigma\cdot\lambda)(u)=\lambda(\sigma^{-1}(u))$. Then $\mathfrak{g}^{\sigma\cdot\lambda}=\sigma(\mathfrak{g}^{\lambda})$.

\end{lemma}

\begin{proof}

Firstly, $u\in\mathfrak{g}^{\sigma\cdot\lambda}$ if and only if $\sigma\cdot\lambda([u,\mathfrak{g}])=\lambda(\sigma^{-1}([u,\mathfrak{g}]))=0$. But since $\sigma$ is a Lie automorphism, this is true if and only if $\lambda([\sigma^{-1}(u),\mathfrak{g}])=0$, i.e. $\sigma^{-1}(u)\in\mathfrak{g}^{\lambda}$ and $u\in\sigma(\mathfrak{g}^{\lambda})$ as required.\end{proof}

\begin{proposition}\label{control1}

Suppose that the subalgebra $\mathfrak{g}^{\lambda}$ is an ideal in $\mathfrak{g}$. Then $P=I(\lambda)$ is controlled by $\mathfrak{g}^{\lambda}$.

\end{proposition}

\begin{proof}

We proceed by induction on $n:=$ dim$_K\mathfrak{g}$. If $n=1$ then $\mathfrak{g}^{\lambda}=\mathfrak{g}$ and the statement is obvious, so the base case holds.\\

\noindent For the inductive step, if there exists a non-zero ideal $\mathfrak{t}$ of $\mathfrak{g}$ such that $\lambda(\mathfrak{t})=0$ then clearly $\mathfrak{t}\subseteq\mathfrak{g}^{\lambda}$. Let $\mathfrak{g}_0:=\mathfrak{g}/\mathfrak{t}$, let $\lambda_0$ be the linear form of $\mathfrak{g}_0$ induced by $\lambda$, and let $\mathcal{L}_0:=\mathcal{L}/\mathfrak{t}\cap\mathcal{L}$. Then $\mathfrak{g}_0^{\lambda_0}=\mathfrak{g}^{\lambda}/\mathfrak{t}$ is an ideal in $\mathfrak{g}_0$, so using induction, the Dixmier annihilator $I(\lambda_0)$ of $\widehat{U(\mathcal{L}_0)}_K$ is controlled by $\mathfrak{g}_0^{\lambda_0}=\mathfrak{g}^{\lambda}/\mathfrak{t}$.

But using Lemma \ref{induced} we see that $\widehat{D(\lambda_0)}\cong\widehat{D(\lambda)}$ as $\widehat{U(\mathcal{L})}_K$-modules, and thus $I(\lambda_0)=I(\lambda)/\mathfrak{t}\widehat{U(\mathcal{L})}_K$ and it follows immediately that $I(\lambda)$ is controlled by $\mathfrak{g}^{\lambda}$.\\

\noindent Therefore, we may assume that $\lambda(\mathfrak{t})\neq 0$ for all non-zero ideals $\mathfrak{t}$ of $\mathfrak{g}$. Note that since $\mathfrak{g}^{\lambda}$ is an ideal in $\mathfrak{g}$, it follows that $\mathfrak{t}=\mathfrak{g}^{\lambda}\cap\ker(\lambda)$ is an ideal in $\mathfrak{g}$, and clearly $\lambda(\mathfrak{t})=0$, so $\mathfrak{t}=0$. Thus $\lambda$ is injective when restricted to $\mathfrak{g}^{\lambda}$, and hence $\mathfrak{g}^{\lambda}$ is one-dimensional over $K$. 

Write $\mathfrak{g}^{\lambda}=Kz$ for some $z\in Z(\mathfrak{g})$ with $\lambda(z)\neq 0$, and since $Z(\mathfrak{g})\subseteq\mathfrak{g}^{\lambda}$, it follows that $Z(\mathfrak{g})=Kz$. Naturally we may assume that $z\in\mathcal{L}$.\\

\noindent Using Proposition \ref{sub-polarisation}, $\mathfrak{g}$ has a reducing quadruple $(x,y,z,\mathfrak{g}')$, with $y,z\in\mathcal{L}$, and we can choose a polarisation $\mathfrak{b}$ of $\mathfrak{g}$ at $\lambda$ with $\mathfrak{b}\subseteq\mathfrak{g}'$. Since $\mathfrak{b}$ is also a polarisation of $\mathfrak{g}'$ at $\mu:=\lambda|_{\mathfrak{g}'}$, it follows from Lemma \ref{polarisation-properties} that dim$_K\mathfrak{b}=\frac{1}{2}($dim$_K\mathfrak{g}+$ dim$_K\mathfrak{g}^{\lambda})=\frac{1}{2}($dim$_K\mathfrak{g}'+$ dim$_K\mathfrak{g}'^{\mu})$. But dim$_K\mathfrak{g}'=$ dim$_K\mathfrak{g}-1$, so this means that dim$_K\mathfrak{g}'^{\mu}=$ dim$_K\mathfrak{g}^{\lambda}+1=2$.

But clearly $y,z\in\mathfrak{g}'^{\mu}$, and hence $\mathfrak{g}'^{\mu}=Kz\oplus Ky$, and clearly this is an ideal in $\mathfrak{g}$, and therefore in $\mathfrak{g}'$.\\

\noindent Now, using Theorem \ref{control}, we see that $I(\lambda)$ is controlled by $\mathcal{L}'=\mathfrak{g}'\cap\mathcal{L}$, so let $Q:=I(\lambda)\cap\widehat{U(\mathcal{L}')}_K$ -- a semiprime ideal of $\widehat{U(\mathcal{L}')}_K$ such that $I(\lambda)=Q\widehat{U(\mathcal{L})}_K=\widehat{U(\mathcal{L})}_KQ$. We will prove that $Q$ is controlled by $Kz$, and it will follow that $I(\lambda)$ is controlled by $Kz$.\\

\noindent Using Corollary \ref{Dix-intersection}, we know that $Q=\underset{i\in I}{\cap}{\text{Ann}_{\widehat{U(\mathcal{L}')}_K}\widehat{D(\mu_i)}_{F_i}}$ for some finite extensions $F_i/K$, linear forms $\mu_i:\mathfrak{g}'\to F$ such that $\mu_i(\mathcal{L}')\subseteq\mathcal{O}_F$. Thus $I(\lambda)=\widehat{U(\mathcal{L})}_KQ\subseteq \widehat{U(\mathcal{L})}_K\text{Ann}_{\widehat{U(\mathcal{L}')}_K}\widehat{D(\mu_i)}_{F_i}$ for each $i$. 

Setting $\lambda_i$ as any extension of $\mu_i$ to $\mathcal{L}$, it follows that $\widehat{D(\lambda_i)}_{F_i}=\widehat{U(\mathcal{L})}_{F_i}\otimes_{\widehat{U(\mathcal{L}')}_{F_i}}\widehat{D(\mu_i)}_{F_i}$. Since $I(\lambda)\otimes_K F_i$ is a two-sided ideal of $\widehat{U(\mathcal{L})}_{F_i}$ contained in $\widehat{U(\mathcal{L})}_{F_i}\text{Ann}_{\widehat{U(\mathcal{L}')}_{F_i}}\widehat{D(\mu_i)}_{F_i}$, it follows from Lemma \ref{ind-annihilator} that $I(\lambda)\subseteq$ Ann$_{\widehat{U(\mathcal{L})}_K}\widehat{D(\lambda_i)}_{F_i}=I(\lambda_i)$.\\ 

\noindent Using Lemma \ref{orbit}, this means that $\lambda$ and $\lambda_i$ lie in the same $\mathbb{G}$-coadjoint orbit for all $i$, i.e. $\lambda_i=g_i\cdot\lambda$ for some $g_i\in\mathbb{G}(F)$\\

\noindent So, fixing $i$, let $F=F_i$ for convenience, and let $\mu=\lambda|_{\mathfrak{g}'}$, then $\mu$ extends to an $F$-linear form on $\mathfrak{g}'_{F}=\mathfrak{g}'\otimes_K F$, and $\mathfrak{g}_F'^{\mu}=(\mathfrak{g}'^{\mu})\otimes_KF$. Let $\sigma$ be the Lie automorphism of $\mathfrak{g}'_F$ given by the restriction of $g_i$ to $\mathfrak{g}'_F$, then $\mu_i=\sigma\cdot\mu$, so it follows from Lemma \ref{twist} that $\mathfrak{g}_F'^{\mu_i}=\sigma(\mathfrak{g}_F'^{\mu})=\mathfrak{g}'^{\mu}\otimes F$ since $\mathfrak{g}'^{\mu}$ is an ideal in $\mathfrak{g}$, and hence is preserved by $\sigma$. 

It follows from induction that $\text{Ann}_{\widehat{U(\mathcal{L}')}_{F_i}}\widehat{D(\mu_i)}$ is controlled by $\mathfrak{g}'^{\mu}\otimes_K F_i$, and hence $\text{Ann}_{\widehat{U(\mathcal{L}')}_{K}}\widehat{D(\mu_i)}$ is controlled by $\mathfrak{g}'^{\mu}$ by Lemma \ref{control-properties}. Since this is true for all $i$, it also follows from Lemma \ref{control-properties} that $Q=\underset{i\in I}{\cap}{\text{Ann}_{\widehat{U(\mathcal{L}')}_K}\widehat{D(\mu_i)}_{F_i}}$ is controlled by $\mathfrak{g}'^{\mu}=Kz\oplus Ky$.\\

\noindent Finally, consider the subalgebra $\mathfrak{h}=$ Span$_K\{x,y,z\}$ of $\mathfrak{g}$, let $\mathcal{H}:=\mathfrak{h}\cap\mathcal{L}$, $\mathcal{A}:=\mathfrak{g}'^{\mu}\cap\mathcal{L}=\mathcal{O}z\oplus\mathcal{O}y$, and let $I:=(Q\cap\widehat{U(\mathcal{A})}_K)\widehat{U(\mathcal{H})}_K$. Then $I$ is a proper two-sided ideal of $\widehat{U(\mathcal{H})}_K$ containing $z-\lambda(z)$, thus $I=(z-\lambda(z))\widehat{U(\mathcal{H})}_K$ by Lemma \ref{Heisenberg}. 

Therefore, $Q\cap\widehat{U(\mathcal{A})}_K=I\cap\widehat{U(\mathcal{A})}_K=(z-\lambda(z))\widehat{U(\mathcal{A})}_K$, and hence $Q=(Q\cap\widehat{U(\mathcal{A}')}_K)\widehat{U(\mathcal{L}')}_K=(z-\lambda(z))\widehat{U(\mathcal{L}')}_K$, i.e. $Q$ is controlled by $\mathcal{O}z$\\

\noindent So since $\mathfrak{g}^{\lambda}=Kz$, it follows that $Q$ is controlled by $\mathfrak{g}^{\lambda}$, and thus $I(\lambda)$ is controlled by $\mathfrak{g}^{\lambda}$ as required.\end{proof}

\begin{proposition}\label{control2}

Suppose that $\mathfrak{a}$ is any ideal of $\mathfrak{g}$ containing $\mathfrak{g}^{\lambda}$. Then $I(\lambda)$ is controlled by $\mathfrak{a}$.

\end{proposition}

\begin{proof}

This proof is very similar to the proof of Proposition \ref{control1}, but uses it as a key step. Again, we proceed by induction on $n:=$ dim$_K\mathfrak{g}$. If $n=1$ then $\mathfrak{g}^{\lambda}=\mathfrak{g}=\mathfrak{a}$ and the statement is obvious, so the base case holds.\\

\noindent For the inductive step, if there exists a non-zero ideal $\mathfrak{t}$ of $\mathfrak{g}$ such that $\lambda(\mathfrak{t})=0$ then clearly $\mathfrak{t}\subseteq\mathfrak{g}^{\lambda}\subseteq\mathfrak{a}$. Let $\mathfrak{g}_0:=\mathfrak{g}/\mathfrak{t}$, let $\mathfrak{a}_0:=\mathfrak{a}/\mathfrak{t}$, let $\lambda_0$ be the linear form of $\mathfrak{g}_0$ induced by $\lambda$, and let $\mathcal{L}_0=\mathcal{L}/\mathfrak{t}\cap\mathcal{L}$. Then $\mathfrak{g}_0^{\lambda_0}=\mathfrak{g}^{\lambda}/\mathfrak{t}$, and $\mathfrak{a}_0$ is an ideal in $\mathfrak{g}_0$ containing $\mathfrak{g}_0^{\mu}$. So, using induction, the Dixmier annihilator $I(\lambda_0)$ of $\widehat{U(\mathcal{L}_0)}_K$ is controlled by $\mathfrak{a}_0$.

But using Lemma \ref{induced}, $\widehat{D(\lambda)}\cong\widehat{D(\lambda)}$ as $\widehat{U(\mathcal{L})}_K$-modules, and thus $I(\lambda_0)=I(\lambda)/\mathfrak{t}\widehat{U(\mathcal{L})}_K$ and it follows immediately that $I(\lambda)$ is controlled by $\mathfrak{a}$.\\

\noindent Therefore, we may assume that $\lambda(\mathfrak{t})\neq 0$ for all non-zero ideals $\mathfrak{t}$ of $\mathfrak{g}$, and it follows that $Z(\mathfrak{g})=Kz$ for some $z\in Z(\mathcal{L})$ such that $\lambda(z)\neq 0$. Furthermore, we may assume that $\mathfrak{a}$ is the ideal of $\mathfrak{g}$ generated by $\mathfrak{g}^{\lambda}$, i.e the subspace $\mathfrak{g}^{\lambda}+[\mathfrak{g},\mathfrak{g}^{\lambda}]+[\mathfrak{g},[\mathfrak{g},\mathfrak{g}^{\lambda}]]+\cdots$, since if $I(\lambda)$ is controlled by this ideal, then it will be controlled by any larger ideal by Lemma \ref{control-properties}.\\

\noindent If $\mathfrak{g}^{\lambda}=\mathfrak{a}$ then $\mathfrak{g}^{\lambda}$ is an ideal of $\mathfrak{g}$, and the result follows from Proposition \ref{control1}. So we may assume that $\mathfrak{g}^{\lambda}\neq\mathfrak{a}$, and hence $\mathfrak{a}$ is not central in $\mathfrak{g}$. Therefore, since $\mathfrak{g}$ is nilpotent, there exists $y\in\mathfrak{a}\cap\mathcal{L}$ with $y\notin Z(\mathfrak{g})$ such that $[y,\mathfrak{g}]\subseteq Z(\mathfrak{g})=Kz$. So setting $\mathfrak{g}':=\ker(\ad(y))$, it follows that $\mathfrak{g}=\mathfrak{g}'\oplus Kx$ and $(x,y,z,\mathfrak{g}')$ is a reducing quadruple.

Moreover, $\lambda([y,\mathfrak{g}^{\lambda}])=0$ and since $[y,\mathfrak{g}]\subseteq Kz$ and $\lambda(z)\neq 0$, this means that $[y,\mathfrak{g}^{\lambda}]=0$. Also, $[y,[\mathfrak{g},\mathfrak{g}]]=0$, so $[y,[\mathfrak{g},[\mathfrak{g}[\cdots,[\mathfrak{g},\mathfrak{g}^{\lambda}]\cdots]]=0$, and hence $[y,\mathfrak{a}]=0$ by the construction of $\mathfrak{a}$, and thus $\mathfrak{a}\subseteq\mathfrak{g}'$.\\

\noindent Now, using Theorem \ref{control}, we see that $I(\lambda)$ is controlled by $\mathcal{L}'=\mathfrak{g}'\cap\mathcal{L}$, so let $Q:=I(\lambda)\cap\widehat{U(\mathcal{L}')}_K$ -- a semiprime ideal of $\widehat{U(\mathcal{L}')}_K$ such that $I(\lambda)=Q\widehat{U(\mathcal{L})}_K=\widehat{U(\mathcal{L})}_KQ$. We will prove that $Q$ is controlled by $\mathfrak{a}$, and it will follow that $I(\lambda)$ is controlled by $\mathfrak{a}$.\\

\noindent Using Proposition \ref{sub-polarisation}, we can choose a polarisation $\mathfrak{b}$ of $\mathfrak{g}$ at $\lambda$ with $\mathfrak{b}\subseteq\mathfrak{g}'$. Since $\mathfrak{b}$ is also a polarisation of $\mathfrak{g}'$ at $\mu=\lambda|_{\mathfrak{g}'}$, it follows from Lemma \ref{polarisation-properties} that dim$_K\mathfrak{b}=\frac{1}{2}($dim$_K\mathfrak{g}+$ dim$_K\mathfrak{g}^{\lambda})=\frac{1}{2}($dim$_K\mathfrak{g}'+$ dim$_K\mathfrak{g}'^{\mu})$. But dim$_K\mathfrak{g}'=$ dim$_K\mathfrak{g}-1$, so this means that dim$_K\mathfrak{g}'^{\mu}=$ dim$_K\mathfrak{g}^{\lambda}+1$. 

So since $\mathfrak{g}'^{\mu}$ contains $\mathfrak{g}^{\lambda}$ and $y$, and since $\lambda([x,y])\neq 0$, $y\notin\mathfrak{g}^{\lambda}$. Therefore $\mathfrak{g}^{\lambda}\oplus Ky\subseteq\mathfrak{g}'^{\mu}$ and dim$_K\mathfrak{g}^{\lambda}\oplus Ky=$ dim$_K\mathfrak{g}'^{\mu}$, so $\mathfrak{g}'^{\mu}=\mathfrak{g}^{\lambda}\oplus Ky\subseteq\mathfrak{a}$.\\

\noindent Using Corollary \ref{Dix-intersection}, we know that $Q=\underset{i\in I}{\cap}{\text{Ann}_{\widehat{U(\mathcal{L}')}_K}\widehat{D(\mu_i)}_{F_i}}$ for some finite extensions $F_i/K$, linear forms $\mu_i:\mathfrak{g}'\to F$ such that $\mu_i(\mathcal{L}')\subseteq\mathcal{O}_F$. Thus $I(\lambda)=\widehat{U(\mathcal{L})}_KQ\subseteq \widehat{U(\mathcal{L})}_K\text{Ann}_{\widehat{U(\mathcal{L}')}_K}\widehat{D(\mu_i)}_{F_i}$ for each $i$. 

Setting $\lambda_i$ as any extension of $\mu_i$ to $\mathcal{L}$, it follows that $\widehat{D(\lambda_i)}_{F_i}=\widehat{U(\mathcal{L})}_{F_i}\otimes_{\widehat{U(\mathcal{L}')}_{F_i}}\widehat{D(\mu_i)}_{F_i}$. Since $I(\lambda)\otimes_K F_i$ is a two-sided ideal of $\widehat{U(\mathcal{L})}_{F_i}$ contained in $\widehat{U(\mathcal{L})}_{F_i}\text{Ann}_{\widehat{U(\mathcal{L}')}_{F_i}}\widehat{D(\mu_i)}_{F_i}$, it follows from Lemma \ref{ind-annihilator} that $I(\lambda)\subseteq$ Ann$_{\widehat{U(\mathcal{L})}_K}\widehat{D(\lambda_i)}_{F_i}=I(\lambda_i)$.\\ 

\noindent Using Lemma \ref{orbit}, this means that $\lambda$ and $\lambda_i$ lie in the same $\mathbb{G}$-coadjoint orbit for all $i$, i.e. $\lambda_i=g_i\cdot\lambda$ for some $g_i\in\mathbb{G}(F)$\\

\noindent So, fixing $i$, let $F=F_i$ for convenience, and let $\mu=\lambda|_{\mathfrak{g}'}$, then $\mu$ extends to an $F$-linear form on $\mathfrak{g}'_{F}=\mathfrak{g}'\otimes_K F$, and $\mathfrak{g}_F'^{\mu}=(\mathfrak{g}'^{\mu})\otimes_KF$. Let $\sigma$ be the Lie automorphism of $\mathfrak{g}'_F$ given by the restriction of $g_i$ to $\mathfrak{g}'_F$, then $\mu_i=\sigma\cdot\mu$, so it follows from Lemma \ref{twist} that $\mathfrak{g}_F'^{\mu_i}=\sigma(\mathfrak{g}_F'^{\mu})\subseteq\sigma(\mathfrak{a}\otimes F)\subseteq{a}\otimes F$.

Therefore, using induction, Ann$_{\widehat{U(\mathcal{L})}_{F_i}}\widehat{D(\mu_i)}_{F_i}$ is controlled by $\mathfrak{a}\otimes F_i$, and hence Ann$_{\widehat{U(\mathcal{L})}_{K}}\widehat{D(\mu_i)}_{F_i}$ is controlled by $\mathfrak{a}$ using Lemma \ref{control-properties}. Since this holds for all $i$ and $Q=\underset{i\in I}{\cap}{\text{Ann}_{\widehat{U(\mathcal{L}')}_K}\widehat{D(\mu_i)}_{F_i}}$, it also follows from Lemma \ref{control-properties} that $Q$ is controlled by $\mathfrak{a}$, and hence $I(\lambda)$ is controlled by $\mathfrak{a}$ as required.\end{proof}

\subsection{Special Linear forms}

\begin{definition}\label{special-Dix-annihilator}

We call a linear form $\lambda:\mathfrak{g}\to K$ \emph{special} if the ideal $\mathfrak{a}$ of $\mathfrak{g}$ generated by $\mathfrak{g}^{\lambda}$ satisfies $\lambda([\mathfrak{a},\mathfrak{a}])=0$.

\end{definition}

\noindent This condition is not universal, but as we will see, it is very useful. Using Proposition \ref{control2}, we see that if $\lambda$ is special then $P=I(\lambda)$ is controlled by an ideal $\mathfrak{a}$ of $\mathfrak{g}$ satisfying $\lambda([\mathfrak{a},\mathfrak{a}])=0$. We will now study the Dixmier annihilators associated to special linear forms.

\begin{lemma}\label{faithful}

Let $F/K$ be a finite extension, and let $\lambda:\mathfrak{g}\to F$ be a linear form such that $\lambda(\mathcal{L})\subseteq\mathcal{O}_F$. Then there exists a unique ideal $\mathfrak{t}$ of $\mathfrak{g}$, maximal with respect to the property that $\lambda(\mathfrak{t})=0$, and in fact $\mathfrak{t}=I(\lambda)\cap\mathfrak{g}$.

\end{lemma}

\begin{proof}

Firstly, it is clear that if $\mathfrak{t}_1,\mathfrak{t_2}\trianglelefteq\mathfrak{g}$ with $\lambda(\mathfrak{t}_1)=\lambda(\mathfrak{t}_2)=0$, then $\lambda(\mathfrak{t}_1+\mathfrak{t}_2)=0$, so it follows that there exists a unique ideal $\mathfrak{t}$ which is maximal with respect to the property that $\lambda(\mathfrak{t})=0$.\\

\noindent Moreover, by Lemma \ref{polarisation-properties}, $\mathfrak{t}\otimes F\subseteq\mathfrak{b}$ for any polarisation $\mathfrak{b}$ of $\mathfrak{g}_F$ at $\lambda$, hence $\mathfrak{t}\widehat{D(\lambda)}_F=\mathfrak{t}\widehat{U(\mathcal{L})}_F\otimes_{\widehat{U(\mathcal{B})}_F}F_{\lambda}=\widehat{U(\mathcal{L})}_F\mathfrak{t}\otimes_{\widehat{U(\mathcal{B})}_F}F_{\lambda}=\widehat{U(\mathcal{L})}_F\otimes_{\widehat{U(\mathcal{B})}_F}\mathfrak{t}F_{\lambda}=0$, so $\mathfrak{t}\subseteq I(\lambda)\cap\mathfrak{g}$.

Conversely, let $\mathfrak{t}':=I(\lambda)\cap\mathfrak{g}$, then $\mathfrak{t}'\widehat{D(\lambda)}_F=0$, so $\mathfrak{t}'\otimes F_{\lambda}=0$, which implies that $\mathfrak{t}'\subseteq\mathfrak{b}$ and $\lambda(\mathfrak{t}')=0$ as required.\end{proof}

\begin{theorem}\label{Dix-maximal}

Assume that $\mathcal{L}$ is a powerful Lie lattice, and $\lambda:\mathfrak{g}\to K$ is a special linear form such that $\lambda(\mathcal{L})\subseteq\mathcal{O}$. Then the Dixmier annihilator $I(\lambda)$ is a maximal ideal of $\widehat{U(\mathcal{L})}_K$.

\end{theorem}

\begin{proof}

Let $\mathfrak{t}=I(\lambda)\cap\mathfrak{g}$, and let $\mathcal{L}_0=\mathcal{L}/\mathcal{L}\cap\mathfrak{t}$, then to prove that $I(\lambda)$ is maximal in $\widehat{U(\mathcal{L})}_K$, we need only prove that $I(\lambda)/\mathfrak{t}\widehat{U(\mathcal{L})}_K$ is maximal in $\widehat{U(\mathcal{L}_0)}_K=\widehat{U(\mathcal{L})}_K/\mathfrak{t}\widehat{U(\mathcal{L})}_K$.

But $\lambda(\mathfrak{t})=0$ by Lemma \ref{faithful}, so $\mathfrak{t}\subseteq\mathfrak{g}^{\lambda}$ and $\lambda$ induces a special linear form $\lambda_0$ of $\mathfrak{g}/\mathfrak{t}$, and it follows from Lemma \ref{induced} that $\widehat{D(\lambda)}\cong\widehat{D(\lambda_0)}$ as $\widehat{U(\mathcal{L})}_K$-modules, and thus Ann$_{\widehat{U(\mathcal{L}_0)}_K}\widehat{D(\lambda_0)}=I(\lambda)/\mathfrak{t}\widehat{U(\mathcal{L})}_K$.\\

\noindent Therefore, after quotienting out by $\mathfrak{t}$, we may assume that $I(\lambda)\cap\mathfrak{g}=0$, and hence $\lambda(\mathfrak{a})\neq 0$ for all non-zero ideals $\mathfrak{a}$ of $\mathfrak{g}$.\\

\noindent Now, since $\lambda$ is special, we may choose an ideal $\mathfrak{a}$ of $\mathfrak{g}$ with $\mathfrak{g}^{\lambda}\subseteq\mathfrak{a}$ and $\lambda([\mathfrak{a},\mathfrak{a}])=0$. But $[\mathfrak{a},\mathfrak{a}]$ is an ideal of $\mathfrak{g}$, and hence $[\mathfrak{a},\mathfrak{a}]=0$, i.e. $\mathfrak{a}$ is an abelian ideal of $\mathfrak{g}$. So suppose that $Q$ is a maximal ideal of $\widehat{U(\mathcal{L})}_K$ with $I(\lambda)\subseteq Q$, we will prove that $I(\lambda)=Q$, and hence $Q$ is maximal.\\

\noindent Since $Q$ is maximal, it is locally closed, so by Theorem \ref{aff-Dix2}, $Q=I(\mu)$ for some finite extension $F/K$ and some linear form $\mu:\mathfrak{g}\to F$ with $\mu(\mathcal{L})\subseteq\mathcal{O}_F$, and after extending $F$ if necessary, we may assume that $F/K$ is a Galois extension.\\

\noindent Using Proposition \ref{Dix-extension}, we see that there exists $\sigma\in Gal(F/K)$ such that Ann$_{\widehat{U(\mathcal{L})}_F}\widehat{D(\sigma\cdot\lambda)}_F\subseteq$ Ann$_{\widehat{U(\mathcal{L})}_F}\widehat{D(\mu)}$, and note that Ann$_{\widehat{U(\mathcal{L})}_F}\widehat{D(\sigma\cdot\lambda)}_F\cap\widehat{U(\mathcal{L})}_K=I(\lambda)$ by Lemma \ref{transport-structure}.

Also, note that $\mathfrak{g}_F^{\sigma\cdot\lambda}=\sigma(\mathfrak{g}_F^{\lambda})$ by Lemma \ref{twist}. But since $\lambda$ takes values in $K$, $\mathfrak{g}_F^{\lambda}=(\mathfrak{g}^{\lambda})\otimes_K F$, so it follows that $\mathfrak{g}_F^{\sigma\cdot\lambda}=\sigma((\mathfrak{g}^{\lambda})\otimes_K F)$ is contained in the abelian ideal $\sigma(\mathfrak{a}\otimes_K F)=\mathfrak{a}\otimes_K F$. Therefore $\sigma\cdot\lambda$ is a special linear form of $\mathfrak{g}\otimes_K F$.\\

\noindent So let us first suppose that $F=K$, and hence $\sigma=1$, then since $I(\lambda)\subseteq I(\mu)$, we know that $\mu=g\cdot\lambda$ for some $g\in\mathbb{G}(K)$, and hence $\mathfrak{g}^{\mu}=g(\mathfrak{g}^{\lambda})$ by Lemma \ref{twist}. Hence $\mathfrak{g}^{\mu}\subseteq g(\mathfrak{a})=\mathfrak{a}$ and hence $\mu$ is also a special linear form, and it follows from Proposition \ref{control2} that $I(\lambda)$ and $I(\mu)$ are both controlled by $\mathfrak{a}$, i.e. if $\mathcal{A}=\mathfrak{a}\cap\mathcal{L}$ then $I(\lambda)=\widehat{U(\mathcal{L})}_K(I(\lambda)\cap\widehat{U(\mathcal{A})}_K)$ and $I(\mu)=\widehat{U(\mathcal{L})}_K(I(\mu)\cap\widehat{U(\mathcal{A})}_K)$.

But since $\mathfrak{a}$ is abelian, we see using Corollary \ref{reduction} that $I(\lambda)\cap\widehat{U(\mathcal{A})}_K=I(\mu)\cap\widehat{U(\mathcal{A})}_K$, and hence $I(\mu)=\widehat{U(\mathcal{L})}_K(I(\mu)\cap\widehat{U(\mathcal{A})}_K)=\widehat{U(\mathcal{L})}_K(I(\lambda)\cap\widehat{U(\mathcal{A})}_K)=I(\lambda)$ as required.\\

\noindent More generally, since Ann$_{\widehat{U(\mathcal{L})}_F}\widehat{D(\sigma\cdot\lambda)}_F\subseteq$ Ann$_{\widehat{U(\mathcal{L})}_F}\widehat{D(\mu)}$ and $\sigma\cdot\lambda$ is special, it follows from this argument that Ann$_{\widehat{U(\mathcal{L})}_F}\widehat{D(\sigma\cdot\lambda)}_F=$ Ann$_{\widehat{U(\mathcal{L})}_F}\widehat{D(\mu)}$. Thus $I(\lambda)=$ Ann$_{\widehat{U(\mathcal{L})}_F}\widehat{D(\sigma\cdot\lambda)}_F\cap\widehat{U(\mathcal{L})}_K=$ Ann$_{\widehat{U(\mathcal{L})}_F}\widehat{D(\mu)}_F\cap\widehat{U(\mathcal{L})}_K=I(\mu)=Q$, and hence $I(\lambda)=Q$ is maximal as required.\end{proof}

\begin{corollary}\label{Dix-maximal2}

Let $F/K$ be a finite, Galois extension, $\mathcal{L}$ a powerful Lie lattice, $\lambda:\mathfrak{g}\to F$ a linear form such that $\lambda(\mathcal{L})\subseteq\mathcal{O}_F$ and the extension of $\lambda$ to $\mathfrak{g}\otimes_K F$ is special. Then $I(\lambda)$ is a maximal ideal of $\widehat{U(\mathcal{L})}_K$.

\end{corollary}

\begin{proof}

Using Theorem \ref{Dix-maximal}, we see that $P=$ Ann$_{\widehat{U(\mathcal{L})}_F}\widehat{D(\lambda)}$ is a maximal ideal of $\widehat{U(\mathcal{L})}_F$. So let us suppose that $Q$ is a maximal ideal of $\widehat{U(\mathcal{L})}_K$ with $I(\lambda)\subseteq Q$.\\

\noindent Choose a prime ideal $Q'$ of $\widehat{U(\mathcal{L})}_F$ which is minimal prime above $Q\otimes_K F$. Then since $I(\lambda)\otimes_K F\subseteq Q\otimes_K F\subseteq Q'$, it follows that $Q'$ contains a minimal prime above $I(\lambda)\otimes_K F$. But using Corollary \ref{min-prime}, this means that $\sigma(P)\subseteq Q'$ for some $\sigma\in Gal(F/K)$, and since $P$ is maximal, $\sigma(P)$ is maximal, which implies that $Q'=\sigma(P)$ is minimal prime above $I(\lambda)\otimes_K F$.

Therefore, $Q'\cap\widehat{U(\mathcal{L})}_K=\sigma(P)\cap\widehat{U(\mathcal{L})}_K=I(\lambda)$, but since $Q\subseteq Q'\cap\widehat{U(\mathcal{L})}_K$, it follows that $Q=I(\lambda)$ as required.\end{proof}

\subsection{Proof of Theorem \ref{B}}

We are now ready to prove Theorem \ref{B}, which now requires only the following result:

\begin{lemma}\label{metabelian}

Assume $\mathfrak{g}$ is nilpotent and \emph{metabelian} (i.e. $[\mathfrak{g},\mathfrak{g}]$ is abelian). Then any linear form $\lambda:\mathfrak{g}\to K$ is special.

\end{lemma}

\begin{proof}

We want to prove that for any linear form $\lambda:\mathfrak{g}\to K$, $\mathfrak{g}^{\lambda}$ is contained in an ideal $\mathfrak{a}$ of $\mathfrak{g}$ such that $\lambda([\mathfrak{a},\mathfrak{a}])=0$. Firstly, since $[\mathfrak{g},\mathfrak{g}]$ is an abelian ideal of $\mathfrak{g}$, using Lemma \ref{standard-polarisation} we see that there exists a polarisation $\mathfrak{b}$ of $\mathfrak{g}$ at $\lambda$ such that $[\mathfrak{g},\mathfrak{g}]\subseteq\mathfrak{b}$.

But since $\mathfrak{b}$ is a subalgebra of $\mathfrak{g}$ containing $[\mathfrak{g},\mathfrak{g}]$, it follows that $\mathfrak{b}$ is an ideal of $\mathfrak{g}$. Also, $\mathfrak{g}^{\lambda}\subseteq\mathfrak{b}$ by Lemma \ref{polarisation-properties}, and of course $\lambda([\mathfrak{b},\mathfrak{b}])=0$. Therefore $\lambda$ is special.\end{proof}

\vspace{0.1in}

\noindent\emph{Proof of Theorem \ref{B}.} As per Definition \ref{deformed-Dix-moeg}, we want to prove that for all prime ideals $P$ of $\widehat{U(\mathcal{L})}_K$, there exist $n$ sufficiently high such that $P\cap\widehat{U(\pi^n\mathcal{L})}_K$ is weakly rational if and only if it is primitive if and only if it is locally closed.\\

\noindent We already know using Proposition \ref{nullstellensatz} that if that if $P\cap\widehat{U(\pi^n\mathcal{L})}_K$ is locally closed, then it is primitive and weakly rational, and also note if $P\cap\widehat{U(\pi^N\mathcal{L})}_K$ is weakly rational for any $N$, then $P\cap\widehat{U(\pi^n\mathcal{L})}_K$ is weakly rational for all $n\geq N$. 

So we can assume that there exists such an $N\geq 0$, and we will prove that $P\cap\widehat{U(\pi^n\mathcal{L})}_K$ is locally closed, and in fact maximal, for sufficiently high $n\geq N$.\\

\noindent Since $P\cap\widehat{U(\pi^N\mathcal{L})}_K$ is weakly rational, it follows from Theorem \ref{A} that for all sufficiently high $n\geq N$, there exists a finite extension $F/K$ and a linear form $\lambda:\mathfrak{g}\to F$ such that $\lambda(\pi^n\mathcal{L})\subseteq\mathcal{O}_F$ and $P\cap\widehat{U(\pi^n\mathcal{L})}_K=$ Ann$_{\widehat{U(\pi^n\mathcal{L})}_K}\widehat{D(\lambda)}_F$. After extending $F$ if necessary, we may assume that $F/K$ is Galois. Since $\mathfrak{g}$ is metabelian, it is clear that $\mathfrak{g}\otimes_K F$ is also metabelian, so it follows from Lemma \ref{metabelian} that $\lambda$ is special.\\

\noindent Choose $n$ such that $\pi^n\mathcal{L}$ is a powerful Lie lattice, i.e. such that $\pi^n\in p\mathcal{O}$, which is true for all sufficiently high $n$. Then using Corollary \ref{Dix-maximal2} we see that Ann$_{\widehat{U(\pi^n\mathcal{L})}_K}\widehat{D(\lambda)}_F=P\cap\widehat{U(\pi^n\mathcal{L})}_K$ is a maximal ideal of $\widehat{U(\pi^n\mathcal{L})}_K$ as required.\\

\noindent Finally, it follows from Corollary \ref{completely-prime} that $\widehat{U(\pi^n\mathcal{L})}_K/$ Ann$_{\widehat{U(\pi^n\mathcal{L})}_K}\widehat{D(\lambda)}_F$ is a domain, and hence $\widehat{U(\pi^n\mathcal{L})}_K/P\cap\widehat{U(\pi^n\mathcal{L})}_K$ is a simple domain.\qed

% You need a bibliography file called mybib.bib
\bibliographystyle{abbrv}

\end{document}